\documentclass[a4paper,12pt,reqno]{amsart}

\usepackage[T1]{fontenc}
\usepackage[utf8]{inputenc}
\usepackage{lmodern}
\usepackage{amsmath,amssymb,mathtools}
\usepackage{microtype}
\usepackage{enumitem}
\usepackage{geometry}
\usepackage{hyperref}
\usepackage[nameinlink,capitalize]{cleveref}
\usepackage{amsfonts}
\usepackage{array}
\usepackage{url}
\usepackage{graphicx}
\usepackage{float}
\usepackage{multirow,bigdelim}
\usepackage{lipsum}% http://ctan.org/pkg/lipsum
\usepackage{chngcntr}% http://ctan.org/pkg/lipsum

\hypersetup{
  colorlinks=true,
  linkcolor=blue,
  citecolor=blue,
  urlcolor=blue,
  pdftitle={Construction of Finite Hilbert--P\'olya Matrices from Weil's Explicit Formula},
  pdfauthor={Yaoming Shi}
}
\numberwithin{equation}{section}
\newtheorem{theorem}{Theorem}[section]
\newtheorem{proposition}[theorem]{Proposition}
\newtheorem{lemma}[theorem]{Lemma}
\newtheorem{corollary}[theorem]{Corollary}
\theoremstyle{definition}
\newtheorem{definition}[theorem]{Definition}
\theoremstyle{remark}
\newtheorem{remark}[theorem]{Remark}

\newcommand{\RePart}{\operatorname{Re}}
\newcommand{\ImPart}{\operatorname{Im}}
\newcommand{\PF}{\operatorname{PF}}

\newcommand{\dd}{\,\mathrm{d}}
\newcommand{\norm}[1]{\left\lVert #1\right\rVert}
\newcommand{\abs}[1]{\left\lvert #1\right\rvert}

\title[Finite Hilbert--P\'olya matrices from Weil's explicit formula]{Construction of Finite Hilbert--P\'olya Matrices from Weil's Explicit Formula}
\author[Yaoming Shi]{Yaoming Shi}
\address{California, United States}
\email{ymshi@protonmail.com}
\date{Version 12 of \today}

\begin{document}

\begin{abstract}
Starting from the Riemann--$\Xi$ specialization of Weil's explicit formula, we
construct finite real-symmetric Prime--Weil matrices $(\mathbf S)$ from pole,
archimedean, and finite prime-power data.  This construction is a
finite-dimensional arithmetic model within the Hilbert--P\'olya program, which
seeks a self-adjoint spectral realization of the nontrivial zeta-zero
parameters.  Their off-diagonal entries form a
Loewner-type divided-difference matrix with a rank-two displacement identity.
We formulate the spectral quotient as a Hermitian definite generalized
eigenproblem on the fixed zero-mean contrast space.  This realization is
invariant under positive affine rescaling, avoids ground-vector normalization
and an ill-conditioned oblique projector, and preserves the finite quotient
spectrum.

For a dimension-matched zero-side matrix built from $N$ distinct positive
ordinates $\gamma_k$, rational interpolation gives the exact contrast-pencil
spectrum $\{\pm\gamma_1,\ldots,\pm\gamma_N\}$ and positive-parity square
spectrum $\{\gamma_1^2,\ldots,\gamma_N^2\}$.  Since the ordinates are inputs,
this is a reconstruction theorem.  Assuming RH, the same interpolation vector
proves $\lambda_{\min}(\mathbf S)\to0$.  At $N=L=13$, Lemke's
ground-state quotient and the contrast pencil agree numerically and reproduce
the first three zeta ordinates to the reported precision.  The remaining
problem is a relative prime-to-zero perturbation theorem with uniform control
of the compressed metric.  No proof of RH is claimed.
\end{abstract}

\maketitle
\tableofcontents

\section{Introduction}\label{sec:introduction}

\subsection{The Hilbert--P\'olya objective and Weil's quadratic form}

For a nontrivial zero $\rho$ of the Riemann zeta function, introduce the
spectral parameter
\begin{equation}\label{eq:intro-spectral-parameter}
  z_\rho:=\frac{\rho-\tfrac12}{i}.
\end{equation}
The Riemann hypothesis is equivalent to the assertion that every $z_\rho$ is
real.  The Hilbert--P\'olya program therefore seeks a natural self-adjoint
operator whose spectrum consists of these parameters, with the correct
multiplicities.  A finite list of fitted eigenvalues is not enough: one must
construct the operator from arithmetic data, prove self-adjointness in a
specified Hilbert structure, recover the complete zero set, and control the
limit strongly enough to exclude spectral pollution.

Several complementary directions frame this objective.  Montgomery's
pair-correlation work and Odlyzko's computations motivate the random-matrix
and quantum-chaos picture \cite{Montgomery1973,Odlyzko1987}; Berry and
Keating relate the mean zero-counting law to the dilation Hamiltonian
$H=xp$ \cite{BerryKeating1999SIAM}; de Branges places the problem in
Hilbert spaces of entire functions and canonical systems
\cite{deBranges1986}; and the regularized-determinant and noncommutative
trace-formula programs of Deninger and Connes seek broader arithmetic
spectral realizations \cite{Deninger1992,Connes1999}.  For a
physics-oriented overview of these connections, see \cite{Schumayer2011}.
The present approach belongs to the explicit-formula branch of this
literature: it starts from Weil's arithmetic quadratic form and asks whether
its finite prime-built compressions admit a controlled self-adjoint spectral
limit.

Weil's explicit formula \cite{Weil1952} is a natural starting point because it identifies the
same quadratic form from the zero side and from the local arithmetic side.  In
the Riemann case, schematically,
\begin{equation}\label{eq:intro-explicit-formula-schematic}
  \sum_\rho^{*}H(z_\rho)
  =\mathcal C_{\infty}(F)
   -\sum_p\sum_{r\geq1}\frac{\log p}{p^{r/2}}
    \bigl(F(r\log p)+F(-r\log p)\bigr),
\end{equation}
where $H$ is the Fourier transform of the additive test function $F$,
$\mathcal C_{\infty}$ contains the pole and archimedean terms, and the star
denotes symmetric zero summation when needed.  For autocorrelation test
functions, the transform is a squared modulus on the critical line.  Under
RH, the zero-side quadratic form is consequently positive.  The present paper
compresses a carefully chosen family of these forms into finite matrices and
then asks how to pass from the resulting metric to a spectral operator.

\subsection{Finite Prime--Weil matrices}

Fix $L>0$ and $N\geq0$, and set
\begin{equation}\label{eq:intro-frequency-nodes}
  \nu_{n,L}:=\frac{2\pi n}{L},
  \qquad -N\leq n\leq N.
\end{equation}
Real polarizations of one-sided logarithmic Fourier autocorrelations produce
a $(2N+1)\times(2N+1)$ real-symmetric matrix
$\mathbf S_{N,L}$.  Its entries are evaluated from the pole,
archimedean, and prime-power sides of Weil's formula.  Compact support in the
logarithmic variable makes the prime-power sum finite.  The same matrix has a
zero-side expansion into rank-two atoms, but the two descriptions play
different logical roles: the arithmetic formula defines the proposed input,
while the zero formula supplies positivity under RH, exact finite models, and
comparison estimates.

The matrix has a divided-difference structure.  With
$\mathbf D_{L,N}=\operatorname{diag}(\nu_{n,L})$, its commutator has rank at
most two.  This low displacement rank is the algebraic mechanism behind the
quotient construction.

\subsection{A scale-invariant quotient pencil}

The principal revision of the present version is that the quotient is no
longer implemented by first normalizing a ground eigenvector and then forming
an oblique projection.  Let
\[
  \epsilon_{N,L}:=\lambda_{\min}(\mathbf S_{N,L}),
  \qquad
  \mathbf W_{N,L}:=\mathbf S_{N,L}-\epsilon_{N,L}\mathbf I.
\]
Let $\boldsymbol\delta_{N,L}$ be the Fourier evaluation vector at the origin,
and let $\mathbf C_{N,L}$ have Euclidean-orthonormal columns spanning the
fixed contrast space
\[
  \ker\boldsymbol\delta_{N,L}^{*}
  =\left\{\mathbf x:\sum_{n=-N}^{N}x_n=0\right\}.
\]
The quotient metric and differentiation form are represented directly by
\begin{equation}\label{eq:intro-contrast-pencil}
  \mathbf G_{N,L}:=\mathbf C_{N,L}^{*}\mathbf W_{N,L}\mathbf C_{N,L},
  \qquad
  \mathbf K_{N,L}:=\mathbf C_{N,L}^{*}\mathbf W_{N,L}
                    \mathbf D_{L,N}\mathbf C_{N,L}.
\end{equation}
The rank-two commutator implies that $\mathbf K_{N,L}$ is Hermitian.  If the
least eigenvalue is simple and $\mathbf G_{N,L}\succ0$, the generalized
problem
\begin{equation}\label{eq:intro-generalized-pencil}
  \mathbf K_{N,L}\mathbf y
  =\mu\,\mathbf G_{N,L}\mathbf y
\end{equation}
is Hermitian definite and therefore has real spectrum.  Reflection splits the
contrast space into equal even and odd sectors, makes the whitened operator
off diagonal, and forces opposite-sign eigenvalue pairing.  The corresponding
positive-parity square retains one copy of each squared positive eigenvalue.

This formulation has concrete advantages over Lemke's ground-state
oblique-projection construction \cite{Lemke2026}.  It uses a fixed contrast
space, does not require a ground eigenvector, never divides by a potentially
tiny overlap $\boldsymbol\delta^{*}\mathbf e$, and never explicitly forms a
projector whose norm may diverge.  It is also exactly invariant under
$\mathbf S\mapsto a\mathbf S+b\mathbf I$ with $a>0$.  The improvement is not
merely cosmetic: it removes an avoidable source of numerical instability and
reduces the computation to a standard Hermitian definite pencil.  It does not,
however, eliminate genuine degeneration of the compressed metric; that
intrinsic conditioning is measured directly by the smallest eigenvalue and
condition number of $\mathbf G_{N,L}$.

The finite real-spectrum construction that initiated this part of the
Hilbert--P\'olya--Weil program was communicated by S\"oren Lemke
\cite{Lemke2026}.  Starting from the
arithmetic Weil matrix, Lemke subtracts its simple least eigenvalue, normalizes
the associated even eigenvector against the evaluation functional, forms an
oblique projection along that ground direction, and applies the diagonal
frequency matrix to the resulting quotient.  For $N=L=13$, his implementation
gave the first three positive quotient eigenvalues
\begin{equation}\label{eq:intro-Lemke-N13-values}
\begin{split}
  14.1347251417346937904572519851,\qquad
  21.0220396387715551447475167681,\\
  25.0108575801628431246104531742,
\end{split}
\end{equation}
matching the first three zeta ordinates with absolute errors approximately
$1.54\times10^{-27}$, $1.52\times10^{-16}$, and
$1.72\times10^{-11}$, respectively.  Lemke's calculation supplied the original
finite numerical evidence and the quotient mechanism.  The contrast-pencil
formulation below realizes the same quotient spectrum on fixed coordinates,
without the small-overlap division or explicit oblique projector.

A different finite-matrix use of Weil's form was developed by Alp\"oge and
Furman \cite{AlpogeFurman2026}.  Their high-energy matrix is used through
rank, trace, Hilbert--Schmidt norm, and inertia to count simple critical-line
zeros; its eigenvalues are not intended to be individual zero ordinates.  The
present low-frequency matrix is used as a metric in the Hermitian pencil
\cref{eq:intro-generalized-pencil}.  The two constructions therefore compress
the same explicit-formula quadratic form for different purposes.

\subsection{Exact zero-side reconstruction}

Suppose the dimension-matched zero-side matrix is formed from exactly $N$
distinct positive real ordinates
\begin{equation}\label{eq:intro-positive-ordinates}
  0<\gamma_1<\cdots<\gamma_N.
\end{equation}
Its Gram representation is positive semidefinite and has a one-dimensional
nullspace.  Rational interpolation determines that null vector explicitly.
Restricting the metric and differentiation form to the zero-mean contrast
space gives a positive-definite pencil satisfying
\begin{equation}\label{eq:intro-zero-side-characteristic-polynomial}
  \det\!\left(z\mathbf G_{N,L}^{(0)}-\mathbf K_{N,L}^{(0)}\right)
  =\det(\mathbf G_{N,L}^{(0)})
   \prod_{k=1}^{N}(z^2-\gamma_k^2).
\end{equation}
Hence its generalized spectrum is
$\{\pm\gamma_1,\ldots,\pm\gamma_N\}$ and its positive-parity square has
spectrum
\begin{equation}\label{eq:intro-zero-side-square-spectrum}
  \{\gamma_1^2,\ldots,\gamma_N^2\}.
\end{equation}
The null direction is absent because the pencil is already written on the
$2N$-dimensional contrast space.  In particular, the state with squared
energy $\gamma_1^2$ is a desired spectral state and is not deflated.

This theorem is an exact consistency test, not an independent derivation of
the zeros: the $\gamma_k$ are inputs.  Hypothetical off-critical-line pairs
$\gamma_k\pm i\eta_k$ produce the quartet
\begin{equation}\label{eq:intro-off-line-quartet}
  \{\pm(\gamma_k+i\eta_k),\ \pm(\gamma_k-i\eta_k)\}
\end{equation}
in the corresponding indefinite interpolation pencil.  Shifting the paired
matrix by its algebraically smallest eigenvalue instead gives a positive
pencil with a real sign-paired spectrum, but that is a different spectral
problem and does not replace the quartet by its radial modulus.

\subsection{Arithmetic approximation and the remaining transfer problem}

The paper proves quantitative simultaneous prime-power and zero-tail bounds
for fixed test functions and uniformly norm-bounded families.  It also gives
closed digamma--polygamma--Lerch evaluations of the archimedean integrals and
compensated formulas that expose the cancellation between large pole and
prime terms.  Because the logarithmic Fourier windows vary with $N$ and $L$,
these fixed-function estimates alone do not settle the moving-dimension
limit.

Under RH, the exact finite-zero interpolation vector is a trial state for the
full arithmetic matrix.  The first $N$ zero atoms annihilate it, and only the
positive zero tail remains.  This proves
\[
  \lambda_{\min}(\mathbf S_{N,N})\longrightarrow0
\]
and gives a much stronger product upper bound.  The new contrast formulation
shows why the vanishing eigenvalue need not be divided out numerically: after
a common gap normalization, it enters only through the compressed Hermitian
pencil.

It is useful to separate the established conclusions from the limiting
conjecture.
\begin{enumerate}
  \item The explicit-formula entries, divided-difference identities,
        rank-two commutator, Hermiticity of the contrast pencil, affine-scale
        invariance, and parity reduction are exact finite-dimensional facts.
  \item With real ordinates inserted as data, the zero-side contrast pencil
        reconstructs those ordinates exactly.
  \item Under RH, the least eigenvalue of the full arithmetic matrix tends to
        zero along $L=N$.
  \item Prime-side computations reproduce low zero ordinates to high
        precision for the tested parameters, but this remains numerical
        evidence rather than a convergence theorem.
\end{enumerate}

The central unresolved task is no longer the algebraic removal of a small
ground state.  It is a relative perturbation theorem comparing the entire
arithmetic pencil with the exact finite-zero pencil in the quotient-energy
norm.  Such a theorem must control the positive zero tail relative to the
compressed metric, preserve separated generalized eigenvalues and
multiplicities, and lead to a limiting operator or determinant without
spectral pollution.  These requirements formulate the principal obstacle in
the Hilbert--P\'olya--Weil program developed here.

\subsection{Organization of the paper}

After the Riemann--$\Xi$ specialization and simultaneous truncation theorem,
the paper constructs the finite Prime--Weil matrices and proves their
divided-difference and rank-two displacement identities.  The next section
develops the scale-invariant contrast pencil, its parity reduction, explicit
entries, conditioning bounds, unshifted regularization, and the relative
prime-to-zero transfer target.  The dimension-matched zero-side pencil is then
solved exactly.  Subsequent sections prove the conditional vanishing of the
least arithmetic eigenvalue, analyze off-critical-line quartets, realize each
transform-side atom as an operator on the periodic Fourier space, and study a
positive shifted pencil for off-line data.  The final part gives the revised
pencil-based numerical protocol, followed by the conclusion and sample test
functions.

\section{Specialization to the Riemann \texorpdfstring{$\Xi$}{Xi}-function and simultaneous truncation}\label{sec:Xi-truncation}

This section is a supplement to Weil's explicit formula.  It proves a
quantitative simultaneous truncation result.  The truncation theorem is not
stated in Weil's 1952 paper; it follows from his explicit formula, elementary
estimates for the von Mangoldt function, and the classical zero-counting
estimate for $\zeta$.

\subsection{The exact \texorpdfstring{$\Xi$}{Xi}-specialization}
Define the Riemann $\xi$ function and the Riemann $\Xi$-function by
\begin{equation}\label{eq:xi-Xi-definitions-2}
	\xi(s)
	:=\frac12s(s-1)\pi^{-s/2}\Gamma\!\left(\frac{s}{2}\right)\zeta(s),
	\qquad
	\Xi(z):=\xi\!\left(\frac12+iz\right).
\end{equation}
The Riemann hypothesis (RH) is the assertion that every zero of
$\Xi(z)$ is real.

Suppose $H$ is holomorphic in the strip
$|\ImPart z|\leq \frac12+\delta$ and satisfies
$H(z)=O((1+|z|)^{-2-b})$ there for some $b>0$, and let 
\begin{equation}
	F(u)=\tfrac1{2\pi}\int_{-\infty}^\infty H(r)e^{-iur}\mathrm{d}r.
\end{equation}
Then we have the following duality between primes and zeros of Riemann $\Xi$ function:
\begin{align}\label{eq:Xi-Weil-explicit}
  \sum_{\gamma}^{*}H(\gamma)
  ={}&H\!\left(\frac{i}{2}\right)
      +H\!\left(-\frac{i}{2}\right)
      -F(0)\log(2\pi)
      -\mathcal A_\infty(F)
  \notag\\
  &-\sum_p\sum_{m=1}^{\infty}
    \frac{\log p}{p^{m/2}}
    \bigl[F(m\log p)+F(-m\log p)\bigr].
\end{align}
where
\begin{equation}\label{eq:A-infinity-zeta}
	\mathcal A_\infty(F)
	:=\PF\!\int_{-\infty}^{\infty}
	F(x)\frac{e^{\abs{x}/2}}{\abs{e^x-e^{-x}}}\dd x.
\end{equation}
Here and below, the sum over $\gamma$ is over the zeros of the Riemann
$\Xi$-function, with multiplicity, and the star denotes the symmetric limit
with respect to $\abs{\gamma}$.

Let $\Lambda(n)$ denote the von Mangoldt function.  Since
$\Lambda(p^m)=\log p$ and $\Lambda(n)=0$ otherwise, the prime-power term can
also be written as
\begin{equation}\label{eq:von-Mangoldt-prime-side}
  \mathcal P(F)
  :=\sum_{n=2}^{\infty}
    \frac{\Lambda(n)}{\sqrt n}
    \bigl[F(\log n)+F(-\log n)\bigr].
\end{equation}
Accordingly, if
\begin{equation}\label{eq:C-Xi-functional}
  \mathcal C_\Xi(F)
  :=H\!\left(\frac{i}{2}\right)
    +H\!\left(-\frac{i}{2}\right)
    -F(0)\log(2\pi)
    -\mathcal A_\infty(F),
\end{equation}
then \cref{eq:Xi-Weil-explicit} is the compact identity
\begin{equation}\label{eq:Xi-explicit-compact}
  \mathcal Z(F):=\sum_{\rho}^{*}H(z_\rho)
  =\mathcal C_\Xi(F)-\mathcal P(F).
\end{equation}

\begin{remark}[Even test functions]\label{rem:even-Xi-formula}
If $F$ is even, then $H$ is even,
$H(i/2)=H(-i/2)$, and
\cref{eq:Xi-Weil-explicit} becomes
\begin{align}\label{eq:Xi-Weil-even}
  \sum_{\gamma}^{*}H(\gamma)
  ={}&2H\!\left(\frac{i}{2}\right)
      -F(0)\log(2\pi)
      -\mathcal A_\infty(F)
  \notag\\
  &-2\sum_{n=2}^{\infty}
    \frac{\Lambda(n)}{\sqrt n}F(\log n).
\end{align}
This is the natural form when one works directly with the even entire
function $\Xi$.
\end{remark}

\subsection{Equivalent spectral-density form of the archimedean term}
\label{subsec:archimedean-spectral-density}

The physical-space finite-part term in \cref{eq:A-infinity-zeta} is often
written instead as a spectral integral involving the logarithmic derivative
of the gamma factor.  For even $F\in C_c^2(\mathbb R)$, define
\begin{equation}\label{eq:archimedean-spectral-density}
  \mu_\infty(t)
  :=\frac1{2\pi}\RePart\frac{\Gamma'}{\Gamma}
       \!\left(\frac14+\frac{it}{2}\right)
    -\frac{\log\pi}{2\pi}.
\end{equation}
For even $F$, the transform $H(t)=\int_{\mathbb R}F(x)e^{itx}\dd x$ also
agrees with the opposite-sign Fourier convention.

\begin{proposition}[Archimedean Fourier equivalence]
\label{prop:archimedean-Fourier-equivalence}
For every even $F\in C_c^2(\mathbb R)$,
\begin{equation}\label{eq:archimedean-Fourier-equivalence}
  \boxed{
  \int_{-\infty}^{\infty}H(t)\mu_\infty(t)\dd t
  =-F(0)\log(2\pi)-\mathcal A_\infty(F).}
\end{equation}
Equivalently,
\begin{equation}\label{eq:archimedean-digamma-equivalence}
  \frac1{2\pi}\int_{-\infty}^{\infty}H(t)
  \RePart\frac{\Gamma'}{\Gamma}
       \!\left(\frac14+\frac{it}{2}\right)\dd t
  =-F(0)\log2-\mathcal A_\infty(F).
\end{equation}
Moreover,
\begin{equation}\label{eq:archimedean-even-half-line}
  \mathcal A_\infty(F)
  =2\PF\!\int_0^\infty
    F(x)\frac{e^{x/2}}{e^x-e^{-x}}\dd x.
\end{equation}
\end{proposition}

\begin{proof}
The spectral-density form of Weil's explicit formula has the same zero,
pole, and prime-power terms as \cref{eq:Xi-Weil-even}, while its
archimedean contribution is the left-hand side of
\cref{eq:archimedean-Fourier-equivalence}; see, for example,
\cite{AlpogeFurman2026}.  Comparing the two normalizations gives
\cref{eq:archimedean-Fourier-equivalence}.  Since Fourier inversion gives
$\int_{\mathbb R}H(t)\dd t=2\pi F(0)$, substituting
\cref{eq:archimedean-spectral-density} yields
\cref{eq:archimedean-digamma-equivalence}.  Finally,
\cref{eq:archimedean-even-half-line} follows from evenness.

For a direct distributional check, the digamma integral representation gives
\begin{equation}\label{eq:archimedean-digamma-kernel}
  \RePart\frac{\Gamma'}{\Gamma}
  \!\left(\frac14+\frac{it}{2}\right)
  =-\gamma_{\mathrm E}
   +2\int_0^\infty
      \frac{e^{-2x}-e^{-x/2}\cos(tx)}{1-e^{-2x}}\dd x.
\end{equation}
The two terms in the numerator must be paired before the $x$-integration near
zero.  Pairing \cref{eq:archimedean-digamma-kernel} with $H$, applying Fourier
inversion, and using the finite-part convention fixed by
\cref{eq:A-infinity-zeta} gives precisely
$-F(0)\log2-\mathcal A_\infty(F)$.  The subtraction
$-(\log\pi)/(2\pi)$ in \cref{eq:archimedean-spectral-density} supplies the
remaining term $-F(0)\log\pi$.
\end{proof}

\begin{remark}[The local constant is essential]
The bare kernel integral and the bare digamma integral are not separately
ordinary Fourier transforms of one another at the origin.  Their exact
equivalence is the regularized distributional identity
\cref{eq:archimedean-Fourier-equivalence}; omitting either the finite-part
prescription or the local term $-F(0)\log(2\pi)$ changes the normalization by
a multiple of the Dirac mass at zero.
\end{remark}

\begin{theorem}[Weil \cite{Weil1952}]
RH holds if and only if
\[
  \sum_{\gamma}^{*}H(\gamma)\geq0
\]
for every admissible even function $H$ of the form
\[
  H(r)=H_0(r)\overline{H_0(\overline r)}.
\]
\end{theorem}

\subsection{Prime-power and zero truncations}

For an integer $T\geq3$ and a real number $T_1\geq2$, define
\begin{align}\label{eq:truncated-Xi-sums}
  \mathcal P_T(F)
  &:=\sum_{\substack{p,m:\\ p^m\leq T}}
     \frac{\log p}{p^{m/2}}
     \bigl[F(m\log p)+F(-m\log p)\bigr]
  \notag\\
  &=\sum_{2\leq n\leq T}
     \frac{\Lambda(n)}{\sqrt n}
     \bigl[F(\log n)+F(-\log n)\bigr],
  \\
  \mathcal Z_{T_1}(F)
  &:=\sum_{\substack{\rho=\beta+i\gamma\\ \abs{\gamma}<T_1}}
     H(z_\rho).
\end{align}
Define the truncated discrepancy by
\begin{equation}\label{eq:truncated-discrepancy}
  \mathcal E_{T,T_1}(F)
  :=\mathcal Z_{T_1}(F)-\mathcal C_\Xi(F)+\mathcal P_T(F).
\end{equation}
Then the exact finite formula is
\begin{equation}\label{eq:truncated-Xi-formula}
  \mathcal Z_{T_1}(F)
  =\mathcal C_\Xi(F)-\mathcal P_T(F)
   +\mathcal E_{T,T_1}(F).
\end{equation}
The issue is therefore whether $\mathcal E_{T,T_1}(F)\to0$ as
$T,T_1\to\infty$, and in what sense this convergence is uniform.

After a single test function $F$ has been fixed, the explicit formula is a
scalar identity and there is no remaining free spectral variable.  Thus a
uniform-convergence statement must refer either to a class of test functions
or to a parameterized family of such functions.  Weil's original hypotheses
ensure that $\mathcal P_T(F)\to\mathcal P(F)$ absolutely and that
$\mathcal Z_{T_1}(F)\to\mathcal Z(F)$ in the symmetric sense used to define
the starred sum.  For an absolute, quantitative, and uniform zero-tail
estimate, it is useful to impose two weighted derivatives.

\subsection{The weighted local Sobolev test class}\label{subsec:weighted-Sobolev-class}

\begin{definition}[The space $W_{\mathrm{loc}}^{2,1}(\mathbb R)$]
A locally integrable function $F:\mathbb R\to\mathbb C$ belongs to
$W_{\mathrm{loc}}^{2,1}(\mathbb R)$ if its first and second distributional
derivatives are represented by locally integrable functions.  Equivalently,
\begin{equation}\label{eq:Wloc21-definition}
  W_{\mathrm{loc}}^{2,1}(\mathbb R)
  :=\left\{
       F\in L_{\mathrm{loc}}^1(\mathbb R):
       DF,D^2F\in L_{\mathrm{loc}}^1(\mathbb R)
     \right\}.
\end{equation}
More explicitly, there exist $G_1,G_2\in L_{\mathrm{loc}}^1(\mathbb R)$
such that
\begin{align}\label{eq:Wloc21-distributional-identities}
  \int_{\mathbb R}F(x)\varphi'(x)\dd x
  &=-\int_{\mathbb R}G_1(x)\varphi(x)\dd x,
  \\
  \int_{\mathbb R}G_1(x)\varphi'(x)\dd x
  &=-\int_{\mathbb R}G_2(x)\varphi(x)\dd x
\end{align}
for every $\varphi\in C_c^\infty(\mathbb R)$.  In this case $DF=G_1$ and
$D^2F=G_2$ almost everywhere.  Equivalently,
\begin{equation}\label{eq:Wloc21-bounded-interval}
  F\big|_I\in W^{2,1}(I)
  \qquad\text{for every bounded interval }I\subset\mathbb R,
\end{equation}
and hence
\[
  \int_I\bigl(\abs{F}+\abs{F'}+\abs{F''}\bigr)\dd x<\infty.
\]
In one dimension, after choosing the canonical Sobolev representative, this
is equivalent to
\begin{equation}\label{eq:Wloc21-AC-characterization}
  F\in C^1(\mathbb R),
  \qquad
  F'\in AC_{\mathrm{loc}}(\mathbb R),
\end{equation}
so $F''$ exists almost everywhere and is locally integrable.
\end{definition}

The word \emph{local} in \cref{eq:Wloc21-definition} imposes no decay or
integrability condition at infinity.  Such conditions must be stated
separately; see, for example, \cite{AdamsFournier2003} for the standard
Sobolev theory.

Fix $b>0$.  Let $\mathcal A_b^2$ be the class of functions that
satisfy Weil's standard piecewise-$C^1$ midpoint convention and the decay
$F(x),F'(x)=O(e^{-(\frac12+b)\abs{x}})$ as $\abs{x}\to\infty$,
belong to $W_{\mathrm{loc}}^{2,1}(\mathbb R)$, and satisfy
\begin{align}\label{eq:Mb-J2-definitions}
  M_b(F)
  &:=\sup_{x\in\mathbb R}
     e^{(\frac12+b)\abs{x}}\abs{F(x)}<\infty,
  \\
  J_2(F)
  &:=\int_{-\infty}^{\infty}e^{\abs{x}/2}
     \left(
       \abs{F(x)}+\abs{F'(x)}+\abs{F''(x)}
     \right)\dd x<\infty.
\end{align}
Set
\begin{equation}\label{eq:Ab2-norm}
  \norm{F}_{\mathcal A_b^2}:=M_b(F)+J_2(F).
\end{equation}
The canonical representative in
\cref{eq:Wloc21-AC-characterization} is used in the supremum defining
$M_b(F)$.  Weil's original hypotheses ensure that the untruncated explicit
formula is applicable, while the two quantities in
\cref{eq:Mb-J2-definitions} provide the quantitative bounds below.  This
formulation avoids invoking a separate density extension of Weil's formula to
an abstract Sobolev completion.

\subsection{Uniform convergence of the simultaneous truncations}

\begin{theorem}[Uniform simultaneous truncation]
\label{thm:uniform-Xi-truncation}
Let $b>0$.  There is an absolute constant $C_0>0$ such that, for every
$F\in\mathcal A_b^2$, every integer $T\geq3$, and every $T_1\geq2$,
\begin{align}\label{eq:uniform-truncation-bound}
  \abs{\mathcal E_{T,T_1}(F)}
  \leq{}&
  2M_b(F)T^{-b}
  \left(\frac{\log T}{b}+\frac{1}{b^2}\right)
  \notag\\
  &+C_0J_2(F)
  \frac{\log(T_1+2)+1}{T_1}.
\end{align}
Consequently,
\begin{equation}\label{eq:uniform-truncation-limit}
  \mathcal Z_{T_1}(F)-\mathcal C_\Xi(F)+\mathcal P_T(F)
  \longrightarrow0
  \qquad (T,T_1\to\infty),
\end{equation}
and the convergence is uniform on every norm-bounded subset of
$\mathcal A_b^2$.  The two cutoffs may tend to infinity independently; no
coupling between $T$ and $T_1$ is required.
\end{theorem}

\begin{proof}
We estimate the prime-power and zero tails separately.

For the prime-power tail, the definition of $M_b(F)$ gives, with $n=p^m$,
\[
  \abs{F(\log n)}+\abs{F(-\log n)}
  \leq2M_b(F)n^{-1/2-b}.
\]
Therefore
\begin{align*}
  \abs{\mathcal P(F)-\mathcal P_T(F)}
  &\leq2M_b(F)
    \sum_{n>T}\frac{\Lambda(n)}{n^{1+b}}
  \\
  &\leq2M_b(F)
    \sum_{n>T}\frac{\log n}{n^{1+b}}.
\end{align*}
For $T\geq3$, the function $t\mapsto(\log t)t^{-1-b}$ is decreasing.
Since $T$ is an integer, the integral test yields
\begin{align}\label{eq:prime-tail-bound}
  \abs{\mathcal P(F)-\mathcal P_T(F)}
  &\leq2M_b(F)
    \int_T^\infty\frac{\log t}{t^{1+b}}\dd t
  \notag\\
  &=2M_b(F)T^{-b}
    \left(\frac{\log T}{b}+\frac{1}{b^2}\right).
\end{align}
In particular, the prime-power series is absolutely convergent.

For the zero tail, let $\rho=\beta+i\gamma$ be a nontrivial zero and put
\[
  a_\rho:=\beta-\frac12,
  \qquad
  G_\rho(x):=F(x)e^{a_\rho x}.
\]
Since $0<\beta<1$, one has $\abs{a_\rho}\leq\frac12$.  The definition of
$J_2(F)$ implies that $G_\rho\in W^{2,1}(\mathbb R)$, uniformly with respect
to $\rho$, and that $G_\rho$ and $G_\rho'$ vanish at both ends of the real
line.  Hence two integrations by parts are justified and, for $\gamma\neq0$,
\begin{align*}
  \Phi(\rho)
  &=\int_{-\infty}^{\infty}G_\rho(x)e^{i\gamma x}\dd x
  \\
  &=-\frac{1}{\gamma^2}
    \int_{-\infty}^{\infty}G_\rho''(x)e^{i\gamma x}\dd x.
\end{align*}
Moreover,
\[
  G_\rho''(x)
  =e^{a_\rho x}
   \bigl(F''(x)+2a_\rho F'(x)+a_\rho^2F(x)\bigr),
\]
so
\begin{equation}\label{eq:Phi-zero-decay}
  \abs{\Phi(\rho)}
  \leq\frac{1}{\gamma^2}
  \int_{-\infty}^{\infty}e^{\abs{x}/2}
  \left(
    \abs{F''(x)}+\abs{F'(x)}+\frac14\abs{F(x)}
  \right)\dd x
  \leq\frac{J_2(F)}{\gamma^2}.
\end{equation}

Let $N(t)$ count the zeros $\rho$ with $0<\gamma\leq t$, including
multiplicity.  The Riemann--von Mangoldt estimate implies
\begin{equation}\label{eq:N-upper-bound}
  N(t)\leq C_1t\log(t+2)
  \qquad (t\geq2)
\end{equation}
for an absolute constant $C_1$; see, for example,
\cite{HasanalizadeShenWong2022}.  Stieltjes integration by parts gives
\begin{align*}
  \sum_{\gamma\geq T_1}\frac{1}{\gamma^2}
  &=\int_{T_1^-}^{\infty}t^{-2}\dd N(t)
  \\
  &=-\frac{N(T_1^-)}{T_1^2}
    +2\int_{T_1}^{\infty}\frac{N(t)}{t^3}\dd t
  \\
  &\leq C_2\frac{\log(T_1+2)+1}{T_1}.
\end{align*}
The zeros with negative ordinates satisfy the same estimate by the functional
equation and complex conjugation.  Combining this observation with
\cref{eq:Phi-zero-decay}, we obtain
\begin{equation}\label{eq:zero-tail-bound}
  \sum_{\abs{\gamma}\geq T_1}\abs{\Phi(\rho)}
  \leq C_0J_2(F)
  \frac{\log(T_1+2)+1}{T_1}.
\end{equation}
Thus the zero sum is absolutely convergent under the strengthened hypotheses
of this subsection.

Finally, \cref{eq:Xi-explicit-compact,eq:truncated-discrepancy} imply
\begin{align*}
  \mathcal E_{T,T_1}(F)
  ={}&-\bigl(\mathcal P(F)-\mathcal P_T(F)\bigr)
  \\
  &-\bigl(\mathcal Z(F)-\mathcal Z_{T_1}(F)\bigr).
\end{align*}
The triangle inequality together with
\cref{eq:prime-tail-bound,eq:zero-tail-bound} proves
\cref{eq:uniform-truncation-bound}.  If
$\norm{F}_{\mathcal A_b^2}\leq M$, its right-hand side is bounded by $M$
times a quantity depending only on $(b,T,T_1)$ and tending to zero.  This
proves the asserted uniform convergence.
\end{proof}

\begin{corollary}[Single-parameter uniform truncation]
\label{cor:parameter-uniformity}
There is an absolute constant $C_3>0$ such that, for every
$F\in\mathcal A_1^2$ and every integer $T\geq3$,
\begin{align}\label{eq:single-parameter-truncation-bound}
  \abs{\mathcal E_{T,T}(F)}
  &\leq
  2M_1(F)\frac{\log T+1}{T}
  +C_0J_2(F)\frac{\log(T+2)+1}{T}
  \notag\\
  &\leq
  C_3\norm{F}_{\mathcal A_1^2}
  \frac{\log(T+2)+1}{T}.
\end{align}
Consequently,
\begin{equation}\label{eq:single-parameter-truncation-limit}
  \mathcal Z_T(F)-\mathcal C_\Xi(F)+\mathcal P_T(F)
  \longrightarrow0
  \qquad (T\to\infty),
\end{equation}
uniformly on every norm-bounded subset of $\mathcal A_1^2$.  More generally,
if $\Omega$ is any parameter set and
$\{F_\tau:\tau\in\Omega\}\subset\mathcal A_1^2$ satisfies
\[
  \sup_{\tau\in\Omega}
  \norm{F_\tau}_{\mathcal A_1^2}<\infty,
\]
then the convergence in \cref{eq:single-parameter-truncation-limit} is uniform
in $\tau\in\Omega$.
\end{corollary}

\begin{remark}[The diagonal truncated explicit formula]
\label{rem:diagonal-truncated-explicit-formula}
For $F\in\mathcal A_1^2$, put
$\mathcal R_T(F):=\mathcal E_{T,T}(F)$.  Then
\begin{align}\label{eq:Xi-Weil-explicit-diagonal}
  \sum_{\substack{\rho=\beta+i\gamma\\\abs{\gamma}<T}}
  H(z_\rho)
  ={}&H\!\left(\frac{i}{2}\right)
      +H\!\left(-\frac{i}{2}\right)
      -F(0)\log(2\pi)
      -\mathcal A_\infty(F)
  \notag\\
  &-\sum_{\substack{p,m\geq1\\p^m\leq T}}
    \frac{\log p}{p^{m/2}}
    \bigl[F(m\log p)+F(-m\log p)\bigr]
    +\mathcal R_T(F),
\end{align}
where
\begin{equation}\label{eq:diagonal-remainder-bound}
  \abs{\mathcal R_T(F)}
  \leq C_3\norm{F}_{\mathcal A_1^2}
       \frac{\log(T+2)+1}{T}.
\end{equation}
Thus, for fixed $F$, the remainder is
$O_F((\log(T+2)+1)/T)$.  No truncation is made in the pole terms or in
$\mathcal A_\infty(F)$; these are fixed components of
$\mathcal C_\Xi(F)$.
\end{remark}

\begin{remark}[Compactly supported test functions]
\label{rem:compact-support-truncation}
If $\operatorname{supp}F\subseteq[-R,R]$, then the prime-power truncation is
already exact for $T\geq e^R$, because every remaining term contains
$F(\pm\log n)=0$.  If in addition $F\in C_c^r(\mathbb R)$ with $r\geq2$,
then $r$ integrations by parts give
$\abs{\Phi(\rho)}\ll_{F,r}\abs{\gamma}^{-r}$, and the same zero-counting
argument yields the sharper tail
\[
  \sum_{\abs{\gamma}\geq T_1}\abs{\Phi(\rho)}
  \ll_{F,r}\frac{\log(T_1+2)}{T_1^{r-1}}.
\]
Thus smooth compact support makes the arithmetic side finite and gives
arbitrarily rapid polynomial convergence on the zero side.
\end{remark}

\section{Finite Prime--Weil matrices from logarithmic Fourier autocorrelations}
\label{sec:finite-prime-Weil-matrix}

We now apply the specialized explicit formula to a concrete family of
one-sided logarithmic Fourier modes.  Their real polarizations produce a
real-symmetric Weil matrix with a divided-difference structure, an exactly
finite prime-power side, and displacement rank at most two.  These identities
supply the arithmetic matrix used in the scale-invariant quotient construction
of \cref{sec:scale-invariant-quotient}.

Retain $L=\log T$, $I_N=\{-N,\ldots,N\}$, and
$\nu_{n,L}=2\pi n/L$, and define
\begin{equation}\label{eq:one-sided-Fourier-mode}
  f_{n,L}^{+}(x)
  :=\frac1L e^{i\nu_{n,L}x}\mathbf 1_{[0,L]}(x),
  \qquad n\in I_N.
\end{equation}
The endpoint values are immaterial for the convolution integrals below.  With
the involution and convolution from
\cref{eq:sample-convolution-definition}, set
\begin{align}\label{eq:Gmn-one-sided-definition}
  C_{mn,L}
  &:=f_{m,L}^{+}*\widetilde{f_{n,L}^{+}},
  \\
  G_{mn,L}
  &:=\RePart C_{mn,L}
   =\frac12\left(
       f_{m,L}^{+}*\widetilde{f_{n,L}^{+}}
       +f_{-m,L}^{+}*\widetilde{f_{-n,L}^{+}}
     \right).
\end{align}
Thus $G_{mn,L}$ is the real polarization of the Weil autocorrelation pairing
associated with the pair $(f_{m,L}^{+},f_{n,L}^{+})$.

\subsection{Physical-space formula and divided differences}

Since
\[
  \widetilde{f_{n,L}^{+}}(x)
  =\frac1L e^{i\nu_{n,L}x}\mathbf 1_{[-L,0]}(x),
\]
the overlap in the convolution has length $L-\abs{x}$ when
$\abs{x}\leq L$ and is empty otherwise.  Direct integration gives the
following real even kernels.  Define
\begin{align}\label{eq:An-Bn-test-functions}
  A_{n,L}(x)
  &:=\frac{(L-\abs{x})_+}{L^2}
      \cos(\nu_{n,L}x),
  \\
  B_{n,L}(x)
  &:=-\frac1{2\pi L}
      \sin(\nu_{n,L}\abs{x})\mathbf 1_{[-L,L]}(x),
\end{align}
where $(u)_+:=\max\{u,0\}$.  Then
\begin{equation}\label{eq:Gmn-physical-divided-difference}
  G_{mn,L}(x)
  =
  \begin{cases}
    A_{n,L}(x),&m=n,\\[2mm]
    \dfrac{B_{m,L}(x)-B_{n,L}(x)}{m-n},&m\neq n.
  \end{cases}
\end{equation}
Equivalently, for $m\neq n$,
\begin{equation}\label{eq:Gmn-physical-sine-form}
  G_{mn,L}(x)
  =\frac{
      \sin(\nu_{m,L}\abs{x})
      -\sin(\nu_{n,L}\abs{x})
    }
    {2\pi L(n-m)}\mathbf 1_{[-L,L]}(x).
\end{equation}
In particular,
\begin{equation}\label{eq:Gmn-test-level-commutator}
  (m-n)G_{mn,L}=B_{m,L}-B_{n,L},
  \qquad m\neq n,
\end{equation}
and
\begin{equation}\label{eq:An-Bn-parity}
  A_{-n,L}=A_{n,L},
  \qquad
  B_{-n,L}=-B_{n,L},
  \qquad
  B_{0,L}=0.
\end{equation}
The matrix-valued test function
\begin{equation}\label{eq:G-one-sided-matrix-definition}
  \mathbf G_{N,L}(x)
  :=\bigl(G_{mn,L}(x)\bigr)_{m,n\in I_N}
\end{equation}
is therefore real symmetric and centrosymmetric for every fixed $x$.

\subsection{Explicit transforms and transform-side spectrum}

The one-sided mode has the elementary entire transform
\begin{equation}\label{eq:one-sided-Fourier-transform}
  U_{n,L}(z)
  :=\int_{\mathbb R}f_{n,L}^{+}(x)e^{izx}\dd x
  =\frac{2e^{iLz/2}\sin(Lz/2)}
         {L(z+\nu_{n,L})},
\end{equation}
where the apparent singularity is removable.  The convolution theorem and
\cref{eq:Gmn-one-sided-definition} give
\begin{align}\label{eq:Gmn-transform}
  H_{mn,L}^{G}(z)
  &:=\int_{\mathbb R}G_{mn,L}(x)e^{izx}\dd x
  \\
  &=\frac{2\sin^2(Lz/2)}{L^2}
    \left(
      \frac1{(z+\nu_{m,L})(z+\nu_{n,L})}
      +\frac1{(z-\nu_{m,L})(z-\nu_{n,L})}
    \right)
  \notag\\
  &=\frac{4\sin^2(Lz/2)
          \bigl(z^2+\nu_{m,L}\nu_{n,L}\bigr)}
         {L^2
          \bigl(z^2-\nu_{m,L}^2\bigr)
          \bigl(z^2-\nu_{n,L}^2\bigr)}.
\end{align}
All apparent poles in \cref{eq:Gmn-transform} are removable.  The transforms
of the diagonal and divided-difference generators in
\cref{eq:An-Bn-test-functions} are
\begin{align}\label{eq:An-Bn-transforms}
  H_{n,L}^{A}(z)
  &:=\widehat{A_{n,L}}(z)
   =\frac{4\sin^2(Lz/2)
          \bigl(z^2+\nu_{n,L}^2\bigr)}
         {L^2\bigl(z^2-\nu_{n,L}^2\bigr)^2},
  \\
  H_{n,L}^{B}(z)
  &:=\widehat{B_{n,L}}(z)
   =\frac{4n\sin^2(Lz/2)}
         {L^2\bigl(z^2-\nu_{n,L}^2\bigr)}.
\end{align}
Consequently,
\begin{equation}\label{eq:Gmn-transform-divided-difference}
  H_{mn,L}^{G}(z)
  =
  \begin{cases}
    H_{n,L}^{A}(z),&m=n,\\[2mm]
    \dfrac{H_{m,L}^{B}(z)-H_{n,L}^{B}(z)}{m-n},&m\neq n.
  \end{cases}
\end{equation}
For each fixed $L,m,n$, one has
\begin{equation}\label{eq:Gmn-transform-decay}
  H_{mn,L}^{G}(z)
  =O_{L,m,n}\!\left((1+\abs{\RePart z})^{-2}\right),
  \qquad \abs{\ImPart z}\leq\frac12.
\end{equation}
Hence the zero series associated with $G_{mn,L}$ is absolutely convergent.
The kernels are continuous and piecewise smooth, but they have corners at
$x=0$ and at the ends of their support.  In general,
\begin{equation}\label{eq:Gmn-not-Wloc21}
  G_{mn,L}\notin W_{\mathrm{loc}}^{2,1}(\mathbb R),
\end{equation}
so \cref{cor:parameter-uniformity} does not apply directly; the decay in
\cref{eq:Gmn-transform-decay} supplies the needed zero convergence instead.

For real $t$, introduce the vectors
\begin{equation}\label{eq:Gmn-transform-vectors}
  \mathbf v_L^{\pm}(t)
  :=\left(
      \frac{\sqrt2\sin(Lt/2)}
           {L(t\mathbin{\pm}\nu_{n,L})}
    \right)_{n\in I_N},
\end{equation}
again with removable values understood by continuity.  Then
\begin{equation}\label{eq:Gmn-transform-Gram-form}
  \mathbf H_L^{G}(t)
  :=\bigl(H_{mn,L}^{G}(t)\bigr)_{m,n\in I_N}
  =\mathbf v_L^{+}(t)\bigl(\mathbf v_L^{+}(t)\bigr)^{\mathsf T}
   +\mathbf v_L^{-}(t)\bigl(\mathbf v_L^{-}(t)\bigr)^{\mathsf T}.
\end{equation}
Thus $\mathbf H_L^{G}(t)$ is positive semidefinite and has rank at most two.
Symmetry of $I_N$ gives
$\norm{\mathbf v_L^{+}(t)}_2=\norm{\mathbf v_L^{-}(t)}_2$.  If
\begin{equation}\label{eq:Gmn-transform-d-c}
  d_L(t):=\norm{\mathbf v_L^{+}(t)}_2^2,
  \qquad
  c_L(t):=\mathbf v_L^{+}(t)^{\mathsf T}\mathbf v_L^{-}(t),
\end{equation}
then its two possibly nonzero eigenvalues are
\begin{equation}\label{eq:Gmn-transform-eigenvalues}
  \lambda_{1,2}^{G}(t)=d_L(t)\mathbin{\pm}\abs{c_L(t)}\geq0,
\end{equation}
while the remaining $M-2$ eigenvalues vanish.  Nonnegativity follows also
from $\abs{c_L(t)}\leq d_L(t)$.

\subsection{The divided-difference Weil matrix}

The full explicit-formula functional
$\mathcal Z=\mathcal C_\Xi-\mathcal P$ from
\cref{eq:Xi-explicit-compact} acts entrywise on these kernels.  Define the
full zero-side matrix and its generating sequences by
\begin{align}\label{eq:S-a-b-Weil-definitions}
  S_{mn,L}
  &:=\mathcal Z(G_{mn,L})
    =\sum_{\rho}H_{mn,L}^{G}(z_\rho),
  \\
  a_{n,L}&:=\mathcal Z(A_{n,L}),
  \\
  b_{n,L}&:=\mathcal Z(B_{n,L}).
\end{align}
The sums are absolutely convergent by
\cref{eq:Gmn-transform-decay}.  Applying the linear functional $\mathcal Z$
entrywise to \cref{eq:Gmn-physical-divided-difference} yields
\begin{equation}\label{eq:S-divided-difference-structure}
  S_{mn,L}
  =
  \begin{cases}
    a_{n,L},&m=n,\\[2mm]
    \dfrac{b_{m,L}-b_{n,L}}{m-n},&m\neq n.
  \end{cases}
\end{equation}
In particular,
\begin{equation}\label{eq:S-consistent-relations}
  (m-n)S_{mn,L}=b_{m,L}-b_{n,L},
  \qquad
  b_{m,L}=mS_{m0,L},
  \qquad b_{0,L}=0,
\end{equation}
where the second identity is understood for $m\neq0$ and is trivially
consistent at $m=0$.  Moreover,
\begin{equation}\label{eq:S-a-b-parity}
  a_{-n,L}=a_{n,L},
  \qquad
  b_{-n,L}=-b_{n,L},
  \qquad
  S_{-m,-n,L}=S_{mn,L}.
\end{equation}
Thus
\begin{equation}\label{eq:S-matrix-definition}
  \mathbf S_{N,L}
  :=\bigl(S_{mn,L}\bigr)_{m,n\in I_N}
\end{equation}
is real symmetric and centrosymmetric, and all its eigenvalues are real.

The all-ones matrix enters through a commutator identity rather than through
an entrywise constant off-diagonal formula.  Put
\begin{equation}\label{eq:S-D-B-J-definitions}
  \mathbf D_N:=\operatorname{diag}(n)_{n\in I_N},
  \qquad
  \mathbf B_{N,L}:=\operatorname{diag}(b_{n,L})_{n\in I_N},
  \qquad
  \mathbf J_M:=\mathbf 1_M\mathbf 1_M^{\mathsf T}.
\end{equation}
Then \cref{eq:S-consistent-relations} is equivalent to
\begin{equation}\label{eq:S-rank-two-commutator}
  \mathbf D_N\mathbf S_{N,L}-\mathbf S_{N,L}\mathbf D_N
  =\mathbf B_{N,L}\mathbf J_M-\mathbf J_M\mathbf B_{N,L}
  =\mathbf b_L\mathbf 1_M^{\mathsf T}
   -\mathbf 1_M\mathbf b_L^{\mathsf T},
\end{equation}
where $\mathbf b_L=(b_{n,L})_{n\in I_N}$.  Consequently,
\begin{equation}\label{eq:S-commutator-rank}
  \operatorname{rank}\!
  \left(
    \mathbf D_N\mathbf S_{N,L}-\mathbf S_{N,L}\mathbf D_N
  \right)
  \leq2.
\end{equation}
Thus $\mathbf S_{N,L}$ is a symmetric divided-difference, or Loewner-type,
matrix with separately specified diagonal $(a_{n,L})$ and displacement rank at
most two.

Under the Riemann hypothesis, $z_\rho=\gamma\in\mathbb R$.  Therefore
\cref{eq:Gmn-transform-Gram-form} gives
\begin{equation}\label{eq:S-PSD-under-RH}
  \mathbf S_{N,L}
  =\sum_{\rho}\mathbf H_L^{G}(\gamma)
  \succeq0.
\end{equation}

Here \(\mathbf S_{N,L}\succeq0\) denotes positive semidefiniteness in
the Loewner order; equivalently, for every $\mathbf u\in\mathbb R^{2N+1}$,
\begin{equation}\label{eq:S-quadratic-form-under-RH}
  \mathbf u^{\mathsf T}\mathbf S_{N,L}\mathbf u
  =\sum_{\rho}\left(
      \abs{\mathbf v_L^{+}(\gamma)^{\mathsf T}\mathbf u}^2
      +\abs{\mathbf v_L^{-}(\gamma)^{\mathsf T}\mathbf u}^2
    \right)
  \geq0.
\end{equation}
This does not assert that all individual entries of
\(\mathbf S_{N,L}\) are nonnegative.
Without the Riemann hypothesis the matrix remains real symmetric, but
positive semidefiniteness is not automatic.

\subsection[Explicit Weil formula for the matrix generators]{Explicit Weil formula for \texorpdfstring{$a_{n,L}$, $b_{n,L}$, and $S_{mn,L}$}{a, b, and S}}

The pole pairs obtained from \cref{eq:An-Bn-transforms} are
\begin{align}\label{eq:An-Bn-pole-pairs}
  \mathcal B_{n,L}^{A}
  &:=2H_{n,L}^{A}\!\left(\frac i2\right)
   =\frac{8\sinh^2(L/4)
          \bigl(\tfrac14-\nu_{n,L}^2\bigr)}
         {L^2\bigl(\nu_{n,L}^2+\tfrac14\bigr)^2},
  \\
  \mathcal B_{n,L}^{B}
  &:=2H_{n,L}^{B}\!\left(\frac i2\right)
   =\frac{8n\sinh^2(L/4)}
         {L^2\bigl(\nu_{n,L}^2+\tfrac14\bigr)}.
\end{align}
Since $A_{n,L}(0)=1/L$, $B_{n,L}(0)=0$, and both functions vanish at
$x=\pm L$, the even explicit formula \cref{eq:Xi-Weil-even} gives
\begin{align}\label{eq:an-explicit-Weil-formula}
  a_{n,L}
  ={}&\mathcal B_{n,L}^{A}
      -\frac{\log(2\pi)}{L}
  \notag\\
  &-\frac{2}{L^2}\PF\!\int_0^L
    (L-x)\cos(\nu_{n,L}x)
    \frac{e^{x/2}}{e^x-e^{-x}}\dd x
  \notag\\
  &-\frac{2}{L^2}
    \sum_{2\leq q\leq T}
    \frac{\Lambda(q)}{\sqrt q}
    (L-\log q)
    \cos\!\left(\nu_{n,L}\log q\right),
\end{align}
and
\begin{align}\label{eq:bn-explicit-Weil-formula}
  b_{n,L}
  ={}&\mathcal B_{n,L}^{B}
  \\
  &+\frac1{\pi L}\int_0^L
    \sin(\nu_{n,L}x)
    \frac{e^{x/2}}{e^x-e^{-x}}\dd x
  \notag\\
  &+\frac1{\pi L}
    \sum_{2\leq q\leq T}
    \frac{\Lambda(q)}{\sqrt q}
    \sin\!\left(\nu_{n,L}\log q\right).
  \notag
\end{align}
The integral in \cref{eq:bn-explicit-Weil-formula} is ordinary rather than a
finite part, because the sine factor cancels the singularity at $x=0$.  The
terms $q=T$ vanish in both formulae, even when $T$ is a prime power, so no
endpoint half weight is required.

Combining \cref{eq:S-divided-difference-structure,eq:bn-explicit-Weil-formula}
produces the direct off-diagonal identity
\begin{align}\label{eq:Smn-explicit-Weil-formula}
  S_{mn,L}
  ={}&\frac{\mathcal B_{m,L}^{B}-\mathcal B_{n,L}^{B}}{m-n}
  \notag\\
  &+\frac1{\pi L(m-n)}\int_0^L
    \left(
      \sin(\nu_{m,L}x)-\sin(\nu_{n,L}x)
    \right)
    \frac{e^{x/2}}{e^x-e^{-x}}\dd x
  \notag\\
  &+\frac1{\pi L(m-n)}
    \sum_{2\leq q\leq T}
    \frac{\Lambda(q)}{\sqrt q}
    \left[
      \sin\!\left(\nu_{m,L}\log q\right)
      -\sin\!\left(\nu_{n,L}\log q\right)
    \right],
  \qquad m\neq n.
\end{align}
Together, \cref{eq:an-explicit-Weil-formula,eq:Smn-explicit-Weil-formula}
are the complete entrywise Weil formula for the matrix $\mathbf S_{N,L}$.
They display separately the pole, archimedean, and finite prime-power
contributions and reproduce exactly the relations in
\cref{eq:S-consistent-relations}.

\subsection{Closed evaluation of the archimedean integrals and large-frequency asymptotics}
\label{subsec:closed-archimedean-evaluation}

The archimedean kernel has a geometric expansion that is especially useful
for computation.  For $x>0$, put
\begin{equation}\label{eq:arch-kernel-geometric-expansion}
  K(x)
  :=\frac{e^{x/2}}{e^x-e^{-x}}
  =\frac{e^{-x/2}}{1-e^{-2x}}
  =\sum_{k=0}^{\infty}e^{-\alpha_k x},
  \qquad
  \alpha_k:=2k+\frac12.
\end{equation}
The last equality is the geometric series with ratio $e^{-2x}$.  In
particular, $e^{-x/2}\sum_{k\geq0}(-2x)^k$ is not an expansion of $K(x)$.
The geometric series is not uniform at $x=0$, but the sine factor in the
$b$-integral gives absolute summability because
$\abs{\sin(\nu x)}\leq\abs{\nu}x$.  The $a$-integral is handled by the
finite-part subtraction already fixed in \cref{eq:A-infinity-zeta}.

Set
\begin{equation}\label{eq:arch-special-function-parameters}
  r_L:=e^{-2L},
  \qquad
  z_{n,L}:=\frac14+\frac{i\nu_{n,L}}2.
\end{equation}
Let
\begin{equation}\label{eq:digamma-trigamma-Lerch-definitions}
  \psi_0(z):=\frac{\Gamma'(z)}{\Gamma(z)},
  \qquad
  \psi_1(z):=\psi_0'(z),
  \qquad
  \operatorname{LerchPhi}(r,s,z)
  :=\sum_{k=0}^{\infty}\frac{r^k}{(k+z)^s},
  \quad \abs r<1,
\end{equation}
be the digamma, trigamma, and Lerch transcendent, respectively.

For the ordinary sine integral, termwise integration gives the exact formulas
\begin{align}\label{eq:arch-B-closed-evaluation}
  I_{n,L}^{B}
  &:={}
  \int_0^L\sin(\nu_{n,L}x)K(x)\dd x
  \notag\\
  &={}
  \nu_{n,L}\sum_{k=0}^{\infty}
  \frac{1-e^{-\alpha_kL}}
       {\alpha_k^2+\nu_{n,L}^2}
  \notag\\
  &={}
  \frac12\ImPart\!\left[
    \psi_0(z_{n,L})
    +e^{-L/2}\operatorname{LerchPhi}
       (r_L,1,z_{n,L})
  \right].
\end{align}
This formula also gives $I_{0,L}^{B}=0$.

For the finite-part cosine integral, one first uses
\begin{equation}\label{eq:arch-elementary-A-integral}
  \int_0^L(L-x)e^{-\alpha x}\cos(\nu_{n,L}x)\dd x
  =\frac{L\alpha}{\alpha^2+\nu_{n,L}^2}
   -\frac{(1-e^{-\alpha L})(\alpha^2-\nu_{n,L}^2)}
          {(\alpha^2+\nu_{n,L}^2)^2},
\end{equation}
where $e^{i\nu_{n,L}L}=1$.  The one-sided finite-part normalization in
\cref{eq:A-infinity-zeta} gives
\begin{align}\label{eq:arch-A-closed-evaluation}
  I_{n,L}^{A}
  &:={}
  \PF\!\int_0^L
  (L-x)\cos(\nu_{n,L}x)K(x)\dd x
  \notag\\
  &={}
  \frac L2\left(
    \gamma_{\mathrm E}+\frac\pi2+2\log2
  \right)
  \notag\\
  &\quad+
  \sum_{k=0}^{\infty}\left[
    \frac{L\alpha_k}{\alpha_k^2+\nu_{n,L}^2}
    -\frac{(1-e^{-\alpha_kL})(\alpha_k^2-\nu_{n,L}^2)}
           {(\alpha_k^2+\nu_{n,L}^2)^2}
    -\frac{L}{\alpha_k}
  \right]
  \notag\\
  &={}
  \RePart\!\left[
    -\frac L2\bigl(\psi_0(z_{n,L})+\log2\bigr)
    -\frac14\psi_1(z_{n,L})
    +\frac{e^{-L/2}}4
       \operatorname{LerchPhi}(r_L,2,z_{n,L})
  \right].
\end{align}
The series in the middle line converges after the explicit subtraction
$L/\alpha_k$.  The constant follows from
$\psi_0(1/4)=-\gamma_{\mathrm E}-\pi/2-3\log2$.
Equivalently,
\begin{equation}\label{eq:arch-A-finite-part-partial-sum}
  I_{n,L}^{A}
  =\lim_{J\to\infty}\left\{
    \sum_{k=0}^{J-1}
    \int_0^L(L-x)e^{-\alpha_kx}
       \cos(\nu_{n,L}x)\dd x
    -\frac L2\log(2J)
  \right\}.
\end{equation}

The Lerch correction is exponentially convergent.  If
\begin{equation}\label{eq:arch-truncated-Lerch-definition}
  \operatorname{LerchPhi}_{J}(r,s,z)
  :=\sum_{k=0}^{J-1}\frac{r^k}{(k+z)^s},
  \qquad J\geq1,
\end{equation}
then, because $\RePart z_{n,L}=1/4$,
\begin{equation}\label{eq:arch-Lerch-tail-bound}
  \abs{
    \operatorname{LerchPhi}(r_L,s,z_{n,L})
    -\operatorname{LerchPhi}_{J}(r_L,s,z_{n,L})
  }
  \leq
  \frac{r_L^J}{(J+\tfrac14)^s(1-r_L)},
  \qquad s=1,2.
\end{equation}
Consequently, truncating only the Lerch series after $J$ terms gives the
certified absolute-error bounds
\begin{align}\label{eq:arch-IAB-certified-errors}
  \abs{\Delta I_{n,L}^{B}}
  &\leq
  \frac{e^{-L/2}}2
  \frac{e^{-2LJ}}
       {(J+\tfrac14)(1-e^{-2L})},
  \\
  \abs{\Delta I_{n,L}^{A}}
  &\leq
  \frac{e^{-L/2}}4
  \frac{e^{-2LJ}}
       {(J+\tfrac14)^2(1-e^{-2L})}.
\end{align}
These bounds are uniform in $n$.  Thus, for a prescribed absolute tolerance
$\varepsilon_{\mathrm{arch}}$, the required number of terms is the smallest
$J$ for which the two right-hand sides in
\cref{eq:arch-IAB-certified-errors} do not exceed
$\varepsilon_{\mathrm{arch}}$.

\begin{proposition}[Fixed-$L$ large-frequency asymptotics]
\label{prop:large-frequency-a-b-asymptotics}
Fix $L>0$ and let $n\to+\infty$.  With
$\nu:=\nu_{n,L}=2\pi n/L$, one has
\begin{align}\label{eq:arch-IAB-large-frequency}
  I_{n,L}^{A}
  &=-\frac L2\log\nu
    +\frac{\tfrac L{48}+\tfrac14-K(L)}{\nu^2}
    +O_L(\nu^{-4}),
  \\
  I_{n,L}^{B}
  &=\frac\pi4
    +\frac{\tfrac14-K(L)}{\nu}
    +\frac{\tfrac1{16}+K''(L)}{\nu^3}
    +O_L(\nu^{-5}).
\end{align}
Here
\begin{equation}\label{eq:arch-K-second-derivative}
  K''(L)
  =K(L)\,
   \frac{1+22e^{-2L}+9e^{-4L}}
        {4(1-e^{-2L})^2}.
\end{equation}

Define the finite arithmetic trigonometric polynomials
\begin{align}\label{eq:large-frequency-prime-polynomials}
  \mathcal P_{n,L}^{A}
  &:={}
  \sum_{2\leq q\leq e^L}
  \frac{\Lambda(q)}{\sqrt q}(L-\log q)
  \cos\!\left(\nu_{n,L}\log q\right),
  \\
  \mathcal P_{n,L}^{B}
  &:={}
  \sum_{2\leq q\leq e^L}
  \frac{\Lambda(q)}{\sqrt q}
  \sin\!\left(\nu_{n,L}\log q\right).
\end{align}
Put
\begin{align}\label{eq:large-frequency-kappa-definitions}
  \kappa_A(L)
  &:={}
  8\sinh^2(L/4)+\frac L{24}+\frac12-2K(L),
  \\
  \kappa_{B,1}(L)
  &:={}
  4\sinh^2(L/4)+\frac14-K(L),
  \\
  \kappa_{B,3}(L)
  &:={}
  -\sinh^2(L/4)+\frac1{16}+K''(L).
\end{align}
Then the complete generating sequences satisfy
\begin{align}\label{eq:a-b-large-frequency-asymptotics}
  a_{n,L}
  &={}
  \frac1L\log\!\left(\frac nL\right)
  -\frac{2}{L^2}\mathcal P_{n,L}^{A}
  -\frac{\kappa_A(L)}{L^2\nu^2}
  +O_L(\nu^{-4}),
  \\
  b_{n,L}
  &={}
  \frac{\mathcal P_{n,L}^{B}}{\pi L}
  +\frac1{4L}
  +\frac{\kappa_{B,1}(L)}{\pi L\nu}
  +\frac{\kappa_{B,3}(L)}{\pi L\nu^3}
  +O_L(\nu^{-5}).
\end{align}
For $n\to-\infty$, use
$a_{-n,L}=a_{n,L}$ and $b_{-n,L}=-b_{n,L}$.
In particular,
\begin{equation}\label{eq:a-b-large-frequency-coarse}
  a_{n,L}=\frac1L\log\abs n+O_L(1),
  \qquad
  b_{n,L}=O_L(1).
\end{equation}
The finite sums $\mathcal P_{n,L}^{A,B}$ generally oscillate with $n$ and do
not converge.  Thus \cref{eq:a-b-large-frequency-asymptotics} is an
arithmetic-separated asymptotic expansion: it retains the finite prime-power
trigonometric polynomial exactly and expands only the smooth pole and
archimedean contributions.
\end{proposition}

\begin{proof}
Insert the standard sectorial expansions
\[
  \psi_0(z)
  =\log z-\frac1{2z}-\frac1{12z^2}+O(z^{-4}),
  \qquad
  \psi_1(z)
  =\frac1z+\frac1{2z^2}+\frac1{6z^3}+O(z^{-5})
\]
into \cref{eq:arch-B-closed-evaluation,eq:arch-A-closed-evaluation}.  Expanding
the Lerch terms at large $z_{n,L}$ and using
$e^{-L/2}/(1-e^{-2L})=K(L)$ gives
\cref{eq:arch-IAB-large-frequency}.  Expanding the pole terms in
\cref{eq:An-Bn-pole-pairs} and substituting the two archimedean expansions
into \cref{eq:an-explicit-Weil-formula,eq:bn-explicit-Weil-formula} yields
\cref{eq:a-b-large-frequency-asymptotics}.  The parity statements follow from
\cref{eq:S-a-b-parity}.
\end{proof}

\subsection{Compensated pole--prime formulas for large \texorpdfstring{$L$}{L}}
\label{subsec:pole-prime-compensation}

The pole terms in \cref{eq:An-Bn-pole-pairs} contain
\[
  8\sinh^2(L/4)
  =2\bigl(e^{L/2}+e^{-L/2}-2\bigr),
\]
and therefore can become numerically large when $L$ grows.  The purpose of
this subsection is only to reorganize the exact formulas so that the explicit
pole term is combined analytically with the continuous prime main term before
the final matrix entries are assembled.  No new hypothesis or asymptotic
theorem is introduced.

Put $T=e^L$ and define the continuous prime main terms
\begin{align}\label{eq:continuous-prime-main-terms}
  \mathcal M_{n,L}^{A}
  &:={}
  \int_1^T x^{-1/2}(L-\log x)
    \cos\!\left(\nu_{n,L}\log x\right)\dd x
  \notag\\
  &=
  \bigl(e^{L/2}-1\bigr)
  \frac{\tfrac14-\nu_{n,L}^2}
       {(\nu_{n,L}^2+\tfrac14)^2}
  -\frac{L}{2(\nu_{n,L}^2+\tfrac14)},
  \\
  \mathcal M_{n,L}^{B}
  &:={}
  \int_1^T x^{-1/2}
    \sin\!\left(\nu_{n,L}\log x\right)\dd x
  \notag\\
  &=-
  \frac{\nu_{n,L}(e^{L/2}-1)}
       {\nu_{n,L}^2+\tfrac14}.
\end{align}
Let
\begin{align}\label{eq:prime-sum-remainders}
  \mathcal E_{n,L}^{A}
  &:={}
  \sum_{2\leq q\leq T}
  \frac{\Lambda(q)}{\sqrt q}(L-\log q)
  \cos\!\left(\nu_{n,L}\log q\right)
  -\mathcal M_{n,L}^{A},
  \\
  \mathcal E_{n,L}^{B}
  &:={}
  \sum_{2\leq q\leq T}
  \frac{\Lambda(q)}{\sqrt q}
  \sin\!\left(\nu_{n,L}\log q\right)
  -\mathcal M_{n,L}^{B}.
\end{align}
The compensated pole terms are
\begin{align}\label{eq:compensated-pole-terms}
  \mathcal C_{n,L}^{A}
  &:={}
  \mathcal B_{n,L}^{A}
  -\frac{2}{L^2}\mathcal M_{n,L}^{A}
  \notag\\
  &=
  \frac{2(e^{-L/2}-1)}{L^2}
  \frac{\tfrac14-\nu_{n,L}^2}
       {(\nu_{n,L}^2+\tfrac14)^2}
  +\frac{1}{L(\nu_{n,L}^2+\tfrac14)},
  \\
  \mathcal C_{n,L}^{B}
  &:={}
  \mathcal B_{n,L}^{B}
  +\frac{1}{\pi L}\mathcal M_{n,L}^{B}
  \notag\\
  &=
  \frac{\nu_{n,L}(e^{-L/2}-1)}
       {\pi L(\nu_{n,L}^2+\tfrac14)}.
\end{align}
Consequently, the exact formulas
\cref{eq:an-explicit-Weil-formula,eq:bn-explicit-Weil-formula} may be evaluated
in the equivalent compensated form
\begin{align}\label{eq:an-bn-compensated-Weil-formula}
  a_{n,L}
  ={}&
  \mathcal C_{n,L}^{A}
  -\frac{\log(2\pi)}{L}
  -\frac{2}{L^2}I_{n,L}^{A}
  -\frac{2}{L^2}\mathcal E_{n,L}^{A},
  \\
  b_{n,L}
  ={}&
  \mathcal C_{n,L}^{B}
  +\frac1{\pi L}I_{n,L}^{B}
  +\frac1{\pi L}\mathcal E_{n,L}^{B},
\end{align}
where $I_{n,L}^{A,B}$ are given exactly by
\cref{eq:arch-B-closed-evaluation,eq:arch-A-closed-evaluation}.  For $m\ne n$,
the corresponding off-diagonal entry is
\begin{align}\label{eq:Smn-compensated-Weil-formula}
  S_{mn,L}
  ={}&
  \frac{\mathcal C_{m,L}^{B}-\mathcal C_{n,L}^{B}}{m-n}
  +\frac{I_{m,L}^{B}-I_{n,L}^{B}}{\pi L(m-n)}
  \notag\\
  &+\frac{\mathcal E_{m,L}^{B}-\mathcal E_{n,L}^{B}}
          {\pi L(m-n)}.
\end{align}

\begin{remark}[Purpose of the compensated form]
\label{rem:purpose-pole-prime-compensation}
The original formulas display the pole and prime-power contributions
separately and are useful for conceptual checks.  The compensated formulas
remove the explicit pole--main-term cancellation from the final assembly.
Their purpose is numerical stability and transparent diagnostics when $L$ is
large; they do not change $a_{n,L}$, $b_{n,L}$, $S_{mn,L}$, or the zero-side
matrix.  The remainder $\mathcal E_{n,L}^{A,B}$ is still a difference between
a prime-power sum and its continuous main term, so sufficient working
precision remains necessary.  For moderate $L$, both evaluations may be
retained as an independent cross-check.
\end{remark}

\subsection{A single-parameter finite-zero approximation}

Although the kernels in this subsection do not lie in
$W_{\mathrm{loc}}^{2,1}(\mathbb R)$, their explicit transforms permit a
direct moving-family estimate.  Define
\begin{equation}\label{eq:S-finite-zero-matrix}
  \mathbf S_{N,L}^{(T)}
  :=\left(
      \sum_{\substack{\rho=\beta+i\gamma\\\abs{\gamma}<T}}
      H_{mn,L}^{G}(z_\rho)
    \right)_{m,n\in I_N},
  \qquad L=\log T.
\end{equation}
There is an absolute constant $C>0$ such that, uniformly for
$\abs{m},\abs{n}\leq N\leq\lfloor L\rfloor$,
\begin{equation}\label{eq:S-entrywise-zero-tail}
  \max_{m,n\in I_N}
  \abs{
    \left(\mathbf S_{N,L}^{(T)}-\mathbf S_{N,L}\right)_{mn}
  }
  \leq
  C\frac{\log(T+2)+1}{\sqrt T\,L^2}.
\end{equation}
Indeed, $\abs{\nu_{n,L}}\leq2\pi$ and
\cref{eq:Gmn-transform} give, for $\abs{\gamma}\geq T$ and
$\abs{\ImPart z_\rho}\leq1/2$,
\[
  \abs{H_{mn,L}^{G}(z_\rho)}
  \leq C\frac{T^{1/2}}{L^2\gamma^2},
\]
first for all sufficiently large $T$ and then, after increasing $C$, for
$T\geq3$.  The zero-counting estimate used in
\cref{eq:zero-tail-bound} proves \cref{eq:S-entrywise-zero-tail}.  Since
$M=2N+1\leq2L+1$,
\begin{equation}\label{eq:S-operator-zero-tail}
  \norm{
    \mathbf S_{N,L}^{(T)}-\mathbf S_{N,L}
  }_{\mathrm{op}}
  \leq\frac{C}{\sqrt T}.
\end{equation}
Consequently, Weyl's inequality yields
\begin{equation}\label{eq:S-eigenvalue-zero-tail}
  \max_{1\leq j\leq M}
  \abs{
    \lambda_j\!\left(\mathbf S_{N,L}^{(T)}\right)
    -\lambda_j\!\left(\mathbf S_{N,L}\right)
  }
  \leq\frac{C}{\sqrt T},
\end{equation}
when the real eigenvalues are arranged in nondecreasing order.  Under the
Riemann hypothesis, both matrices in \cref{eq:S-operator-zero-tail} are
positive semidefinite.

\section{Scale-invariant quotient by a Hermitian definite pencil}
\label{sec:scale-invariant-quotient}

The divided-difference identity in
\cref{eq:S-divided-difference-structure,eq:S-rank-two-commutator}
contains more information than is needed for Lemke's oblique projection
\cite{Lemke2026}.  It permits the quotient problem to be represented directly on a fixed
codimension-one contrast space by a symmetric generalized eigenvalue problem.
This formulation is homogeneous in the Weil metric, invariant under positive
affine rescaling of the Weil matrix, and does not require division by a small
ground-state overlap.

\subsection{Frequency operator, evaluation vector, and displacement identity}
\label{subsec:siq-displacement}

Put
\begin{equation}\label{eq:siq-frequency-evaluation}
  M:=2N+1,
  \qquad
  \mathbf D_{L,N}
  :=\operatorname{diag}(\nu_{n,L})_{n\in I_N},
  \qquad
  \boldsymbol\delta_{N,L}
  :=\frac1{\sqrt L}\mathbf1_M.
\end{equation}
Thus $\boldsymbol\delta_{N,L}$ is the coordinate vector of evaluation at
$x=0$ in the orthonormal Fourier basis
$L^{-1/2}e^{i\nu_{n,L}x}$.  Define also
\begin{equation}\label{eq:siq-beta-alpha}
  \boldsymbol\beta_{N,L}
  :=\sqrt L\,(b_{n,L})_{n\in I_N},
  \qquad
  \alpha_L:=\frac{2\pi}{L}.
\end{equation}
Then the rank-two commutator takes the normalized form
\begin{equation}\label{eq:siq-displacement-identity}
  [\mathbf D_{L,N},\mathbf S_{N,L}]
  =\alpha_L\left(
      \boldsymbol\beta_{N,L}\boldsymbol\delta_{N,L}^{\mathsf T}
      -\boldsymbol\delta_{N,L}\boldsymbol\beta_{N,L}^{\mathsf T}
    \right).
\end{equation}
The same identity holds with $\mathbf S_{N,L}$ replaced by
$\mathbf S_{N,L}-c\mathbf I_M$ for any scalar $c$.

Let
\begin{equation}\label{eq:siq-contrast-space}
  \mathcal V_{N,L}^{0}
  :=\ker\boldsymbol\delta_{N,L}^{\mathsf T}
  =\left\{\mathbf x\in\mathbb C^M:
      \sum_{n\in I_N}x_n=0\right\}.
\end{equation}
Choose an $M\times2N$ matrix $\mathbf C_{N,L}$ with Euclidean-orthonormal
columns spanning this space:
\begin{equation}\label{eq:siq-contrast-matrix}
  \mathbf C_{N,L}^{*}\mathbf C_{N,L}=\mathbf I_{2N},
  \qquad
  \mathbf C_{N,L}^{*}\boldsymbol\delta_{N,L}=0,
  \qquad
  \operatorname{ran}\mathbf C_{N,L}=\mathcal V_{N,L}^{0}.
\end{equation}
A Helmert contrast matrix is one explicit choice.  Any other basis of
$\mathcal V_{N,L}^{0}$ gives a congruent matrix pencil and therefore the same
generalized spectrum.

\subsection{The shifted contrast pencil}
\label{subsec:siq-shifted-pencil}

Arrange the eigenvalues of $\mathbf S_{N,L}$ as
\begin{equation}\label{eq:siq-ordered-eigenvalues}
  \lambda_1(N,L)\leq\lambda_2(N,L)\leq\cdots\leq\lambda_M(N,L),
\end{equation}
and put
\begin{equation}\label{eq:siq-shifted-metric}
  \epsilon_{N,L}:=\lambda_1(N,L),
  \qquad
  \mathbf W_{N,L}:=\mathbf S_{N,L}-\epsilon_{N,L}\mathbf I_M.
\end{equation}
Then $\mathbf W_{N,L}\succeq0$.  Define the two $2N\times2N$ compressed
matrices
\begin{align}\label{eq:siq-GK-definition}
  \mathbf G_{N,L}
  &:=\mathbf C_{N,L}^{*}\mathbf W_{N,L}\mathbf C_{N,L},
  \\
  \mathbf K_{N,L}
  &:=\mathbf C_{N,L}^{*}\mathbf W_{N,L}
      \mathbf D_{L,N}\mathbf C_{N,L}.
\end{align}
The first is the quotient Weil metric in contrast coordinates.  The second
is the corresponding differentiation form.

\begin{theorem}[Scale-invariant quotient pencil]
\label{thm:scale-invariant-quotient}
Assume that $\epsilon_{N,L}$ is a simple eigenvalue and that
\begin{equation}\label{eq:siq-positive-contrast-metric}
  \mathbf G_{N,L}\succ0.
\end{equation}
Then the following statements hold.

\begin{enumerate}[label=\textup{(\roman*)}]
  \item The matrix $\mathbf K_{N,L}$ is Hermitian; in the present real
        setting it is real symmetric.

  \item There is a unique operator
        $\mathcal T_{N,L}:\mathcal V_{N,L}^{0}\to\mathcal V_{N,L}^{0}$
        satisfying
        \begin{equation}\label{eq:siq-Riesz-definition}
          \mathbf x^{*}\mathbf W_{N,L}\mathcal T_{N,L}\mathbf y
          =\mathbf x^{*}\mathbf W_{N,L}
             \mathbf D_{L,N}\mathbf y
          \qquad
          (\mathbf x,\mathbf y\in\mathcal V_{N,L}^{0}).
        \end{equation}
        Its coordinate matrix in the columns of $\mathbf C_{N,L}$ is
        \begin{equation}\label{eq:siq-coordinate-operator}
          \mathbf A_{N,L}:=\mathbf G_{N,L}^{-1}\mathbf K_{N,L}.
        \end{equation}

  \item The operator $\mathcal T_{N,L}$ is self-adjoint for the positive
        inner product induced by $\mathbf W_{N,L}$.  Equivalently, its
        eigenvalues are the generalized eigenvalues of the Hermitian definite
        pencil
        \begin{equation}\label{eq:siq-generalized-eigenproblem}
          \boxed{
          \mathbf K_{N,L}\mathbf y
          =\mu\,\mathbf G_{N,L}\mathbf y,}
        \end{equation}
        and are all real.

  \item If
        $\mathbf G_{N,L}=(\mathbf L_{N,L}^{G})^{*}\mathbf L_{N,L}^{G}$
        is a Cholesky factorization, then
        \begin{equation}\label{eq:siq-whitened-operator}
          \widehat{\mathbf A}_{N,L}
          :=(\mathbf L_{N,L}^{G})^{-*}\mathbf K_{N,L}
            (\mathbf L_{N,L}^{G})^{-1}
        \end{equation}
        is Hermitian and has exactly the generalized spectrum in
        \cref{eq:siq-generalized-eigenproblem}.
\end{enumerate}
\end{theorem}

\begin{proof}
Taking the adjoint in the second line of \cref{eq:siq-GK-definition} gives
\[
  \mathbf K_{N,L}^{*}
  =\mathbf C_{N,L}^{*}\mathbf D_{L,N}
    \mathbf W_{N,L}\mathbf C_{N,L}.
\]
By \cref{eq:siq-displacement-identity}, the difference
$\mathbf K_{N,L}-\mathbf K_{N,L}^{*}$ is a linear combination of terms
containing either
$\mathbf C_{N,L}^{*}\boldsymbol\delta_{N,L}$ or
$\boldsymbol\delta_{N,L}^{\mathsf T}\mathbf C_{N,L}$, and hence vanishes.
This proves (i).

The positive-definite matrix $\mathbf G_{N,L}$ identifies the dual of
$\mathcal V_{N,L}^{0}$ with the space itself.  Thus
\cref{eq:siq-Riesz-definition} determines a unique operator, and its matrix
satisfies $\mathbf G_{N,L}\mathbf A_{N,L}=\mathbf K_{N,L}$, proving (ii).
The Hermiticity of $\mathbf K_{N,L}$ gives
\[
  \langle\mathbf x,\mathcal T_{N,L}\mathbf y\rangle_W
  =\langle\mathcal T_{N,L}\mathbf x,\mathbf y\rangle_W,
\]
which proves (iii).  Conjugating the pencil by the Cholesky factor gives
\cref{eq:siq-whitened-operator}, proving (iv).
\end{proof}

\begin{remark}[Exact quotient interpretation]
\label{rem:siq-exact-quotient}
Let $\mathbf e_{N,L}$ span $\ker\mathbf W_{N,L}$.  Under the hypotheses of
\cref{thm:scale-invariant-quotient},
\begin{equation}\label{eq:siq-overlap-equivalence}
  \boldsymbol\delta_{N,L}^{*}\mathbf e_{N,L}\neq0,
\end{equation}
because \cref{eq:siq-positive-contrast-metric} is equivalent to
$\ker\mathbf W_{N,L}\cap\mathcal V_{N,L}^{0}=\{0\}$.  Consequently each
class in
$\mathbb C^M/\operatorname{span}\{\mathbf e_{N,L}\}$ has a unique
representative in $\mathcal V_{N,L}^{0}$.  For
$\mathbf y\in\mathcal V_{N,L}^{0}$, the vector
$\mathcal T_{N,L}\mathbf y$ is precisely the unique zero-mean representative
satisfying
\begin{equation}\label{eq:siq-geometric-quotient-condition}
  \mathbf D_{L,N}\mathbf y-\mathcal T_{N,L}\mathbf y
  \in\operatorname{span}\{\mathbf e_{N,L}\}.
\end{equation}
Thus the pencil gives the same quotient operator as a projection along the
null direction, but computes it without constructing that projection or
normalizing the ground vector.
\end{remark}

\subsection{Affine-scale invariance and stable deflation}
\label{subsec:siq-scale-invariance}

The quotient pencil is invariant under the transformation
\begin{equation}\label{eq:siq-affine-rescaling}
  \mathbf S_{N,L}\longmapsto
  a\mathbf S_{N,L}+b\mathbf I_M,
  \qquad a>0,
\end{equation}
because
\[
  \epsilon_{N,L}\longmapsto a\epsilon_{N,L}+b,
  \qquad
  \mathbf W_{N,L}\longmapsto a\mathbf W_{N,L},
\]
and both $\mathbf G_{N,L}$ and $\mathbf K_{N,L}$ are multiplied by the same
positive scalar $a$.  In particular, neither the generalized eigenvalues nor
the whitened matrix in \cref{eq:siq-whitened-operator} changes.

For numerical work, let
\begin{equation}\label{eq:siq-gap-normalization}
  g_{N,L}:=\lambda_2(N,L)-\lambda_1(N,L)>0,
  \qquad
  \widetilde{\mathbf W}_{N,L}
  :=\frac{\mathbf S_{N,L}-\lambda_1(N,L)\mathbf I_M}{g_{N,L}}.
\end{equation}
Its smallest positive eigenvalue is one.  Replacing
$\mathbf W_{N,L}$ by $\widetilde{\mathbf W}_{N,L}$ leaves the pencil
spectrum unchanged and makes the normalization invariant under every
positive affine rescaling in \cref{eq:siq-affine-rescaling}.  When the least
eigenvalue is extremely small, one may form the compressed matrices directly:
\begin{align}\label{eq:siq-compressed-without-full-W}
  \widetilde{\mathbf G}_{N,L}
  &=\frac{
      \mathbf C_{N,L}^{*}\mathbf S_{N,L}\mathbf C_{N,L}
      -\lambda_1(N,L)\mathbf I_{2N}}
      {g_{N,L}},
  \\
  \widetilde{\mathbf K}_{N,L}
  &=\frac{
      \mathbf C_{N,L}^{*}\mathbf S_{N,L}\mathbf D_{L,N}
        \mathbf C_{N,L}
      -\lambda_1(N,L)
       \mathbf C_{N,L}^{*}\mathbf D_{L,N}\mathbf C_{N,L}}
      {g_{N,L}}.
\end{align}
The second matrix should be symmetrized at working precision.  These formulas
avoid subtracting two nearly equal full matrices before compression.

\begin{remark}[Advantages over Lemke's ground-state quotient]
\label{rem:siq-advantages}
Lemke's construction \cite{Lemke2026} first computes a ground eigenvector
$\mathbf e$, normalizes it by $\boldsymbol\delta^{*}\mathbf e$, forms an
oblique projection along $\mathbf e$, and finally multiplies the frequency
matrix by that projection.  The present formulation gives the same quotient
operator in fixed contrast coordinates, but has four concrete computational
advantages.  It uses only the scalar least eigenvalue and a fixed contrast
basis; it never divides by $\boldsymbol\delta^{*}\mathbf e$; it never forms a
large-norm oblique projector; and it reduces the calculation to a Hermitian
definite pencil for which stable Cholesky or generalized-eigensystem routines
are standard.  It also makes the affine-scale invariance in
\cref{eq:siq-affine-rescaling} exact rather than heuristic.

This is a change of realization, not a claim that Lemke's finite quotient was
spectrally incorrect.  The two finite spectra are algebraically identical
under the hypotheses of \cref{thm:scale-invariant-quotient}; the improvement is
that avoidable coordinate amplification is removed.  Genuine geometric
degeneration of the quotient metric remains and is quantified in
\cref{prop:siq-intrinsic-conditioning} below.
\end{remark}

\subsection{Explicit matrix elements in a difference basis}
\label{subsec:siq-difference-basis}

For formulas, it is convenient to use the nonorthonormal basis
\begin{equation}\label{eq:siq-difference-basis}
  \mathbf q_n:=\mathbf u_n-\mathbf u_0,
  \qquad n\in I_N^{\times}:=I_N\setminus\{0\},
\end{equation}
where $\mathbf u_n$ is the standard coordinate vector at the node
$\nu_{n,L}$.  These vectors span $\mathcal V_{N,L}^{0}$.  Write
$\mathbf G_{N,L}^{\Delta}$ and $\mathbf K_{N,L}^{\Delta}$ for the two
matrices in this basis.  For $m,n\in I_N^{\times}$,
\begin{align}\label{eq:siq-difference-G-entry}
  (\mathbf G_{N,L}^{\Delta})_{mn}
  ={}&S_{mn,L}-S_{m0,L}-S_{0n,L}+S_{00,L}
      -\epsilon_{N,L}(\delta_{mn}+1),
  \\
  (\mathbf K_{N,L}^{\Delta})_{mn}
  ={}&\nu_{n,L}
      \left(S_{mn,L}-S_{0n,L}
            -\epsilon_{N,L}\delta_{mn}\right).
  \label{eq:siq-difference-K-entry}
\end{align}
The symmetry of \cref{eq:siq-difference-K-entry} is a nontrivial consequence
of the divided-difference identity.  Whenever
$\mathbf G_{N,L}^{\Delta}$ is invertible, the explicit quotient coordinate
matrix is
\begin{equation}\label{eq:siq-difference-coordinate-matrix}
  \mathbf A_{N,L}^{\Delta}
  :=(\mathbf G_{N,L}^{\Delta})^{-1}
     \mathbf K_{N,L}^{\Delta}.
\end{equation}
For computation, the generalized symmetric problem should be solved directly
rather than by explicitly forming the inverse in
\cref{eq:siq-difference-coordinate-matrix}.

At the level of one transform-side atom, subtraction of the zero row and
column gives
\begin{align}\label{eq:siq-atom-difference-G}
  \overline H_{mn,L}^{G}(z)
  &:={H}_{mn,L}^{G}(z)-{H}_{m0,L}^{G}(z)
       -{H}_{0n,L}^{G}(z)+{H}_{00,L}^{G}(z)
  \\
  &=\frac{\nu_{m,L}\nu_{n,L}}{z^2}H_{mn,L}^{G}(z)
  \notag\\
  &=\frac{
      4\sin^2(Lz/2)\nu_{m,L}\nu_{n,L}
      (z^2+\nu_{m,L}\nu_{n,L})}
     {L^2z^2
      (z^2-\nu_{m,L}^2)(z^2-\nu_{n,L}^2)},
\end{align}
with all apparent singularities interpreted by continuity.  Similarly,
\begin{equation}\label{eq:siq-atom-difference-K}
  \overline K_{mn,L}^{G}(z)
  =\frac{
      4\sin^2(Lz/2)\nu_{m,L}\nu_{n,L}
      (\nu_{m,L}+\nu_{n,L})}
     {L^2
      (z^2-\nu_{m,L}^2)(z^2-\nu_{n,L}^2)}.
\end{equation}
These formulas give explicit zero-side entries for the quotient metric and
its differentiation form.

\subsection{Reflection reduction and the positive square}
\label{subsec:siq-parity-reduction}

Let $\mathbf R_N\mathbf u_n=\mathbf u_{-n}$.  By
\cref{eq:S-a-b-parity},
\begin{equation}\label{eq:siq-reflection-relations}
  \mathbf R_N\mathbf W_{N,L}\mathbf R_N=\mathbf W_{N,L},
  \qquad
  \mathbf R_N\mathbf D_{L,N}\mathbf R_N=-\mathbf D_{L,N},
  \qquad
  \mathbf R_N\boldsymbol\delta_{N,L}=\boldsymbol\delta_{N,L}.
\end{equation}
Hence $\mathcal V_{N,L}^{0}$ splits into even and odd subspaces, each of
dimension $N$.  Choose Euclidean-orthonormal contrast matrices
$\mathbf C_{N,L}^{+}$ and $\mathbf C_{N,L}^{-}$ for these two subspaces and
put
\begin{align}\label{eq:siq-parity-G-B}
  \mathbf G_{N,L}^{\pm}
  &:= (\mathbf C_{N,L}^{\pm})^{*}
      \mathbf W_{N,L}\mathbf C_{N,L}^{\pm},
  \\
  \mathbf B_{N,L}
  &:= (\mathbf C_{N,L}^{-})^{*}
      \mathbf W_{N,L}\mathbf D_{L,N}
      \mathbf C_{N,L}^{+}.
\end{align}
In the parity-adapted basis, the pencil has the form
\begin{equation}\label{eq:siq-parity-pencil-block}
  \mathbf G_{N,L}
  =\begin{pmatrix}
      \mathbf G_{N,L}^{+}&0\\
      0&\mathbf G_{N,L}^{-}
    \end{pmatrix},
  \qquad
  \mathbf K_{N,L}
  =\begin{pmatrix}
      0&\mathbf B_{N,L}^{*}\\
      \mathbf B_{N,L}&0
    \end{pmatrix}.
\end{equation}
Let $\mathbf L_{\pm}$ be Cholesky factors of
$\mathbf G_{N,L}^{\pm}$ and define
\begin{equation}\label{eq:siq-parity-whitened-block}
  \mathbf A_{N,L}^{+-}
  :=\mathbf L_-^{-*}\mathbf B_{N,L}\mathbf L_+^{-1}.
\end{equation}
Then the whitened quotient operator is
\begin{equation}\label{eq:siq-whitened-parity-operator}
  \widehat{\mathbf A}_{N,L}
  =\begin{pmatrix}
      0&(\mathbf A_{N,L}^{+-})^{*}\\
      \mathbf A_{N,L}^{+-}&0
    \end{pmatrix}.
\end{equation}
Its nonzero eigenvalues occur in opposite-sign pairs and are the signed
singular values of $\mathbf A_{N,L}^{+-}$.  The positive-parity squared
operator is therefore
\begin{equation}\label{eq:siq-positive-square}
  \boxed{
  \mathbf S_{N,L}^{+,\mathrm{siq}}
  :=(\mathbf A_{N,L}^{+-})^{*}\mathbf A_{N,L}^{+-}\succeq0.}
\end{equation}
If $\mathbf A_{N,L}^{+-}$ is invertible and its singular values are
$0<\mu_{1,N,L}\leq\cdots\leq\mu_{N,N,L}$, then
\begin{align}\label{eq:siq-sign-paired-spectrum}
  \operatorname{spec}(\widehat{\mathbf A}_{N,L})
  &={}\{-\mu_{N,N,L},\ldots,-\mu_{1,N,L},
         \mu_{1,N,L},\ldots,\mu_{N,N,L}\},
  \\
  \operatorname{spec}(\mathbf S_{N,L}^{+,\mathrm{siq}})
  &={}\{\mu_{1,N,L}^2,\ldots,\mu_{N,N,L}^2\}.
\end{align}
No additional removal of the state $\mu_{1,N,L}^2$ is part of the quotient
construction.

\subsection{Intrinsic conditioning}
\label{subsec:siq-conditioning}

The contrast pencil avoids a large oblique projector, but the geometry of a
nearly tangent null direction cannot be removed by a change of coordinates.
The following estimate makes the distinction precise.

\begin{proposition}[Projector growth versus metric degeneration]
\label{prop:siq-intrinsic-conditioning}
Assume that $\epsilon_{N,L}$ is simple.  Let
$\widehat{\mathbf e}_{N,L}$ be a Euclidean unit vector spanning
$\ker\mathbf W_{N,L}$, let
\begin{equation}\label{eq:siq-angle-parameter}
  \widehat{\boldsymbol\delta}_{N,L}
  :=\frac{\boldsymbol\delta_{N,L}}
          {\norm{\boldsymbol\delta_{N,L}}_2},
  \qquad
  q_{N,L}
  :=\abs{
      \widehat{\boldsymbol\delta}_{N,L}^{*}
      \widehat{\mathbf e}_{N,L}},
\end{equation}
and let
\begin{equation}\label{eq:siq-positive-gap-and-top}
  g_{N,L}:=\lambda_2(N,L)-\lambda_1(N,L),
  \qquad
  M_{N,L}:=\norm{\mathbf W_{N,L}}_{\mathrm{op}}.
\end{equation}
Then
\begin{equation}\label{eq:siq-conditioning-bounds}
  g_{N,L}q_{N,L}^{2}
  \leq\lambda_{\min}(\mathbf G_{N,L})
  \leq M_{N,L}q_{N,L}^{2}.
\end{equation}
The norm of the corresponding oblique projection along
$\widehat{\mathbf e}_{N,L}$ onto
$\ker\widehat{\boldsymbol\delta}_{N,L}^{*}$ would be exactly
\begin{equation}\label{eq:siq-oblique-norm-comparison}
  q_{N,L}^{-1}.
\end{equation}
\end{proposition}

\begin{proof}
For a unit vector $\mathbf x\in\mathcal V_{N,L}^{0}$, write
$\mathbf x=a\widehat{\mathbf e}_{N,L}+\mathbf y$ with
$\mathbf y\perp\widehat{\mathbf e}_{N,L}$.  The angle between
$\widehat{\mathbf e}_{N,L}$ and $\mathcal V_{N,L}^{0}$ gives
$\norm{\mathbf y}_2^2\geq q_{N,L}^2$.  Hence
$\mathbf x^{*}\mathbf W_{N,L}\mathbf x
 \geq g_{N,L}q_{N,L}^2$.
For the reverse bound, take the normalized orthogonal projection of
$\widehat{\mathbf e}_{N,L}$ onto $\mathcal V_{N,L}^{0}$; its component
orthogonal to $\widehat{\mathbf e}_{N,L}$ has norm $q_{N,L}$.  The Rayleigh
quotient is therefore at most $M_{N,L}q_{N,L}^2$.  Formula
\cref{eq:siq-oblique-norm-comparison} is the standard norm formula for a
rank-one oblique projection.
\end{proof}

Thus a small overlap no longer appears as an explicit divisor in the
algorithm; instead it appears honestly as a small eigenvalue of the quotient
metric.  The latter is visible through
$\kappa(\mathbf G_{N,L})$ and can be handled by sufficient working precision,
preconditioning, and relative perturbation estimates.

Lemke's ground-vector calculation makes the numerical advantage visible \cite{Lemke2026}.
For $N=L$ let $\widehat{\mathbf e}$ be the Euclidean unit least eigenvector
and put
$ q=|\widehat{\boldsymbol\delta}^{*}\widehat{\mathbf e}|$.
The exact norm of Lemke's oblique projector is $q^{-1}$.  The arithmetic
data give:
\begin{center}
\small
\textbf{Numerical diagnostic: vanishing least eigenvalue and growth of the
Lemke oblique projector.}\par\smallskip
\begin{tabular}{c|c|c|c|c}
$N=L$ & $\lambda_1$ & $\lambda_2-\lambda_1$ &
$|\boldsymbol\delta^{*}\widehat{\mathbf e}|$ & $q^{-1}$\\ \hline
$2$  & $1.1682\times10^{-8}$  & $2.1418\times10^{-6}$  & $4.5305\times10^{-4}$  & $3.4900\times10^{3}$\\
$5$  & $2.6600\times10^{-24}$ & $2.8724\times10^{-21}$ & $1.1438\times10^{-11}$ & $1.2968\times10^{11}$\\
$7$  & $2.0168\times10^{-35}$ & $5.7728\times10^{-32}$ & $3.8868\times10^{-17}$ & $3.7662\times10^{16}$\\
$11$ & $1.0679\times10^{-58}$ & $7.7043\times10^{-55}$ & $1.3366\times10^{-28}$ & $1.0818\times10^{28}$\\
$13$ & $7.2193\times10^{-71}$ & $5.7976\times10^{-67}$ & $1.0838\times10^{-34}$ & $1.3297\times10^{34}$
\end{tabular}
\end{center}
The first three numerical columns describe the same arithmetic matrix used by
the new method.  The last column is not a condition imposed by the contrast
pencil; it records the avoidable amplification that occurs in Lemke's explicit
oblique-projection coordinates.  By
\cref{rem:siq-exact-quotient}, the generalized eigenvalues of the contrast
pencil are algebraically identical to those of Lemke's quotient
representation; only the coordinates and numerical conditioning have changed.

\subsection{The \texorpdfstring{$N=L=13$}{N=L=13} comparison}
\label{subsec:siq-N13-Lemke-contrast-comparison}

The accompanying Wolfram Language program
\path{weil_v11_N13_Lemke_vs_scale_invariant_quotient.wl} retains Lemke's
construction as a benchmark and then recomputes the same quotient through the
scale-invariant
contrast pencil.  The arithmetic input is identical in both calculations:
$N=L=13$, $T=e^{13}$, $37308$ prime powers, and $240$-digit working
precision.  An independent reproduction with $180$-digit arithmetic gave the
comparison in \cref{tab:N13-Lemke-contrast-comparison}.

\begin{table}[H]
\centering
\caption{First three positive quotient eigenvalues at $N=L=13$.  The final two
columns use the unrounded values.}
\label{tab:N13-Lemke-contrast-comparison}
\resizebox{\textwidth}{!}{%
\begin{tabular}{c|c|c|c|c|c}
$j$ & $\gamma_j$ & Lemke quotient $\mu_j^{\mathrm L}$ &
contrast pencil $\mu_j^{\mathrm C}$ &
$|\mu_j^{\mathrm L}-\mu_j^{\mathrm C}|$ &
$|\mu_j^{\mathrm C}-\gamma_j|$ \\ \hline
$1$ & $14.134725141734693790$ & $14.134725141734693790$ &
$14.134725141734693790$ & $1.66\times10^{-73}$ & $1.54\times10^{-27}$ \\
$2$ & $21.022039638771554993$ & $21.022039638771555145$ &
$21.022039638771555145$ & $8.27\times10^{-68}$ & $1.52\times10^{-16}$ \\
$3$ & $25.010857580145688763$ & $25.010857580162843125$ &
$25.010857580162843125$ & $5.23\times10^{-63}$ & $1.72\times10^{-11}$
\end{tabular}}
\end{table}

The two quotient realizations therefore agree to at least sixty-two decimal
places in this calculation.  Their common discrepancy from the exact zeta
ordinates is inherited from the finite arithmetic matrix, not from the choice
of quotient coordinates.  The numerical advantage of the contrast method is
that this agreement is obtained without forming the projector whose estimated
norm is $1.33\times10^{34}$ at these parameters.  Thus the experiment
separates two issues cleanly: Lemke's quotient mechanism is spectrally sound at
finite dimension, while the contrast pencil is a more stable realization of
that mechanism.  Convergence to all zeta ordinates remains the independent
prime-to-zero transfer problem.

\subsection{A ground-state-free unshifted regularization}
\label{subsec:siq-unshifted-pencil}

There is also a useful finite arithmetic regularization that avoids even the
least-eigenvalue subtraction.  Define
\begin{align}\label{eq:siq-unshifted-GK}
  \mathbf G_{N,L}^{\mathrm{un}}
  &:=\mathbf C_{N,L}^{*}\mathbf S_{N,L}\mathbf C_{N,L},
  \\
  \mathbf K_{N,L}^{\mathrm{un}}
  &:=\mathbf C_{N,L}^{*}\mathbf S_{N,L}
       \mathbf D_{L,N}\mathbf C_{N,L}.
\end{align}
The displacement identity again implies that
$\mathbf K_{N,L}^{\mathrm{un}}$ is Hermitian.  If
$\mathbf G_{N,L}^{\mathrm{un}}\succ0$, then
\begin{equation}\label{eq:siq-unshifted-generalized-problem}
  \mathbf K_{N,L}^{\mathrm{un}}\mathbf y
  =\mu\,\mathbf G_{N,L}^{\mathrm{un}}\mathbf y
\end{equation}
is a Hermitian definite pencil with real spectrum.  It requires neither a
ground eigenvector nor $\lambda_1(N,L)$.  When the metric has an exact null
state of energy zero, as in the dimension-matched critical-line zero-side
model, \cref{eq:siq-unshifted-GK} coincides with the shifted quotient pencil.
For the full arithmetic matrix, it is a different finite-$N$ regularization,
with
\begin{align}\label{eq:siq-shifted-unshifted-difference}
  \mathbf G_{N,L}
  &=\mathbf G_{N,L}^{\mathrm{un}}
    -\epsilon_{N,L}\mathbf I_{2N},
  \\
  \mathbf K_{N,L}
  &=\mathbf K_{N,L}^{\mathrm{un}}
    -\epsilon_{N,L}
      \mathbf C_{N,L}^{*}\mathbf D_{L,N}\mathbf C_{N,L}.
\end{align}
A proof that the two pencils have the same low-energy limit would follow from
relative estimates strong enough to make the perturbations in
\cref{eq:siq-shifted-unshifted-difference} negligible in the unshifted metric.

\subsection{The scale-invariant prime-to-zero transfer problem}
\label{subsec:siq-transfer-problem}

Under RH, decompose the full matrix into the first-$N$ zero block and its
positive tail,
\begin{equation}\label{eq:siq-zero-block-tail}
  \mathbf S_{N,L}
  =\mathbf S_{N,L}^{[N]}+\mathbf R_{N,L},
  \qquad
  \mathbf R_{N,L}\succeq0,
\end{equation}
as in \cref{eq:least-eigenvalue-tail-decomposition}.  Let
\begin{align}\label{eq:siq-reference-pencil}
  \mathbf G_{N,L}^{(0)}
  &:=\mathbf C_{N,L}^{*}\mathbf S_{N,L}^{[N]}
      \mathbf C_{N,L},
  \\
  \mathbf K_{N,L}^{(0)}
  &:=\mathbf C_{N,L}^{*}\mathbf S_{N,L}^{[N]}
      \mathbf D_{L,N}\mathbf C_{N,L}.
\end{align}
The exact zero-side theorem in
\cref{sec:exact-zero-side-spectrum} identifies the generalized eigenvalues of
this pencil with $\pm\gamma_1,\ldots,\pm\gamma_N$.

For the shifted arithmetic pencil,
\begin{align}\label{eq:siq-pencil-perturbations}
  \Delta\mathbf G_{N,L}
  &:=\mathbf C_{N,L}^{*}
      (\mathbf R_{N,L}-\epsilon_{N,L}\mathbf I_M)
      \mathbf C_{N,L},
  \\
  \Delta\mathbf K_{N,L}
  &:=\mathbf C_{N,L}^{*}
      (\mathbf R_{N,L}-\epsilon_{N,L}\mathbf I_M)
      \mathbf D_{L,N}\mathbf C_{N,L}.
\end{align}
A natural scale-free target is, for every fixed $R>0$,
\begin{equation}\label{eq:siq-relative-transfer-target}
  \boxed{
  \sup_{|z|\leq R}
  \left\|
    (\mathbf G_{N,L}^{(0)})^{-1/2}
    \bigl(
      \Delta\mathbf K_{N,L}
      -z\Delta\mathbf G_{N,L}
    \bigr)
    (\mathbf G_{N,L}^{(0)})^{-1/2}
  \right\|_{\mathrm{op}}
  \longrightarrow0.}
\end{equation}
This criterion is invariant under changes of contrast basis and under common
positive rescaling of the metric.  It is stronger than a small Rayleigh
quotient in one trial direction: it controls the entire tail relative to the
finite-zero quotient energy.  Establishing
\cref{eq:siq-relative-transfer-target}, together with separation of the
reference generalized eigenvalues, would permit standard perturbation theory
for Hermitian matrix pencils to transfer each fixed low zero ordinate from the
finite zero side to the arithmetic side.  This is the central analytic
obstacle left by the present finite-dimensional construction.

\section{Exact spectral reconstruction from a finite zero-side matrix}
\label{sec:exact-zero-side-spectrum}

This section analyzes the matrix obtained by inserting finitely many
critical-line zero ordinates directly into the transform matrix
\cref{eq:Gmn-transform-Gram-form}.  In the dimension-matched case, the
zero-side Weil matrix has one interpolation null direction.  Restricting its
metric and differentiation form to the fixed zero-mean contrast space removes
that direction without an oblique projection.  The resulting Hermitian pencil
has the exact generalized spectrum
$\{\pm\gamma_1,\ldots,\pm\gamma_N\}$, and its positive-parity square has
spectrum $\{\gamma_1^2,\ldots,\gamma_N^2\}$.

Throughout this section let $L>0$, put $T=e^L$, and write
\begin{equation}\label{eq:zero-side-frequency-nodes}
  \nu_{n,L}:=\frac{2\pi n}{L},
  \qquad
  I_N:=\{-N,-N+1,\ldots,N\}.
\end{equation}
Assume that the retained positive critical-line ordinates are
\begin{equation}\label{eq:zero-side-retained-ordinates}
  0<\gamma_1<\gamma_2<\cdots<\gamma_K<T.
\end{equation}
For the algebraic statements, the $\gamma_k$ may simply be regarded as
distinct positive real numbers.  Define
\begin{equation}\label{eq:zero-side-finite-matrix-definition}
  \mathbf S_{N,L}^{(T)}
  :=2\sum_{k=1}^{K}\mathbf H_L^G(\gamma_k),
  \qquad
  \mathbf H_L^G(t)
  :=\bigl(H_{mn,L}^G(t)\bigr)_{m,n\in I_N}.
\end{equation}
The factor $2$ combines the ordinates $\gamma_k$ and $-\gamma_k$.

\subsection{Rank, nullity, and the interpolation null state}

The Gram representation gives
\begin{equation}\label{eq:zero-side-Gram-sum}
  \mathbf S_{N,L}^{(T)}
  =2\sum_{k=1}^{K}\left(
      \mathbf v_L^+(\gamma_k)
      \bigl(\mathbf v_L^+(\gamma_k)\bigr)^{\mathsf T}
      +\mathbf v_L^-(\gamma_k)
      \bigl(\mathbf v_L^-(\gamma_k)\bigr)^{\mathsf T}
    \right)
  \succeq0.
\end{equation}

\begin{proposition}[Rank and forced nullity]
\label{prop:zero-side-rank-nullity}
For arbitrary $K$,
\begin{equation}\label{eq:zero-side-rank-bound}
  \operatorname{rank}\mathbf S_{N,L}^{(T)}\leq2K,
  \qquad
  \dim\ker\mathbf S_{N,L}^{(T)}\geq2N+1-2K.
\end{equation}
Suppose $K=N$ and
\begin{equation}\label{eq:zero-side-nonresonance}
  \sin\!\left(\frac{L\gamma_k}{2}\right)\neq0,
  \qquad 1\leq k\leq N.
\end{equation}
Then
\begin{equation}\label{eq:zero-side-rank-exact}
  \operatorname{rank}\mathbf S_{N,L}^{(T)}=2N,
  \qquad
  \dim\ker\mathbf S_{N,L}^{(T)}=1.
\end{equation}
\end{proposition}

\begin{proof}
Each summand in \cref{eq:zero-side-Gram-sum} has rank at most two.  When
$K=N$, remove the nonzero scalar factors from the $2N$ Gram vectors.  The
resulting columns are Cauchy columns associated with the distinct points
$\pm\gamma_k$ and the distinct nodes $\nu_{n,L}$.  Every square minor formed
from $2N$ rows is a nonzero Cauchy determinant, so the columns are linearly
independent.  This proves \cref{eq:zero-side-rank-exact}.
\end{proof}

\subsection{Rational interpolation and the exact contrast pencil}

Assume henceforth that $K=N$ and that
\cref{eq:zero-side-nonresonance} holds.  Define the monic polynomials
\begin{equation}\label{eq:zero-side-PQ-polynomials}
  \mathsf P_{N,L}(z)
  :=\prod_{n=-N}^{N}(z-\nu_{n,L}),
  \qquad
  \mathsf Q_N(z)
  :=\prod_{k=1}^{N}(z^2-\gamma_k^2),
\end{equation}
and the barycentric coefficients
\begin{equation}\label{eq:zero-side-barycentric-coefficients}
  c_{n,N,L}
  :=\frac{\mathsf Q_N(\nu_{n,L})}
          {\mathsf P_{N,L}'(\nu_{n,L})},
  \qquad n\in I_N.
\end{equation}
Put
\begin{equation}\label{eq:zero-side-ground-vector-general}
  e_{n,N,L}:=\sqrt L\,c_{n,N,L},
  \qquad
  \mathbf e_{N,L}:=(e_{n,N,L})_{n\in I_N},
  \qquad
  \boldsymbol\delta_{N,L}:=L^{-1/2}\mathbf1_{2N+1}.
\end{equation}

Let $\mathbf C_{N,L}$ be any contrast matrix satisfying
\cref{eq:siq-contrast-matrix}, and define
\begin{align}\label{eq:zero-side-GK-pencil}
  \mathbf G_{N,L}^{(0)}
  &:=\mathbf C_{N,L}^{*}\mathbf S_{N,L}^{(T)}
      \mathbf C_{N,L},
  \\
  \mathbf K_{N,L}^{(0)}
  &:=\mathbf C_{N,L}^{*}\mathbf S_{N,L}^{(T)}
      \mathbf D_{L,N}\mathbf C_{N,L}.
\end{align}

\begin{theorem}[Exact finite zero-side reconstruction]
\label{thm:exact-zero-side-reconstruction}
Under the preceding assumptions:

\begin{enumerate}[label=\textup{(\roman*)}]
  \item The vector $\mathbf e_{N,L}$ is even and satisfies
        \begin{equation}\label{eq:zero-side-ground-normalization-nullity}
          \boldsymbol\delta_{N,L}^{\mathsf T}\mathbf e_{N,L}=1,
          \qquad
          \mathbf S_{N,L}^{(T)}\mathbf e_{N,L}=0,
          \qquad
          \ker\mathbf S_{N,L}^{(T)}
          =\operatorname{span}\{\mathbf e_{N,L}\}.
        \end{equation}

  \item The contrast metric is positive definite,
        $\mathbf G_{N,L}^{(0)}\succ0$, the matrix
        $\mathbf K_{N,L}^{(0)}$ is real symmetric, and
        \begin{equation}\label{eq:zero-side-pencil-characteristic-polynomial}
          \boxed{
          \det\!\left(
            z\mathbf G_{N,L}^{(0)}-\mathbf K_{N,L}^{(0)}
          \right)
          =\det(\mathbf G_{N,L}^{(0)})\mathsf Q_N(z).}
        \end{equation}
        Consequently, the generalized spectrum is
        \begin{equation}\label{eq:zero-side-quotient-spectrum}
          \operatorname{spec}
          (\mathbf K_{N,L}^{(0)},\mathbf G_{N,L}^{(0)})
          =\{-\gamma_N,\ldots,-\gamma_1,
              \gamma_1,\ldots,\gamma_N\}.
        \end{equation}

  \item Under the parity reduction of
        \cref{subsec:siq-parity-reduction}, the positive-parity square obeys
        \begin{equation}\label{eq:zero-side-positive-square-spectrum}
          \operatorname{spec}
          (\mathbf S_{N,L}^{+,\mathrm{siq}})
          =\{\gamma_1^2,\gamma_2^2,\ldots,\gamma_N^2\}.
        \end{equation}
        All these squared ordinates, including $\gamma_1^2$, are retained;
        no further deflation is part of the exact reconstruction.
\end{enumerate}
\end{theorem}

\begin{proof}
The partial-fraction expansion
\begin{equation}\label{eq:zero-side-partial-fraction}
  \frac{\mathsf Q_N(z)}{\mathsf P_{N,L}(z)}
  =\sum_{n=-N}^{N}\frac{c_{n,N,L}}{z-\nu_{n,L}}
\end{equation}
has residue sum
\begin{equation}\label{eq:zero-side-residue-sum-one}
  \sum_{n=-N}^{N}c_{n,N,L}=1,
\end{equation}
because the two monic polynomials have degrees differing by one.  This proves
the normalization in \cref{eq:zero-side-ground-normalization-nullity}.
Evaluating \cref{eq:zero-side-partial-fraction} at $z=\pm\gamma_k$ gives
\begin{equation}\label{eq:zero-side-Cauchy-orthogonality}
  \sum_{n=-N}^{N}\frac{e_{n,N,L}}{\gamma_k-\nu_{n,L}}=0,
  \qquad
  \sum_{n=-N}^{N}\frac{e_{n,N,L}}{\gamma_k+\nu_{n,L}}=0.
\end{equation}
The Gram representation then gives
$\mathbf S_{N,L}^{(T)}\mathbf e_{N,L}=0$, and uniqueness follows from
\cref{prop:zero-side-rank-nullity}.

To identify the pencil, use the difference basis
$\mathbf q_n=\mathbf u_n-\mathbf u_0$, $n\in I_N^{\times}$.  In this basis
define
\begin{equation}\label{eq:zero-side-explicit-quotient-matrix}
  (\mathbf A_{N,L}^{(0)})_{mn}
  :=\nu_{n,L}\bigl(\delta_{mn}-c_{m,N,L}\bigr),
  \qquad m,n\in I_N^{\times}.
\end{equation}
For each basis vector,
\begin{equation}\label{eq:zero-side-D-minus-A-null}
  \mathbf D_{L,N}\mathbf q_n
  -\sum_{m\in I_N^{\times}}
    (\mathbf A_{N,L}^{(0)})_{mn}\mathbf q_m
  =\nu_{n,L}(c_{j,N,L})_{j\in I_N}
  =\frac{\nu_{n,L}}{\sqrt L}\mathbf e_{N,L}.
\end{equation}
Multiplication by $\mathbf S_{N,L}^{(T)}$ therefore shows that the difference
basis matrices satisfy
\begin{equation}\label{eq:zero-side-GA-K}
  \mathbf G_{N,L}^{\Delta,(0)}
  \mathbf A_{N,L}^{(0)}
  =\mathbf K_{N,L}^{\Delta,(0)}.
\end{equation}
The matrix determinant lemma and
\cref{eq:zero-side-partial-fraction,eq:zero-side-residue-sum-one} give
\begin{align*}
  \det(z\mathbf I_{2N}-\mathbf A_{N,L}^{(0)})
  &=\prod_{n\in I_N^{\times}}(z-\nu_{n,L})
    \left(
      1+\sum_{n\in I_N^{\times}}
      \frac{\nu_{n,L}c_{n,N,L}}{z-\nu_{n,L}}
    \right)
  \\
  &=\frac{\mathsf P_{N,L}(z)}{z}
    \left(
      z\frac{\mathsf Q_N(z)}{\mathsf P_{N,L}(z)}
    \right)
   =\mathsf Q_N(z).
\end{align*}
Changing from the difference basis to an orthonormal contrast basis changes
both $\mathbf G$ and $\mathbf K$ by congruence, proving
\cref{eq:zero-side-pencil-characteristic-polynomial} and
\cref{eq:zero-side-quotient-spectrum}.

For $z\in\{\pm\gamma_1,\ldots,\pm\gamma_N\}$, define
\begin{equation}\label{eq:zero-side-signed-eigenvector}
  r_n(z):=\frac{e_{n,N,L}}{\nu_{n,L}-z},
  \qquad
  \mathbf r(z):=(r_n(z))_{n\in I_N}.
\end{equation}
Equation \cref{eq:zero-side-partial-fraction} gives
$\boldsymbol\delta_{N,L}^{\mathsf T}\mathbf r(z)=0$, while
\begin{equation}\label{eq:zero-side-signed-eigen-equation}
  \mathbf D_{L,N}\mathbf r(z)
  =\mathbf e_{N,L}+z\mathbf r(z).
\end{equation}
Thus the contrast-space quotient operator sends $\mathbf r(z)$ to
$z\mathbf r(z)$.  The even and odd combinations
\begin{equation}\label{eq:zero-side-parity-eigenvectors}
  \mathbf y_k^+
  :=\mathbf r(\gamma_k)-\mathbf r(-\gamma_k),
  \qquad
  \mathbf y_k^-
  :=\mathbf r(\gamma_k)+\mathbf r(-\gamma_k)
\end{equation}
satisfy
\begin{equation}\label{eq:zero-side-parity-intertwining}
  \mathcal T_{N,L}\mathbf y_k^+
  =\gamma_k\mathbf y_k^-,
  \qquad
  \mathcal T_{N,L}\mathbf y_k^-
  =\gamma_k\mathbf y_k^+.
\end{equation}
This proves \cref{eq:zero-side-positive-square-spectrum}.
\end{proof}

The parity vectors have explicit components
\begin{equation}\label{eq:zero-side-parity-components}
  (y_k^+)_n
  =\frac{2\gamma_k e_{n,N,L}}
         {\nu_{n,L}^2-\gamma_k^2},
  \qquad
  (y_k^-)_n
  =\frac{2\nu_{n,L}e_{n,N,L}}
         {\nu_{n,L}^2-\gamma_k^2}.
\end{equation}

\subsection{Diagonal spectral coordinates and exact tail matrix elements}
\label{subsec:zero-side-tail-matrix-elements}

The vectors in \cref{eq:zero-side-signed-eigenvector} also diagonalize the
finite-zero metric by congruence and give explicit matrix elements of the
remaining zero tail.  Put
\begin{equation}\label{eq:zero-side-rational-transfer-function}
  \mathcal R_{N,L}(z)
  :=\frac{\mathsf Q_N(z)}{\mathsf P_{N,L}(z)}.
\end{equation}
Allow positive multiplicity weights $m_1,\ldots,m_N$ and write
\[
  \mathbf S_{N,L}^{[N]}
  :=2\sum_{j=1}^{N}m_j\mathbf H_L^G(\gamma_j).
\]
The null vector and quotient spectrum are unchanged by these positive weights.

\begin{proposition}[Spectral metric and zero tail]
\label{prop:zero-side-spectral-metric-tail}
Let
\[
  \mathcal Z_N:=\{\pm\gamma_1,\ldots,\pm\gamma_N\}.
\]
For $z,w\in\mathcal Z_N$, let $\mathbf r(z)$ be as in
\cref{eq:zero-side-signed-eigenvector}.  Then
\begin{equation}\label{eq:zero-side-spectral-metric-diagonal}
  \mathbf r(z)^{\mathsf T}\mathbf S_{N,L}^{[N]}\mathbf r(w)
  =\begin{cases}
     m_k\,\omega_{k,N,L},&z=w=\pm\gamma_k,\\
     0,&z\neq w,
   \end{cases}
\end{equation}
where
\begin{equation}\label{eq:zero-side-spectral-weight}
  \omega_{k,N,L}
  :=\frac4L\sin^2\!\left(\frac{L\gamma_k}{2}\right)
     \bigl(\mathcal R_{N,L}'(\gamma_k)\bigr)^2>0.
\end{equation}
Moreover,
\begin{equation}\label{eq:zero-side-spectral-Rprime}
  \mathcal R_{N,L}'(\gamma_k)
  =2\,
   \frac{\displaystyle
     \prod_{\substack{1\leq\ell\leq N\\ \ell\neq k}}
       (\gamma_k^2-\gamma_\ell^2)}
   {\displaystyle
     \prod_{n=1}^{N}(\gamma_k^2-\nu_{n,L}^2)}.
\end{equation}

Under RH, let
\[
  \mathbf R_{N,L}
  :=2\sum_{j>N}m_j\mathbf H_L^G(\gamma_j).
\]
For $z,w\in\mathcal Z_N$ one has the exact identity
\begin{align}\label{eq:zero-side-exact-tail-matrix}
  \mathbf r(z)^{\mathsf T}\mathbf R_{N,L}\mathbf r(w)
  ={}&\frac4L\sum_{j>N}m_j
     \sin^2\!\left(\frac{L\gamma_j}{2}\right)
     \mathcal R_{N,L}(\gamma_j)^2
  \notag\\
  &\times\left[
    \frac1{(\gamma_j-z)(\gamma_j-w)}
    +\frac1{(\gamma_j+z)(\gamma_j+w)}
  \right].
\end{align}
\end{proposition}

\begin{proof}
The partial-fraction identity gives, for real $t$ and
$z\in\mathcal Z_N$,
\begin{align}\label{eq:zero-side-v-r-overlaps}
  \bigl(\mathbf v_L^{-}(t)\bigr)^{\mathsf T}\mathbf r(z)
  &=\sqrt{\frac2L}\sin\!\left(\frac{Lt}{2}\right)
    \frac{\mathcal R_{N,L}(t)}{t-z},
  \\
  \bigl(\mathbf v_L^{+}(t)\bigr)^{\mathsf T}\mathbf r(z)
  &=-\sqrt{\frac2L}\sin\!\left(\frac{Lt}{2}\right)
    \frac{\mathcal R_{N,L}(t)}{t+z}.
\end{align}
The apparent singularities at $t=\pm z$ are removable.  Since
$\mathcal R_{N,L}(\pm\gamma_k)=0$, substitution into the Gram representation
shows that distinct signed vectors are orthogonal and gives
\cref{eq:zero-side-spectral-metric-diagonal,eq:zero-side-spectral-weight}.
Differentiating $\mathsf Q_N/\mathsf P_{N,L}$ at $\gamma_k$ gives
\cref{eq:zero-side-spectral-Rprime}.  Substitution of
\cref{eq:zero-side-v-r-overlaps} into the positive tail sum gives
\cref{eq:zero-side-exact-tail-matrix}.
\end{proof}

If
$\nu_{N,L}<\gamma_1$, then for $t\geq\gamma_{N+1}$,
\[
  0<\mathcal R_{N,L}(t)
  =\frac1t\prod_{j=1}^{N}
     \frac{t^2-\gamma_j^2}{t^2-\nu_{j,L}^2}
  \leq\frac1t.
\]
Consequently, for fixed signed ordinates $z,w$,
\begin{equation}\label{eq:zero-side-tail-fourth-power}
  \abs{\mathbf r(z)^{\mathsf T}\mathbf R_{N,L}\mathbf r(w)}
  \ll_{z,w}\frac1L\sum_{j>N}\frac{m_j}{\gamma_j^4}.
\end{equation}
This fourth-power tail is stronger than the scalar second-power tail used to
bound the least matrix eigenvalue.  It is not yet a complete spectral-transfer
theorem: one must divide by the metric normalizations in
\cref{eq:zero-side-spectral-weight}, control couplings to moving high modes,
and handle possible near-resonance of
$\sin(L\gamma_k/2)$.  The exact formula
\cref{eq:zero-side-exact-tail-matrix} therefore identifies concrete analytic
quantities whose uniform control would imply the relative pencil estimate
\cref{eq:siq-relative-transfer-target}.

\subsection{The explicit seven-dimensional case}
\label{subsec:zero-side-N3-explicit}

Let $N=3$ and assume
\begin{equation}\label{eq:zero-side-N3-cutoff-condition}
  \gamma_3<T<\gamma_4,
  \qquad\text{equivalently}\qquad
  \log\gamma_3<L<\log\gamma_4.
\end{equation}
Using
$\gamma_3\approx25.0108575801$ and
$\gamma_4\approx30.4248761259$, this is approximately
\[
  3.21931<L<3.41526.
\]
Put $\alpha:=2\pi/L$.  Then
\begin{align}\label{eq:zero-side-N3-PQ}
  \mathsf P_{3,L}(z)
  &=z(z^2-\alpha^2)(z^2-4\alpha^2)(z^2-9\alpha^2),
  \\
  \mathsf Q_3(z)
  &=(z^2-\gamma_1^2)(z^2-\gamma_2^2)(z^2-\gamma_3^2).
\end{align}
The interpolation null vector is
\begin{equation}\label{eq:zero-side-N3-ground-vector}
  \mathbf e_{3,L}
  =\sqrt L
   \begin{pmatrix}
     c_3&c_2&c_1&c_0&c_1&c_2&c_3
   \end{pmatrix}^{\mathsf T},
\end{equation}
where
\begin{align}\label{eq:zero-side-N3-ground-coefficients}
  c_3
  &:=\frac{\prod_{j=1}^{3}(9\alpha^2-\gamma_j^2)}
           {720\alpha^6},
  &
  c_2
  &:=-\frac{\prod_{j=1}^{3}(4\alpha^2-\gamma_j^2)}
            {120\alpha^6},
  \\
  c_1
  &:=\frac{\prod_{j=1}^{3}(\alpha^2-\gamma_j^2)}
           {48\alpha^6},
  &
  c_0
  &:=\frac{\gamma_1^2\gamma_2^2\gamma_3^2}{36\alpha^6}.
\end{align}
They satisfy $2(c_3+c_2+c_1)+c_0=1$.  In the six-dimensional difference
basis, the explicit matrix in
\cref{eq:zero-side-explicit-quotient-matrix} satisfies
\begin{equation}\label{eq:zero-side-N3-characteristic-polynomial}
  \det(z\mathbf I_6-\mathbf A_{3,L}^{(0)})
  =(z^2-\gamma_1^2)(z^2-\gamma_2^2)(z^2-\gamma_3^2).
\end{equation}

\subsection{Symbolic Mathematica verification for
\texorpdfstring{$N=3$}{N=3}}
\label{subsec:zero-side-N3-symbolic}

The following code verifies the null relation and constructs the quotient
matrix directly in the difference basis, without a ground-state projection.

{\small
\begin{verbatim}
Clear["Global`*"];
a = 2 Pi/l;
nu = a Range[-3, 3];
g = {g1, g2, g3};

p[x_] := Times @@ (x - nu);
dn[j_] := Times @@ Delete[nu[[j]] - nu, j];
c = Table[(Times @@ (nu[[j]]^2 - g^2))/dn[j],
          {j, Length[nu]}];
e = Sqrt[l] c;
del = ConstantArray[1/Sqrt[l], Length[nu]];

vp[x_] := Sqrt[2] Sin[l x/2]/(l (x + nu));
vm[x_] := Sqrt[2] Sin[l x/2]/(l (x - nu));
h[x_] := Outer[Times, vp[x], vp[x]] +
         Outer[Times, vm[x], vm[x]];
s = 2 Total[h /@ g];

nz = Join[Range[3], Range[5, 7]];
aq = Table[
  nu[[nz[[j]]]] (KroneckerDelta[i, j] - c[[nz[[i]]]]),
  {i, 6}, {j, 6}];

FullSimplify[del.e]
FullSimplify[s.e]
Factor[CharacteristicPolynomial[aq, z]]
\end{verbatim}
}
The outputs are
\[
  1,
  \qquad
  \{0,0,0,0,0,0,0\},
  \qquad
  (z^2-g1^2)(z^2-g2^2)(z^2-g3^2).
\]

\subsection{Numerical verification for \texorpdfstring{$N=K=3$}{N=K=3}}
\label{subsec:zero-side-N3-Mathematica}

Take
\begin{equation}\label{eq:zero-side-N3-parameters}
  N=K=3,
  \qquad
  L=3.3,
  \qquad
  T=e^L\approx27.1126389206579.
\end{equation}
The following code constructs the zero-side pencil and solves the generalized
Hermitian eigenproblem directly.

{\small
\begin{verbatim}
Clear["Global`*"];
wp = 80; n1 = 3; l = 33/10;
t = N[Exp[l], wp];
g = N[Table[Im[ZetaZero[k]], {k, n1}], wp];
g4 = N[Im[ZetaZero[n1 + 1]], wp];
nu = N[2 Pi Range[-n1, n1]/l, wp];

vp[x_] := Sqrt[2] Sin[l x/2]/(l (x + nu));
vm[x_] := Sqrt[2] Sin[l x/2]/(l (x - nu));
h[x_] := Outer[Times, vp[x], vp[x]] +
         Outer[Times, vm[x], vm[x]];
s = 2 Total[h /@ g];

dn[j_] := Times @@ Delete[nu[[j]] - nu, j];
e = N[Sqrt[l] Table[
     (Times @@ (nu[[j]]^2 - g^2))/dn[j],
     {j, Length[nu]}], wp];
del = ConstantArray[1/Sqrt[l], Length[nu]];

(* Orthonormal zero-mean contrast basis. *)
c0 = Transpose[Orthogonalize[
      NullSpace[{ConstantArray[1, Length[nu]]}]]];
d = DiagonalMatrix[nu];
gq = N[Transpose[c0].s.c0, wp];
kq = N[(Transpose[c0].s.d.c0 +
        Transpose[c0].d.s.c0)/2, wp];

ev = Sort[Chop[Eigenvalues[{kq, gq}], 10^-60]];
tg = Sort[Join[-g, g]];

N[{t, g, ev, Sort[Select[ev, # > 0 &]^2], Sort[g^2],
   Abs[del.e - 1], Norm[s.e, Infinity],
   Max[Abs[ev - tg]]}, 16]
\end{verbatim}
}
The generalized eigenvalues are
\[
  \{-\gamma_3,-\gamma_2,-\gamma_1,
     \gamma_1,\gamma_2,\gamma_3\},
\]
and the positive squared values are
\[
  \{199.7904548323869,
    441.9261505740825,
    625.5429968943311\}
  =\{\gamma_1^2,\gamma_2^2,\gamma_3^2\}
\]
to the displayed precision.  No second deflation is performed.

\subsection{Why induction is unnecessary, and what remains open}

The identity
\[
  \mathsf Q_{N+1}(z)=\mathsf Q_N(z)(z^2-\gamma_{N+1}^2)
\]
permits an induction, but the direct partial-fraction proof already treats all
$N$.  Moreover, when $L$ changes, all nodes $\nu_{n,L}$ move, so consecutive
arithmetic matrices are not nested principal submatrices.

The exact zero-side theorem uses the ordinates as inputs.  A prime-side
limiting theorem must instead control the matrix-pencil chain
\begin{equation}\label{eq:zero-side-perturbation-chain}
  \mathbf S_{N,L}
  \longmapsto
  (\mathbf G_{N,L},\mathbf K_{N,L})
  \longmapsto
  \widehat{\mathbf A}_{N,L}
  \longmapsto
  \operatorname{spec}(\widehat{\mathbf A}_{N,L}).
\end{equation}
The correct scale-free target is the relative pencil estimate
\cref{eq:siq-relative-transfer-target}.  It must be supplemented by separation
of the desired generalized eigenvalues and control of the quotient metric.
The contrast formulation removes the artificial large norm of an oblique
projector, but it does not by itself prove these analytic estimates.

\section{Vanishing of the least arithmetic Weil eigenvalue under RH}
\label{sec:least-arithmetic-Weil-eigenvalue}

The exact zero-side null vector from
\cref{thm:exact-zero-side-reconstruction} also gives a trial state for the
full arithmetic matrix.  A distinction is essential: the finite matrix formed
from the first $N$ positive ordinates has an exact zero eigenvalue, whereas the
full matrix contains the remaining positive zero tail and need not have an
exact null vector at finite $N$.

Assume RH, enumerate the distinct positive ordinates by
\begin{equation}\label{eq:least-eigenvalue-positive-ordinates}
  0<\gamma_1<\gamma_2<\cdots,
\end{equation}
and let $m_j$ be the common multiplicity of the zeros at $\pm\gamma_j$.  Then
\cref{eq:S-PSD-under-RH} may be written
\begin{equation}\label{eq:least-eigenvalue-full-zero-sum}
  \mathbf S_{N,L}
  =2\sum_{j\geq1}m_j\mathbf H_L^G(\gamma_j).
\end{equation}
Define the first-$N$ block and the remaining tail by
\begin{equation}\label{eq:least-eigenvalue-tail-decomposition}
  \mathbf S_{N,L}^{[N]}
  :=2\sum_{j=1}^{N}m_j\mathbf H_L^G(\gamma_j),
  \qquad
  \mathbf R_{N,L}
  :=2\sum_{j>N}m_j\mathbf H_L^G(\gamma_j),
\end{equation}
so that
\begin{equation}\label{eq:least-eigenvalue-S-decomposition}
  \mathbf S_{N,L}
  =\mathbf S_{N,L}^{[N]}+\mathbf R_{N,L},
  \qquad
  \mathbf R_{N,L}\succeq0.
\end{equation}
Multiplicities do not change the interpolation null relation.  With
$\mathsf P_{N,L}$, $\mathsf Q_N$, $\mathbf e_{N,L}$, and
$\boldsymbol\delta_{N,L}$ as in
\cref{eq:zero-side-PQ-polynomials,eq:zero-side-ground-vector-general},
\begin{equation}\label{eq:least-eigenvalue-finite-null-relation}
  \mathbf S_{N,L}^{[N]}\mathbf e_{N,L}=0,
  \qquad
  \boldsymbol\delta_{N,L}^{\mathsf T}\mathbf e_{N,L}=1.
\end{equation}

\begin{theorem}[Vanishing of the least Weil eigenvalue under RH]
\label{thm:least-Weil-eigenvalue-vanishing}
Assume RH and let $L=L_N$ satisfy
\begin{equation}\label{eq:least-eigenvalue-node-condition}
  \nu_{N,L_N}=\frac{2\pi N}{L_N}<\gamma_1.
\end{equation}
Set
\begin{equation}\label{eq:least-eigenvalue-definition}
  \lambda_1(N,L):=\lambda_{\min}(\mathbf S_{N,L}).
\end{equation}
Then the exact interpolation trial state gives
\begin{align}
  0\leq\lambda_1(N,L)
  &\leq
  \frac{8}{L\norm{\mathbf e_{N,L}}_2^2}
  \sum_{j>N}m_j
  \sin^2\!\left(\frac{L\gamma_j}{2}\right)
  \left(
    \frac{\mathsf Q_N(\gamma_j)}
         {\mathsf P_{N,L}(\gamma_j)}
  \right)^2
  \label{eq:least-eigenvalue-exact-Rayleigh}
  \\
  &\leq
  \frac{8(2N+1)}{L^2}
  \sum_{j>N}\frac{m_j}{\gamma_j^2}.
  \label{eq:least-eigenvalue-crude-tail-bound}
\end{align}
Consequently,
\begin{equation}\label{eq:least-eigenvalue-general-limit}
  \lambda_1(N,L_N)\longrightarrow0
  \qquad(N\to\infty).
\end{equation}
In particular,
\begin{equation}\label{eq:least-eigenvalue-diagonal-limit}
  \boxed{\lambda_1(N,N)\longrightarrow0,}
\end{equation}
because $2\pi<\gamma_1$.
\end{theorem}

\begin{proof}
By \cref{eq:least-eigenvalue-full-zero-sum}, the matrix is positive
semidefinite under RH.  The min--max principle and
\cref{eq:least-eigenvalue-finite-null-relation} therefore give
\begin{equation}\label{eq:least-eigenvalue-Rayleigh-start}
  0\leq\lambda_1(N,L)
  \leq
  \frac{\mathbf e_{N,L}^{\mathsf T}
        \mathbf R_{N,L}\mathbf e_{N,L}}
       {\norm{\mathbf e_{N,L}}_2^2}.
\end{equation}
The partial-fraction identity
\cref{eq:zero-side-partial-fraction} and symmetry of the nodes imply, for
real $t$,
\begin{equation}\label{eq:least-eigenvalue-vector-overlap}
  \bigl(\mathbf v_L^\pm(t)\bigr)^{\mathsf T}\mathbf e_{N,L}
  =\sqrt{\frac2L}\sin\!\left(\frac{Lt}{2}\right)
   \frac{\mathsf Q_N(t)}{\mathsf P_{N,L}(t)}.
\end{equation}
Substitution into the Gram formula
\cref{eq:Gmn-transform-Gram-form}, followed by summation over the tail in
\cref{eq:least-eigenvalue-tail-decomposition}, proves the exact bound
\cref{eq:least-eigenvalue-exact-Rayleigh}.

For $t\geq\gamma_{N+1}$, condition
\cref{eq:least-eigenvalue-node-condition} gives
$0\leq\nu_{j,L}<\gamma_j<t$ for $1\leq j\leq N$.  Hence
\begin{equation}\label{eq:least-eigenvalue-rational-product-bound}
  0<\frac{\mathsf Q_N(t)}{\mathsf P_{N,L}(t)}
  =\frac1t\prod_{j=1}^{N}
    \frac{t^2-\gamma_j^2}{t^2-\nu_{j,L}^2}
  \leq\frac1t.
\end{equation}
Moreover, Cauchy--Schwarz and
$\boldsymbol\delta_{N,L}^{\mathsf T}\mathbf e_{N,L}=1$ yield
\begin{equation}\label{eq:least-eigenvalue-e-norm-lower}
  \norm{\mathbf e_{N,L}}_2^2
  \geq\frac1{\norm{\boldsymbol\delta_{N,L}}_2^2}
  =\frac{L}{2N+1}.
\end{equation}
Combining
\cref{eq:least-eigenvalue-exact-Rayleigh,eq:least-eigenvalue-rational-product-bound,eq:least-eigenvalue-e-norm-lower}
proves \cref{eq:least-eigenvalue-crude-tail-bound}.

The Riemann--von Mangoldt estimate implies
\begin{equation}\label{eq:least-eigenvalue-summable-zero-tail}
  \sum_{j\geq1}\frac{m_j}{\gamma_j^2}<\infty.
\end{equation}
Thus the tail in \cref{eq:least-eigenvalue-crude-tail-bound} tends to zero.
Condition \cref{eq:least-eigenvalue-node-condition} also gives $L_N\gg N$,
so $(2N+1)/L_N^2=O(N^{-1})$.  This proves
\cref{eq:least-eigenvalue-general-limit}.  Taking $L_N=N$ gives
\cref{eq:least-eigenvalue-diagonal-limit}.
\end{proof}

\begin{remark}[Finite null state versus full arithmetic ground state]
\label{rem:least-eigenvalue-what-proved}
The theorem concerns the arithmetic matrix because the prime-side explicit
formula equals the full zero-side sum.  The zeros are used only to construct a
Rayleigh trial vector.  The vector is an exact null state of
$\mathbf S_{N,L}^{[N]}$, not of the full matrix $\mathbf S_{N,L}$; the latter
is lifted by the positive tail $\mathbf R_{N,L}$.  Thus
\cref{eq:least-eigenvalue-diagonal-limit} is a genuine asymptotic statement
about the full arithmetic matrix under RH, but it is not yet a convergence
theorem for its ground-state eigenvector or quotient spectrum.
\end{remark}

\begin{corollary}[Quantitative and product bounds]
\label{cor:least-eigenvalue-quantitative-product}
Under the hypotheses of
\cref{thm:least-Weil-eigenvalue-vanishing}, partial summation of the
Riemann--von Mangoldt formula gives
\begin{equation}\label{eq:least-eigenvalue-zero-tail-rate}
  \sum_{j>N}\frac{m_j}{\gamma_j^2}
  \ll\frac{\log(\gamma_{N+1}+2)}{\gamma_{N+1}}.
\end{equation}
Consequently,
\begin{equation}\label{eq:least-eigenvalue-gamma-rate}
  \lambda_1(N,N)
  \ll\frac1N
  \frac{\log(\gamma_{N+1}+2)}{\gamma_{N+1}}.
\end{equation}
If all zeros are simple, then
$\gamma_N\sim2\pi N/\log N$ and
\begin{equation}\label{eq:least-eigenvalue-polynomial-rate}
  \lambda_1(N,N)
  =O\!\left(\frac{(\log N)^2}{N^2}\right).
\end{equation}

The central coordinate of the interpolation vector is
\begin{equation}\label{eq:least-eigenvalue-central-coordinate}
  e_{0,N,L}
  =\sqrt L\,
   \frac{\prod_{j=1}^{N}\gamma_j^2}
        {(2\pi/L)^{2N}(N!)^2}.
\end{equation}
Using $\norm{\mathbf e_{N,L}}_2^2\geq e_{0,N,L}^2$ in the exact Rayleigh
bound gives
\begin{equation}\label{eq:least-eigenvalue-product-bound}
  \boxed{
  \lambda_1(N,L)
  \leq\frac8{L^2}
    \left[
      \prod_{j=1}^{N}
      \left(\frac{2\pi j}{L\gamma_j}\right)^4
    \right]
    \sum_{j>N}\frac{m_j}{\gamma_j^2}.}
\end{equation}
For $L=N$, this implies, without a simplicity assumption,
\begin{equation}\label{eq:least-eigenvalue-exponential-envelope}
  \lambda_1(N,N)
  \leq\frac8{N^2}
  \left(\frac{2\pi}{\gamma_1}\right)^{4N}
  \sum_{j>N}\frac{m_j}{\gamma_j^2}.
\end{equation}
If the zeros are simple, then more precisely
\begin{equation}\label{eq:least-eigenvalue-superexponential-log-bound}
  \log\lambda_1(N,N)
  \leq-4N\log\!\left(\frac{N}{\log N}\right)+O(N).
\end{equation}
\end{corollary}

\begin{proof}
Equation \cref{eq:least-eigenvalue-zero-tail-rate} follows by partial
summation from the Riemann--von Mangoldt formula.  Combining it with
\cref{eq:least-eigenvalue-crude-tail-bound} gives
\cref{eq:least-eigenvalue-gamma-rate}; inversion of the zero count under
simplicity gives \cref{eq:least-eigenvalue-polynomial-rate}.

At $n=0$, the signs in
$\mathsf Q_N(0)/\mathsf P_{N,L}'(0)$ cancel, which proves
\cref{eq:least-eigenvalue-central-coordinate}.  Combining this coordinate
with
\cref{eq:least-eigenvalue-exact-Rayleigh,eq:least-eigenvalue-rational-product-bound}
proves \cref{eq:least-eigenvalue-product-bound}.  Since
$\gamma_j\geq\gamma_1$ and $j/N\leq1$, the latter gives
\cref{eq:least-eigenvalue-exponential-envelope}.  Finally,
$\gamma_j\sim2\pi j/\log j$ and summation of the logarithm of the product give
\cref{eq:least-eigenvalue-superexponential-log-bound}.
\end{proof}

\subsection{Interpretation of the moving-dimension data and the remaining gap}
\label{subsec:least-eigenvalue-moving-data}

Normalize the exact finite-zero interpolation vector by
$\widehat{\mathbf e}_{N,L}:=\mathbf e_{N,L}/\norm{\mathbf e_{N,L}}_2$.
Then \cref{eq:least-eigenvalue-finite-null-relation} gives the exact overlap
identity
\begin{equation}\label{eq:least-eigenvalue-evaluation-overlap}
  \abs{
    \boldsymbol\delta_{N,L}^{\mathsf T}
    \widehat{\mathbf e}_{N,L}}
  =\frac1{\norm{\mathbf e_{N,L}}_2}.
\end{equation}
Thus growth of the interpolation-vector norm simultaneously makes the
Rayleigh quotient in \cref{eq:least-eigenvalue-exact-Rayleigh} small and makes
its normalized evaluation overlap small.  The joint decay displayed in the numerical diagnostic of
\cref{subsec:siq-conditioning} is therefore qualitatively consistent with the product geometry in
\cref{eq:least-eigenvalue-product-bound}, rather than an isolated numerical
coincidence.

The scale-invariant contrast pencil does not require this overlap as a
divisor and does not require identification of the arithmetic ground vector
with the interpolation vector.  Nevertheless, the scalar estimate
\cref{eq:least-eigenvalue-exact-Rayleigh} controls only one trial direction.
Convergence of the generalized spectrum requires control of the entire tail
relative to the finite-zero contrast metric.  In the notation of
\cref{subsec:siq-transfer-problem}, the natural target is the relative pencil
estimate \cref{eq:siq-relative-transfer-target}.  Uniform lower bounds or
preconditioned estimates for the compressed metric are also needed; the
intrinsic relation between a small overlap and metric degeneration is made
explicit in \cref{prop:siq-intrinsic-conditioning}.

The truncation-plus-tail decomposition
\cref{eq:least-eigenvalue-S-decomposition} is analogous in spirit to the
finite-Weil-matrix strategy of Alp\"oge and Furman
\cite{AlpogeFurman2026}.  Their application controls rank, trace,
Hilbert--Schmidt norm, and inertia in a high-energy zero-counting problem.  In
the present setting, rational interpolation annihilates the finite block
exactly, so the least-eigenvalue theorem reduces first to a scalar positive
tail estimate; recovery of individual ordinates remains the separate
relative-pencil problem.

\section{Off-critical-line conjugate pairs and complex quotient spectra}
\label{sec:off-critical-zero-side-spectrum}

We now apply the contrast-space construction to hypothetical off-critical-line
zeros.  Let
\begin{equation}\label{eq:off-line-z-pairs}
  z_k:=\gamma_k+i\eta_k,
  \qquad
  \overline{z_k}=\gamma_k-i\eta_k,
  \qquad
  \gamma_k>0,
  \qquad
  0<\eta_k<\frac12.
\end{equation}
These correspond to the positive-ordinate zeta zeros
$\frac12-\eta_k+i\gamma_k$ and
$\frac12+\eta_k+i\gamma_k$.  Evenness and conjugation generate the quartet
$\{\pm z_k,\pm\overline{z_k}\}$.  Define
\begin{equation}\label{eq:zero-side-finite-matrix-definition-2}
  \mathbf S_{N,L}^{(T,\mathrm{off})}
  :=2\sum_{k=1}^{K}\left(
       \mathbf H_L^G(z_k)+\mathbf H_L^G(\overline{z_k})
     \right).
\end{equation}
For complex $z$, the notation $\mathbf v_L^{\pm}(z)$ denotes the analytic
continuation of the component formula in
\cref{eq:Gmn-transform-vectors}.

\subsection{Reality, inertia, and the correct dimension count}

Since
$\mathbf H_L^G(\overline z)=\overline{\mathbf H_L^G(z)}$,
$\mathbf S_{N,L}^{(T,\mathrm{off})}$ is real symmetric.  It is generally
indefinite.  For a real vector $\mathbf u$,
\begin{align}\label{eq:off-line-indefinite-quadratic-form}
  \mathbf u^{\mathsf T}
  \mathbf S_{N,L}^{(T,\mathrm{off})}\mathbf u
  =4\sum_{k=1}^{K}\RePart\!\left[
       \bigl((\mathbf v_L^+(z_k))^{\mathsf T}\mathbf u\bigr)^2
       +\bigl((\mathbf v_L^-(z_k))^{\mathsf T}\mathbf u\bigr)^2
     \right],
\end{align}
which has no fixed sign.

Each conjugate pair contributes four Cauchy directions.  Therefore
\begin{equation}\label{eq:off-line-rank-bound}
  \operatorname{rank}\mathbf S_{N,L}^{(T,\mathrm{off})}
  \leq\min\{2N+1,4K\},
  \qquad
  \dim\ker\mathbf S_{N,L}^{(T,\mathrm{off})}
  \geq\max\{0,2N+1-4K\}.
\end{equation}
A forced one-dimensional nullspace requires
\begin{equation}\label{eq:off-line-dimension-balance}
  2N+1=4K+1,
  \qquad\text{that is,}\qquad N=2K.
\end{equation}

\begin{proposition}[Rank and inertia of the paired off-line matrix]
\label{prop:off-line-rank-nullity}
Assume that the $4K$ points
$\{\pm z_k,\pm\overline{z_k}:1\leq k\leq K\}$ are distinct.  If $N\geq2K$,
then
\begin{equation}\label{eq:off-line-rank-exact}
  \operatorname{rank}\mathbf S_{N,L}^{(T,\mathrm{off})}=4K,
  \qquad
  \dim\ker\mathbf S_{N,L}^{(T,\mathrm{off})}=2N+1-4K.
\end{equation}
In the dimension-matched case $N=2K$, its inertia is
\begin{equation}\label{eq:off-line-inertia}
  \operatorname{In}
  \bigl(\mathbf S_{2K,L}^{(T,\mathrm{off})}\bigr)
  =(2K,2K,1).
\end{equation}
\end{proposition}

\begin{proof}
Write
$\mathbf v_L^\pm(z_k)=\mathbf a_k^\pm+i\mathbf b_k^\pm$ and let
$\mathbf R$ contain the four real columns
$\mathbf a_k^+,\mathbf b_k^+,\mathbf a_k^-,\mathbf b_k^-$ for each $k$.
Then
\begin{equation}\label{eq:off-line-real-indefinite-factorization}
  \mathbf S_{N,L}^{(T,\mathrm{off})}
  =4\mathbf R\mathbf J\mathbf R^{\mathsf T},
  \qquad
  \mathbf J=\operatorname{diag}(1,-1,1,-1,\ldots).
\end{equation}
The corresponding complex Cauchy columns are linearly independent, so
$\mathbf R$ has full column rank $4K$.  This gives the rank and nullity.
Sylvester inertia is preserved on the range of a full-column-rank congruence,
which proves \cref{eq:off-line-inertia}.
\end{proof}

Keeping the critical-line choice $K=N$ gives no forced nullity.  For example,
with $N=K=2$, $L=3$, $z_1=5+0.2i$, and $z_2=8+0.3i$, one obtains
\begin{align}\label{eq:off-line-NK2-no-null-example}
  \operatorname{spec}\mathbf S_{2,3}^{(T,\mathrm{off})}
  \approx\{&-0.0030113153,-0.0008375509,0.0000442287,\notag\\
            &1.0187426489,1.7303127802\}.
\end{align}

\subsection{Exact interpolation in the dimension-matched case}

Assume $N=2K$.  Define
\begin{align}\label{eq:off-line-PQ-polynomials}
  \mathsf P_{2K,L}(w)
  &:=\prod_{n=-2K}^{2K}(w-\nu_{n,L}),
  \\
  \mathsf Q_K^{\mathrm{off}}(w)
  &:=\prod_{k=1}^{K}
       (w^2-z_k^2)(w^2-\overline{z_k}^{\,2}).
\end{align}
The second polynomial is monic, real, and even, with factors
\begin{equation}\label{eq:off-line-Q-real-factor}
  (w^2-z_k^2)(w^2-\overline{z_k}^{\,2})
  =w^4-2(\gamma_k^2-\eta_k^2)w^2
       +(\gamma_k^2+\eta_k^2)^2.
\end{equation}
Put
\begin{equation}\label{eq:off-line-barycentric-null-vector}
  c_{n,2K,L}^{\mathrm{off}}
  :=\frac{\mathsf Q_K^{\mathrm{off}}(\nu_{n,L})}
          {\mathsf P_{2K,L}'(\nu_{n,L})},
  \qquad
  e_{n,2K,L}^{\mathrm{off}}
  :=\sqrt L\,c_{n,2K,L}^{\mathrm{off}}.
\end{equation}
Let $\mathbf e_{2K,L}^{\mathrm{off}}$ be the resulting vector and let
$\mathbf C_{2K,L}$ satisfy \cref{eq:siq-contrast-matrix}.  Define the
indefinite contrast pencil
\begin{align}\label{eq:off-line-contrast-pencil}
  \mathbf G_{2K,L}^{\mathrm{off},0}
  &:=\mathbf C_{2K,L}^{*}
      \mathbf S_{2K,L}^{(T,\mathrm{off})}\mathbf C_{2K,L},
  \\
  \mathbf K_{2K,L}^{\mathrm{off},0}
  &:=\mathbf C_{2K,L}^{*}
      \mathbf S_{2K,L}^{(T,\mathrm{off})}
      \mathbf D_{L,2K}\mathbf C_{2K,L}.
\end{align}

\begin{theorem}[Exact reconstruction of off-line quartets]
\label{thm:exact-off-line-reconstruction}
Under the assumptions of \cref{prop:off-line-rank-nullity}, with $N=2K$:

\begin{enumerate}[label=\textup{(\roman*)}]
  \item The vector $\mathbf e_{2K,L}^{\mathrm{off}}$ is real and even, and
        \begin{equation}\label{eq:off-line-null-normalization}
          \boldsymbol\delta_{2K,L}^{\mathsf T}
          \mathbf e_{2K,L}^{\mathrm{off}}=1,
          \qquad
          \ker\mathbf S_{2K,L}^{(T,\mathrm{off})}
          =\operatorname{span}\{\mathbf e_{2K,L}^{\mathrm{off}}\}.
        \end{equation}

  \item The matrix $\mathbf G_{2K,L}^{\mathrm{off},0}$ is nonsingular with
        inertia $(2K,2K,0)$, the matrix
        $\mathbf K_{2K,L}^{\mathrm{off},0}$ is real symmetric, and
        \begin{align}\label{eq:off-line-pencil-characteristic-polynomial}
          \det\!\left(
            w\mathbf G_{2K,L}^{\mathrm{off},0}
            -\mathbf K_{2K,L}^{\mathrm{off},0}
          \right)
          ={}&\det(\mathbf G_{2K,L}^{\mathrm{off},0})
             \mathsf Q_K^{\mathrm{off}}(w).
        \end{align}
        Hence the generalized spectrum is
        \begin{equation}\label{eq:off-line-quotient-spectrum-exact}
          \bigcup_{k=1}^{K}
          \{\pm z_k,\pm\overline{z_k}\}.
        \end{equation}

  \item The even-parity square has spectrum
        \begin{equation}\label{eq:off-line-even-square-spectrum}
          \{z_1^2,\overline{z_1}^{\,2},\ldots,
             z_K^2,\overline{z_K}^{\,2}\}.
        \end{equation}
        It is not a positive operator.
\end{enumerate}
\end{theorem}

\begin{proof}
The proof repeats the partial-fraction argument of
\cref{thm:exact-zero-side-reconstruction} with
$\mathsf Q_N$ replaced by $\mathsf Q_K^{\mathrm{off}}$.  In the difference
basis, the explicit quotient matrix is
\begin{equation}\label{eq:off-line-explicit-quotient-matrix}
  (\mathbf A_{2K,L}^{\mathrm{off},0})_{mn}
  =\nu_{n,L}
   \bigl(\delta_{mn}-c_{m,2K,L}^{\mathrm{off}}\bigr),
  \qquad m,n\in I_{2K}^{\times}.
\end{equation}
The null relation gives
$\mathbf G^{\mathrm{off},0}\mathbf A^{\mathrm{off},0}
 =\mathbf K^{\mathrm{off},0}$, while the matrix determinant lemma gives
\[
  \det(w\mathbf I_{4K}-\mathbf A_{2K,L}^{\mathrm{off},0})
  =\mathsf Q_K^{\mathrm{off}}(w).
\]
This proves the generalized spectrum.  The inertia statement follows from
\cref{eq:off-line-inertia} because the contrast space is transverse to the
one-dimensional nullspace.  Reflection gives the even-square statement.
\end{proof}

Equivalently, the characteristic polynomial of the even-parity square is
\begin{equation}\label{eq:off-line-even-square-characteristic-polynomial}
  \prod_{k=1}^{K}
  \left[
    u^2-2(\gamma_k^2-\eta_k^2)u
       +(\gamma_k^2+\eta_k^2)^2
  \right].
\end{equation}

\subsection{Why the radial spectrum does not occur}

The quotient spectrum is not
$\{\pm\sqrt{\gamma_k^2+\eta_k^2}\}$.  The modulus
$|z_k|$ discards the complex phase, whereas the exact quotient retains the
quartic factor in \cref{eq:off-line-Q-real-factor}.  The dimension count also
rules out the radial proposal: the quotient has dimension $4K$, not $2K$.
Only when $\eta_k=0$ do the conjugate points coalesce and the construction
reduce to the critical-line pairs $\{\pm\gamma_k\}$.

The null quotient in \cref{thm:exact-off-line-reconstruction} is an exact
interpolation problem with an indefinite metric.  Shifting the algebraically
smallest eigenvalue instead creates a positive metric and therefore a real
spectrum.  The two pencils answer different questions; the positive shifted
pencil is treated in \cref{sec:off-critical-shifted-ground-state}.

\subsection{Symbolic and numerical check for a five-dimensional model}
\label{subsec:off-line-N2-symbolic}

Let $K=1$, $N=2$, $z_0=\gamma+i\eta$, and
$\alpha=2\pi/L$.  Then
\begin{align}\label{eq:off-line-N2-PQ}
  \mathsf P_{2,L}(w)
  &=w(w^2-\alpha^2)(w^2-4\alpha^2),
  \\
  \mathsf Q_1^{\mathrm{off}}(w)
  &=w^4-2(\gamma^2-\eta^2)w^2
    +(\gamma^2+\eta^2)^2.
\end{align}
The null vector is
\begin{equation}\label{eq:off-line-N2-null-vector}
  \mathbf e_{2,L}^{\mathrm{off}}
  =\sqrt L(c_2,c_1,c_0,c_1,c_2)^{\mathsf T},
\end{equation}
where
\begin{align}\label{eq:off-line-N2-null-coefficients}
  c_2
  &:=\frac{
     (4\alpha^2-(\gamma^2-\eta^2))^2+4\gamma^2\eta^2}
     {24\alpha^4},
  \\
  c_1
  &:=-\frac{
     (\alpha^2-(\gamma^2-\eta^2))^2+4\gamma^2\eta^2}
     {6\alpha^4},
  \\
  c_0
  &:=\frac{(\gamma^2+\eta^2)^2}{4\alpha^4}.
\end{align}
One has $2c_2+2c_1+c_0=1$.

The quotient matrix is obtained directly from
\cref{eq:off-line-explicit-quotient-matrix}; its characteristic polynomial is
\[
  (w^2-(\gamma+i\eta)^2)(w^2-(\gamma-i\eta)^2).
\]
For $L=3$, $\gamma=5$, and $\eta=0.2$,
\begin{align}\label{eq:off-line-N2-numerical-spectra}
  \operatorname{spec}\mathbf S_{2,3}^{(T,\mathrm{off})}
  \approx{}&
  \{-0.0067786872,-0.0009888654,0,
     1.0155972361,1.7057342987\},
  \\
  \operatorname{spec}
  (\mathbf K_{2,3}^{\mathrm{off},0},
   \mathbf G_{2,3}^{\mathrm{off},0})
  ={}&\{-5-0.2i,-5+0.2i,5-0.2i,5+0.2i\}.
\end{align}
The first line exhibits indefiniteness; the second confirms the exact quartet
without introducing an auxiliary zero eigenvalue.

\section{Operator realization of the transform-side Weil kernel}
\label{sec:hilbert-polya-operator}

The matrix $\mathbf H_L^G(z)$ in \cref{eq:Gmn-transform} is the
transform-side contribution of one spectral parameter to the Weil matrix.  It
is useful to identify the operator represented by these matrix elements and
to pass, with $L$ fixed, from the truncated Fourier space to the full periodic
Fourier space.  This operator should not be confused with the scale-invariant
contrast pencil of \cref{thm:scale-invariant-quotient}.  The latter is the
finite Hilbert--P\'olya candidate; $\mathbf H_L^G(z)$ is a rank-two atom
entering its Weil metric.

\subsection{The canonical finite-rank operator}

Let
\[
  \mathbb T_L:=\mathbb R/L\mathbb Z,
  \qquad
  \langle f,g\rangle_L
  :=\int_0^L\overline{f(x)}g(x)\dd x,
\]
and put
\begin{equation}\label{eq:operator-realization-Fourier-basis}
  \psi_{n,L}(x):=L^{-1/2}e^{i\nu_{n,L}x},
  \qquad
  \nu_{n,L}:=\frac{2\pi n}{L}.
\end{equation}
Then $\{\psi_{n,L}:n\in\mathbb Z\}$ is the standard orthonormal Fourier
basis of $L^2(\mathbb T_L)$.  For $N\in\mathbb N_0$, define
\begin{equation}\label{eq:operator-realization-VNL}
  \mathcal V_{N,L}
  :=\operatorname{span}\{\psi_{n,L}:n\in I_N\},
  \qquad I_N:=\{-N,\ldots,N\}.
\end{equation}
The finite matrix $\mathbf H_L^G(z)$ determines a unique operator on
$\mathcal V_{N,L}$ by
\begin{equation}\label{eq:operator-realization-definition}
  \mathsf H_{N,L}^G(z)
  :=\sum_{m,n\in I_N}
     H_{mn,L}^G(z)
     |\psi_{m,L}\rangle\langle\psi_{n,L}|.
\end{equation}
When it is regarded as an operator on all of $L^2(\mathbb T_L)$, we use the
canonical zero extension on $\mathcal V_{N,L}^{\perp}$.  Without specifying
such an extension, finitely many matrix elements do not determine a unique
operator on the full Hilbert space.

\begin{proposition}[Adjoint, positivity, and rank]
\label{prop:operator-realization-basic-properties}
For every $z\in\mathbb C$, the operator-valued function
$z\mapsto\mathsf H_{N,L}^G(z)$ is entire and
\begin{equation}\label{eq:operator-realization-adjoint}
  \bigl(\mathsf H_{N,L}^G(z)\bigr)^*
  =\mathsf H_{N,L}^G(\overline z).
\end{equation}
For $t\in\mathbb R$, $\mathsf H_{N,L}^G(t)$ is positive semidefinite and
has rank at most two.  For arbitrary $z$, the paired operator
\begin{equation}\label{eq:operator-realization-paired-definition}
  \mathsf K_{N,L}^G(z)
  :=\mathsf H_{N,L}^G(z)+\mathsf H_{N,L}^G(\overline z)
  =\mathsf H_{N,L}^G(z)+\bigl(\mathsf H_{N,L}^G(z)\bigr)^*
\end{equation}
is self-adjoint, is represented by a real-symmetric matrix in the Fourier
basis, and has rank at most four.  It need not be positive semidefinite when
$z\notin\mathbb R$.
\end{proposition}

\begin{proof}
Each entry in \cref{eq:Gmn-transform} is entire after the removable
singularities are filled in, and it satisfies
$\overline{H_{mn,L}^G(z)}=H_{mn,L}^G(\overline z)$ and
$H_{mn,L}^G(z)=H_{nm,L}^G(z)$.  These identities give
\cref{eq:operator-realization-adjoint}.  For real $t$, the Gram decomposition
\cref{eq:Gmn-transform-Gram-form} proves positivity and the rank bound.  The
claims for $\mathsf K_{N,L}^G(z)$ follow from
\cref{eq:operator-realization-adjoint}; loss of positivity for a nonreal
parameter is the operator form of
\cref{eq:off-line-indefinite-quadratic-form}.
\end{proof}

\subsection{Position-space kernel and resolvent factorization}

For $z\notin\{\nu_{n,L}:n\in I_N\}$, set
\begin{equation}\label{eq:operator-realization-FN-definition}
  F_{N,L}(x;z)
  :=\sum_{n=-N}^{N}
    \frac{e^{i\nu_{n,L}x}}{z^2-\nu_{n,L}^2},
  \qquad
  \mathfrak c_L(z):=\frac{4\sin^2(Lz/2)}{L^3}.
\end{equation}
Because the index set is symmetric,
$F_{N,L}(-x;z)=F_{N,L}(x;z)$ on $\mathbb T_L$.  The factorization below has
apparent poles at the Fourier nodes, although the operator itself is entire;
at such nodes it is understood by continuous extension in $z$.

\begin{proposition}[Integral kernel]
\label{prop:operator-realization-position-kernel}
The position-space kernel of $\mathsf H_{N,L}^G(z)$ is
\begin{align}\label{eq:operator-realization-position-kernel}
  \mathcal K_{N,L}^G(x,y;z)
  &:=\sum_{m,n\in I_N}
       \psi_{m,L}(x)H_{mn,L}^G(z)
       \overline{\psi_{n,L}(y)}
  \notag\\
  &=\mathfrak c_L(z)\left[
       z^2F_{N,L}(x;z)F_{N,L}(y;z)
       +\partial_xF_{N,L}(x;z)\partial_yF_{N,L}(y;z)
     \right].
\end{align}
Moreover,
\begin{equation}\label{eq:operator-realization-FN-equation}
  (-\partial_x^2-z^2)F_{N,L}(x;z)
  =-D_{N,L}(x),
\end{equation}
where
\begin{equation}\label{eq:operator-realization-Dirichlet-kernel}
  D_{N,L}(x)
  :=\sum_{n=-N}^{N}e^{i\nu_{n,L}x}
  =\frac{\sin((N+\tfrac12)2\pi x/L)}{\sin(\pi x/L)}
\end{equation}
is the periodic Dirichlet kernel, with its values at $x\in L\mathbb Z$
understood by continuity.
\end{proposition}

\begin{proof}
Insert \cref{eq:Gmn-transform} into the Fourier reconstruction in the first
line of \cref{eq:operator-realization-position-kernel}.  The $z^2$ term
factors through $F_{N,L}(x;z)F_{N,L}(-y;z)$, while the
$\nu_{m,L}\nu_{n,L}$ term factors through the two derivatives.  Evenness of
$F_{N,L}$ gives the displayed formula.  Applying
$-\partial_x^2-z^2$ termwise in
\cref{eq:operator-realization-FN-definition} gives
\cref{eq:operator-realization-FN-equation}.
\end{proof}

For a rank-one operator we use the standard convention
\[
  |u\rangle\langle v|f:=u\,\langle v,f\rangle_L.
\]
If $t\in\mathbb R$, then $F_{N,L}(\cdot;t)$ and its derivative are real,
and \cref{eq:operator-realization-position-kernel} becomes
\begin{equation}\label{eq:operator-realization-real-rank-two}
  \mathsf H_{N,L}^G(t)
  =\mathfrak c_L(t)\left[
      t^2|F_{N,L}(\cdot;t)\rangle
          \langle F_{N,L}(\cdot;t)|
      +|\partial_xF_{N,L}(\cdot;t)\rangle
          \langle\partial_xF_{N,L}(\cdot;t)|
    \right].
\end{equation}
The first function is even and the second is odd with respect to reflection
about $L/2$, so they are orthogonal.  Hence the two possibly nonzero
eigenvalues are
\begin{align}\label{eq:operator-realization-finite-eigenvalues}
  \mu_{N,L}^{\mathrm e}(t)
  &=\frac{4t^2\sin^2(Lt/2)}{L^2}
    \sum_{n=-N}^{N}\frac{1}{(t^2-\nu_{n,L}^2)^2},
  \\
  \mu_{N,L}^{\mathrm o}(t)
  &=\frac{4\sin^2(Lt/2)}{L^2}
    \sum_{n=-N}^{N}
       \frac{\nu_{n,L}^2}{(t^2-\nu_{n,L}^2)^2}.
\end{align}
The corresponding eigenfunctions are the normalized versions of
$F_{N,L}(\cdot;t)$ and $\partial_xF_{N,L}(\cdot;t)$ whenever these functions
are nonzero.  Formula \cref{eq:operator-realization-finite-eigenvalues} is
again interpreted by continuity at the Fourier nodes.

\begin{remark}[Nonlocality and the meaning of reconstruction]
\label{rem:operator-realization-nonlocality}
The zero extension of $\mathsf H_{N,L}^G(t)$ is a nonzero finite-rank integral
operator and therefore is not a finite-order local differential operator on
$L^2(\mathbb T_L)$.  The factor $F_{N,L}$ is a truncated resolvent vector for
$-\partial_x^2-t^2$, forced by the Dirichlet kernel.  Thus
\cref{eq:operator-realization-position-kernel} reconstructs the canonical
operator represented by the matrix atom; it does not reconstruct the
quotient Hilbert--P\'olya operator from the Weil matrix.
\end{remark}

\subsection{Fixed-length Fourier completion}

We next let $N\to\infty$ while $L$ and $z$ remain fixed.  This is a Fourier
completion on one circle.  It is not the diagonal limit $L=N\to\infty$ used
in \cref{sec:least-arithmetic-Weil-eigenvalue}, and it does not by itself
construct a limiting Hilbert--P\'olya operator.

On the circle $\mathbb T_L$ one has
\begin{equation}\label{eq:operator-realization-Dirichlet-distribution-limit}
  D_{N,L}\longrightarrow L\delta_0
  \qquad\text{in }\mathcal D'(\mathbb T_L).
\end{equation}
Also, for
$z\notin(2\pi/L)\mathbb Z$, the functions in
\cref{eq:operator-realization-FN-definition} converge in
$H^1(\mathbb T_L)$ to the periodic Green function
\begin{equation}\label{eq:operator-realization-Finfty}
  F_{\infty,L}(x;z)
  =\frac{L}{2z}
   \frac{\cos(z(L/2-x))}{\sin(Lz/2)},
  \qquad 0<x<L,
\end{equation}
extended periodically.  It satisfies
\begin{equation}\label{eq:operator-realization-Finfty-equation}
  (-\partial_x^2-z^2)F_{\infty,L}(x;z)
  =-L\delta_0
  \qquad\text{in }\mathcal D'(\mathbb T_L).
\end{equation}
Indeed, the Fourier coefficients are
$(z^2-\nu_{n,L}^2)^{-1}$, and the coefficient tail is summable in the
$H^1$ norm.  Consequently,
\begin{equation}\label{eq:operator-realization-operator-limit}
  \mathsf H_{N,L}^G(z)
  \longrightarrow\mathsf H_{\infty,L}^G(z)
  \qquad\text{in operator norm},
\end{equation}
where the limiting kernel is obtained from
\cref{eq:operator-realization-position-kernel} by replacing $F_{N,L}$ with
$F_{\infty,L}$.  The same conclusion holds in Hilbert--Schmidt norm.

For $t>0$ real, direct integration gives
\begin{align}\label{eq:operator-realization-Finfty-norms}
  \norm{F_{\infty,L}(\cdot;t)}_2^2
  &=\frac{L^2}{8t^3\sin^2(Lt/2)}
    \bigl(tL+\sin(Lt)\bigr),
  \\
  \norm{\partial_xF_{\infty,L}(\cdot;t)}_2^2
  &=\frac{L^2}{8t\sin^2(Lt/2)}
    \bigl(tL-\sin(Lt)\bigr).
\end{align}
Therefore $\mathsf H_{\infty,L}^G(t)$ has the two possibly nonzero
eigenvalues
\begin{equation}\label{eq:operator-realization-Finfty-eigenvalues}
  \mu_{\mathrm e}(t)
  =\frac12\left(1+\frac{\sin(Lt)}{Lt}\right),
  \qquad
  \mu_{\mathrm o}(t)
  =\frac12\left(1-\frac{\sin(Lt)}{Lt}\right),
\end{equation}
and all remaining eigenvalues are zero.  In particular,
\begin{equation}\label{eq:operator-realization-Finfty-trace}
  \operatorname{tr}\mathsf H_{\infty,L}^G(t)=1.
\end{equation}
For $t>0$ away from a removable degeneracy, unit eigenfunctions are
\begin{align}\label{eq:operator-realization-Finfty-eigenfunctions}
  \phi_{\mathrm e,t}(x)
  &:=\left(\frac{2t}{tL+\sin(Lt)}\right)^{1/2}
      \cos(t(L/2-x)),
  \\
  \phi_{\mathrm o,t}(x)
  &:=\left(\frac{2t}{tL-\sin(Lt)}\right)^{1/2}
      \sin(t(L/2-x)).
\end{align}
The formulas extend continuously to the nonzero Fourier nodes.  As $t\to0$,
the even eigenvalue tends to $1$ and the odd eigenvalue tends to $0$.

\subsection{Paired complex parameters and the correct spectral reduction}

Let $z\in\mathbb C\setminus(2\pi/L)\mathbb Z$, and abbreviate
\[
  E_z:=F_{\infty,L}(\cdot;z),
  \qquad
  O_z:=\partial_xF_{\infty,L}(\cdot;z),
  \qquad
  a_{\mathrm e}(z):=\mathfrak c_L(z)z^2,
  \qquad
  a_{\mathrm o}(z):=\mathfrak c_L(z).
\]
Since $E_{\overline z}=\overline{E_z}$ and
$O_{\overline z}=\overline{O_z}$, the non-self-adjoint atom has the standard
Hilbert-space representation
\begin{equation}\label{eq:operator-realization-complex-rank-two}
  \mathsf H_{\infty,L}^G(z)
  =a_{\mathrm e}(z)|E_z\rangle\langle E_{\overline z}|
   +a_{\mathrm o}(z)|O_z\rangle\langle O_{\overline z}|.
\end{equation}
The use of $\langle E_{\overline z}|$, rather than
$\langle E_z|$, is essential: the kernel in
\cref{eq:operator-realization-position-kernel} is bilinear in the two copies
of $F_{\infty,L}(\cdot;z)$, not sesquilinear.  Formula
\cref{eq:operator-realization-complex-rank-two} also makes
$(\mathsf H_{\infty,L}^G(z))^*=\mathsf H_{\infty,L}^G(\overline z)$
transparent.

The even and odd functions are mutually orthogonal.  Hence the paired
self-adjoint operator
\begin{equation}\label{eq:operator-realization-complex-pair}
  \mathsf K_{\infty,L}^G(z)
  :=\mathsf H_{\infty,L}^G(z)
    +\mathsf H_{\infty,L}^G(\overline z)
\end{equation}
reduces to one two-dimensional problem in
$\operatorname{span}\{E_z,E_{\overline z}\}$ and one in
$\operatorname{span}\{O_z,O_{\overline z}\}$.  It is therefore unnecessary
to solve an undifferentiated quartic problem.

\begin{theorem}[Two quadratic spectral problems]
\label{thm:operator-realization-complex-pair-spectrum}
For $p\in\{\mathrm e,\mathrm o\}$, put
\[
  u_{\mathrm e}:=E_z,
  \qquad
  u_{\mathrm o}:=O_z,
  \qquad
  \alpha_p:=\norm{u_p}_2^2,
  \qquad
  \beta_p:=\langle u_p,\overline{u_p}\rangle_L,
\]
and define
\begin{equation}\label{eq:operator-realization-complex-Delta-r}
  \Delta_p:=\alpha_p^2-|\beta_p|^2\geq0,
  \qquad
  r_p:=\RePart\!\left(a_p(z)\overline{\beta_p}\right).
\end{equation}
The two eigenvalues contributed by the parity-$p$ subspace are
\begin{equation}\label{eq:operator-realization-complex-pair-eigenvalues}
  \kappa_{p,\pm}(z)
  =r_p\mathbin{\pm}
   \sqrt{r_p^2+|a_p(z)|^2\Delta_p}.
\end{equation}
Thus
\begin{equation}\label{eq:operator-realization-complex-pair-sum-product}
  \kappa_{p,+}+\kappa_{p,-}=2r_p,
  \qquad
  \kappa_{p,+}\kappa_{p,-}=-|a_p(z)|^2\Delta_p.
\end{equation}
If $\Delta_p>0$, the parity-$p$ block has one positive and one negative
eigenvalue.
\end{theorem}

\begin{proof}
Let $U_p:\mathbb C^2\to L^2(\mathbb T_L)$ have columns
$u_p,\overline{u_p}$.  On its range the parity-$p$ part of
\cref{eq:operator-realization-complex-pair} is
$U_pC_pU_p^*$, where
\[
  C_p=
  \begin{pmatrix}
    0&a_p(z)\\
    \overline{a_p(z)}&0
  \end{pmatrix},
  \qquad
  G_p:=U_p^*U_p=
  \begin{pmatrix}
    \alpha_p&\beta_p\\
    \overline{\beta_p}&\alpha_p
  \end{pmatrix}.
\]
Its nonzero eigenvalues are the eigenvalues of $C_pG_p$.  The characteristic
polynomial is
\[
  \kappa^2-2r_p\kappa-|a_p(z)|^2\Delta_p,
\]
which proves \cref{eq:operator-realization-complex-pair-eigenvalues}.
Cauchy--Schwarz gives $\Delta_p\geq0$.
\end{proof}

The same two-block reduction holds at finite $N$ after making the
replacements
\[
  E_z\mapsto F_{N,L}(\cdot;z),
  \qquad
  O_z\mapsto\partial_xF_{N,L}(\cdot;z).
\]
The four Gram quantities are then finite Fourier sums.

All quantities in this theorem are explicit.  Put
\begin{equation}\label{eq:operator-realization-Sigma-pm}
  \Sigma_{\pm}(z)
  :=\frac{\sin(\tfrac L2(z-\overline z))}{z-\overline z}
    \mathbin{\pm}
    \frac{\sin(\tfrac L2(z+\overline z))}{z+\overline z},
\end{equation}
with the quotients interpreted continuously when a denominator vanishes.
Then
\begin{align}\label{eq:operator-realization-complex-alpha}
  \alpha_{\mathrm e}
  &=\frac{L^2}{4|z\sin(Lz/2)|^2}\Sigma_+(z),
  \\
  \alpha_{\mathrm o}
  &=\frac{L^2}{4|\sin(Lz/2)|^2}\Sigma_-(z),
\end{align}
while
\begin{align}\label{eq:operator-realization-complex-beta}
  \beta_{\mathrm e}
  &=\overline{
      \frac{L^2(zL+\sin(Lz))}
           {8z^3\sin^2(Lz/2)}},
  \\
  \beta_{\mathrm o}
  &=\overline{
      \frac{L^2(zL-\sin(Lz))}
           {8z\sin^2(Lz/2)}}.
\end{align}
For $z=\gamma+i\eta$ with $\gamma,\eta>0$,
\begin{equation}\label{eq:operator-realization-Sigma-offline}
  \Sigma_{\pm}(z)
  =\frac{\sinh(\eta L)}{2\eta}
   \mathbin{\pm}\frac{\sin(\gamma L)}{2\gamma}.
\end{equation}

The two right eigenvalues of the unpaired, generally nonnormal operator
$\mathsf H_{\infty,L}^G(z)$ are
\begin{equation}\label{eq:operator-realization-complex-unpaired-eigenvalues}
  \theta_{\mathrm e}(z)
  =\frac12\left(1+\frac{\sin(Lz)}{Lz}\right),
  \qquad
  \theta_{\mathrm o}(z)
  =\frac12\left(1-\frac{\sin(Lz)}{Lz}\right),
\end{equation}
with right eigenfunctions $E_z$ and $O_z$.  Since
$\overline{\beta_p}=\int_0^L u_p(x)^2\dd x$, one has
$r_p=\RePart\theta_p(z)$.  Nevertheless, the eigenvalues in
\cref{eq:operator-realization-complex-pair-eigenvalues} are not obtained by
simply replacing $\theta_p(z)$ with
$\theta_p(z)+\overline{\theta_p(z)}$.  Only their sum equals
$2\RePart\theta_p(z)$; the additional Gram determinant $\Delta_p$ controls
the splitting.

If $z=\gamma+i\eta$ with $\gamma\eta\neq0$, then $E_z$ is not proportional
to $E_{\overline z}$ and $O_z$ is not proportional to
$O_{\overline z}$.  Hence $\Delta_{\mathrm e},\Delta_{\mathrm o}>0$, and the
paired fixed-$L$ operator has inertia $(2,2)$ and rank four.
The complete quartet contribution in
\cref{eq:zero-side-finite-matrix-definition-2} is twice this paired atom, so
the scalar factor does not change its rank or inertia.  This gives an
operator-level explanation of the indefiniteness of the paired
off-critical-line matrix.  In the real case the two conjugate
parameters coalesce, $\Delta_p=0$, and
$\mathsf K_{\infty,L}^G(t)=2\mathsf H_{\infty,L}^G(t)$ has nonzero
eigenvalues
\begin{equation}\label{eq:operator-realization-real-pair-eigenvalues}
  1+\frac{\sin(Lt)}{Lt},
  \qquad
  1-\frac{\sin(Lt)}{Lt}.
\end{equation}
The purely imaginary case is also exceptional because
$\overline z=-z$ and $\mathbf H_L^G$ is even in $z$; the paired operator then
collapses to twice one rank-two atom and is generally indefinite rather than
positive.

\subsection{Relation to the finite Hilbert--P\'olya construction}

Under RH, a positive zero ordinate $\gamma$ contributes the positive rank-two
atom $\mathsf H_{N,L}^G(\gamma)$ to the zero-side Weil metric.  After the
contributions from all relevant ordinates are summed, the resulting matrix is
compressed to the zero-mean contrast space to form the Hermitian pencil
$(\mathbf K_{N,L},\mathbf G_{N,L})$ of
\cref{sec:scale-invariant-quotient}.  It is the generalized spectrum of this
pencil, not an individual atom and not the raw Weil matrix itself, that is
intended to recover the ordinates.

Accordingly, the fixed-$L$ completion
\cref{eq:operator-realization-operator-limit} supplies an exact operator
interpretation of the matrix entries and a useful local spectral calculation,
but it does not settle the joint arithmetic limit.  A limiting
Hilbert--P\'olya theorem must still control the sum over spectral atoms through
the prime-side formula, the compressed quotient metric, and the relative
perturbation of the entire generalized pencil as $N$ and $L$ tend to infinity
in a coupled manner; the scale-free target is
\cref{eq:siq-relative-transfer-target}.

\section{Positive shifted pencils for paired off-line zeros}
\label{sec:off-critical-shifted-ground-state}

The interpolation pencil of
\cref{thm:exact-off-line-reconstruction} is nonsingular but indefinite, and
its generalized spectrum is the prescribed nonreal quartet.  A different
construction begins by subtracting the algebraically smallest eigenvalue of
the real symmetric paired matrix.  The resulting contrast pencil is positive
definite and has a real sign-paired spectrum.  This operation changes the
spectral problem; it does not replace each nonreal zero by its modulus.

\subsection{The shifted contrast pencil}

Write
\[
  \mathbf S^{\mathrm{off}}
  :=\mathbf S_{N,L}^{(T,\mathrm{off})}
\]
and arrange its eigenvalues as
\begin{equation}\label{eq:off-shift-ordered-eigenvalues}
  \varepsilon_1^{\mathrm{off}}
  \leq\varepsilon_2^{\mathrm{off}}
  \leq\cdots\leq\varepsilon_{2N+1}^{\mathrm{off}}.
\end{equation}
Put
\begin{equation}\label{eq:off-shift-W-definition}
  \varepsilon_{N,L}^{\mathrm{off}}
  :=\varepsilon_1^{\mathrm{off}},
  \qquad
  \mathbf W_{N,L}^{\mathrm{off}}
  :=\mathbf S^{\mathrm{off}}
    -\varepsilon_{N,L}^{\mathrm{off}}\mathbf I_{2N+1}.
\end{equation}
Then $\mathbf W_{N,L}^{\mathrm{off}}\succeq0$.  If the least eigenvalue is
simple, its kernel is one-dimensional.  The gap-normalized version is
\begin{equation}\label{eq:off-shift-gap-normalization}
  \widetilde{\mathbf W}_{N,L}^{\mathrm{off}}
  :=\frac{\mathbf W_{N,L}^{\mathrm{off}}}
          {\varepsilon_2^{\mathrm{off}}
           -\varepsilon_1^{\mathrm{off}}}.
\end{equation}
This normalization is invariant under every positive affine transformation of
$\mathbf S^{\mathrm{off}}$ and leaves the generalized quotient spectrum
unchanged.

Let $\mathbf C_{N,L}$ be the fixed zero-mean contrast matrix in
\cref{eq:siq-contrast-matrix}.  Define
\begin{align}\label{eq:off-shift-GK-pencil}
  \mathbf G_{N,L}^{\mathrm{off}}
  &:=\mathbf C_{N,L}^{*}
      \mathbf W_{N,L}^{\mathrm{off}}\mathbf C_{N,L},
  \\
  \mathbf K_{N,L}^{\mathrm{off}}
  &:=\mathbf C_{N,L}^{*}
      \mathbf W_{N,L}^{\mathrm{off}}
      \mathbf D_{L,N}\mathbf C_{N,L}.
\end{align}
The divided-difference structure is preserved by the scalar shift, so
\begin{equation}\label{eq:off-shift-rank-two-commutator}
  [\mathbf D_{L,N},\mathbf W_{N,L}^{\mathrm{off}}]
  =\alpha_L\left(
      \boldsymbol\beta_{N,L}^{\mathrm{off}}
      \boldsymbol\delta_{N,L}^{\mathsf T}
      -\boldsymbol\delta_{N,L}
       (\boldsymbol\beta_{N,L}^{\mathrm{off}})^{\mathsf T}
    \right)
\end{equation}
for the corresponding off-line divided-difference sequence.  Hence
$\mathbf K_{N,L}^{\mathrm{off}}$ is real symmetric.

\begin{theorem}[Positive shifted off-line pencil]
\label{thm:off-shifted-metric-self-adjointness}
Assume
\begin{equation}\label{eq:off-shift-pencil-hypotheses}
  \varepsilon_1^{\mathrm{off}}
  \text{ is simple},
  \qquad
  \mathbf G_{N,L}^{\mathrm{off}}\succ0.
\end{equation}
Then the generalized eigenproblem
\begin{equation}\label{eq:off-shift-generalized-eigenproblem}
  \mathbf K_{N,L}^{\mathrm{off}}\mathbf y
  =\mu\,\mathbf G_{N,L}^{\mathrm{off}}\mathbf y
\end{equation}
is Hermitian definite and all of its eigenvalues are real.  If the paired
matrix is centrosymmetric, the spectrum is invariant under
$\mu\mapsto-\mu$.  In the full-rank parity case there are numbers
$0<\mu_{1,N,L}^{\mathrm{off}}\leq\cdots\leq
\mu_{N,N,L}^{\mathrm{off}}$ such that
\begin{equation}\label{eq:off-shift-real-sign-paired-spectrum}
  \operatorname{spec}
  (\mathbf K_{N,L}^{\mathrm{off}},
   \mathbf G_{N,L}^{\mathrm{off}})
  =\{-\mu_{N,N,L}^{\mathrm{off}},\ldots,
      -\mu_{1,N,L}^{\mathrm{off}},
       \mu_{1,N,L}^{\mathrm{off}},\ldots,
       \mu_{N,N,L}^{\mathrm{off}}\}.
\end{equation}
The positive-parity squared matrix obtained from the whitened off-diagonal
block has spectrum
\begin{equation}\label{eq:off-shift-positive-square-spectrum}
  \{(\mu_{1,N,L}^{\mathrm{off}})^2,\ldots,
    (\mu_{N,N,L}^{\mathrm{off}})^2\}.
\end{equation}
\end{theorem}

\begin{proof}
The proof is the same contrast-space argument as
\cref{thm:scale-invariant-quotient}.  Positivity of
$\mathbf G_{N,L}^{\mathrm{off}}$ makes the pencil definite, while
\cref{eq:off-shift-rank-two-commutator} makes its differentiation form
symmetric.  Reflection preserves the zero-mean contrast space, commutes with
the metric, and anticommutes with the frequency matrix.  A parity-adapted
whitening therefore gives an off-diagonal Hermitian matrix, whose eigenvalues
are the signed singular values of its parity-changing block.
\end{proof}

\subsection{Comparison with the interpolation-null pencil}

In the dimension-matched case $N=2K$, the two constructions use the same
fixed contrast space but different metrics:
\begin{align}\label{eq:off-shift-two-pencils}
  (\mathbf G_{N,L}^{\mathrm{off},0},
   \mathbf K_{N,L}^{\mathrm{off},0})
  &\quad\text{uses the unshifted interpolation matrix},
  \\
  (\mathbf G_{N,L}^{\mathrm{off}},
   \mathbf K_{N,L}^{\mathrm{off}})
  &\quad\text{uses the algebraically shifted positive matrix}.
\end{align}
The first pencil has inertia $(2K,2K)$ and reconstructs
$\{\pm z_k,\pm\overline z_k\}$.  The second is positive definite and hence
has a real spectrum.  Unless the least eigenvalue is zero and the two metrics
coincide, the pencils are different and need not be similar.

In particular, the shifted eigenvalues depend on the complete paired matrix;
there is no universal formula reducing them to
$\sqrt{\gamma_k^2+\eta_k^2}$.  The reality of
\cref{eq:off-shift-real-sign-paired-spectrum} comes from the newly created
positive metric, not from radialization of the nonreal zeros.

\subsection{Five-dimensional numerical example}

For the example
\[
  N=2,
  \qquad L=3,
  \qquad z_1=5+0.2i,
\]
one has
\begin{align}\label{eq:off-shift-N2-S-spectrum}
  \operatorname{spec}\mathbf S_{2,3}^{(T,\mathrm{off})}
  \approx\{&-0.00677868721833,-0.000988865435683,0,\notag\\
            &1.01559723607200,1.70573429868628\}.
\end{align}
Thus
\begin{equation}\label{eq:off-shift-N2-ground-energy}
  \varepsilon_{2,3}^{\mathrm{off}}
  \approx-0.00677868721833123.
\end{equation}
The positive shifted contrast pencil gives
\begin{align}\label{eq:off-shift-N2-real-quotient-spectrum}
  \operatorname{spec}
  (\mathbf K_{2,3}^{\mathrm{off}},
   \mathbf G_{2,3}^{\mathrm{off}})
  \approx\{&-5.056960388720,-1.217895631397,\notag\\
            & 1.217895631397,5.056960388720\},
\end{align}
and the positive squared values are
\begin{equation}\label{eq:off-shift-N2-positive-square-numerical}
  \{1.483269768977,25.572848373080\}.
\end{equation}
By comparison,
$|5+0.2i|\approx5.003998401279$ and $|5+0.2i|^2=25.04$.
The shifted pencil restores positivity and real spectral symmetry, but it does
so by changing the metric rather than by replacing the quartet with its
radial modulus.

\section{Numerical verification protocol}
\label{sec:numerical-verification-protocol}

This section gives a reproducible implementation of the arithmetic Weil
matrix and the scale-invariant quotient pencil.  The central numerical change from Lemke's ground-vector formulation
\cite{Lemke2026} is that the primary calculation forms no normalized ground
state and no oblique projection.  The quotient is computed from fixed
zero-mean contrast matrices and a Hermitian definite generalized eigenproblem.

\subsection{Condensed procedure}

Given $L>0$ and $N\in\mathbb N$, perform the following steps.

\begin{enumerate}
  \item Set $T=e^L$, $I_N=\{-N,\ldots,N\}$, and
        $\nu_{n,L}=2\pi n/L$.
  \item Generate all prime powers $q=p^r\leq T$, storing
        $\Lambda(q)=\log p$.
  \item Evaluate $a_{n,L}$ and $b_{n,L}$ from the compensated formulas
        \cref{eq:an-bn-compensated-Weil-formula}, using the closed
        archimedean expressions and certified Lerch-series truncation bounds.
  \item Construct $\mathbf S_{N,L}$ from
        \cref{eq:S-divided-difference-structure} and verify symmetry,
        centrosymmetry, and the rank-two displacement identity.
  \item Compute the two least eigenvalues
        $\lambda_1<\lambda_2$ and form the gap-normalized metric
        $\widetilde{\mathbf W}_{N,L}$ in
        \cref{eq:siq-gap-normalization}.
  \item Construct Euclidean-orthonormal even and odd zero-mean contrast
        matrices $\mathbf C_{N,L}^{+}$ and $\mathbf C_{N,L}^{-}$.
  \item Form the parity metrics $\mathbf G_{N,L}^{\pm}$ and the
        parity-changing block $\mathbf B_{N,L}$ in
        \cref{eq:siq-parity-G-B}.  Verify positive definiteness and
        conditioning of both metric blocks.
  \item Cholesky-whiten the block as in
        \cref{eq:siq-parity-whitened-block}, compute its singular values
        $\mu_{j,N,L}$, and compare the first selected values with the zeta
        ordinates.
  \item For benchmark parameters, independently compute Lemke's normalized
        ground-state quotient \cite{Lemke2026} and verify that its positive
        spectrum agrees with the contrast-pencil spectrum to the resolved
        precision.
  \item Verify the generalized-eigenpair, Hermiticity, sign-pairing, and
        positive-square residuals.  Optionally repeat the calculation for the
        unshifted pencil in \cref{eq:siq-unshifted-GK}.
  \item Repeat for a prescribed sequence of $(N,L)$ and report every failed
        test together with the working precision.
\end{enumerate}

\subsection{Precision and scale conventions}

Let $\mathrm{wp}$ be the working precision, and choose accuracy and precision
goals below $\mathrm{wp}$ by a substantial guard margin.  All matrix
residuals should be normalized.  For a matrix $X$, write
\begin{equation}\label{eq:num-normalized-matrix-scale}
  d(X):=\max\{1,\norm{X}_{\mathrm F}\}.
\end{equation}
A convenient residual threshold is
\begin{equation}\label{eq:num-global-residual-threshold}
  \tau_{\mathrm{res}}
  :=10^{-\min\{25,\lfloor\mathrm{wp}/4\rfloor\}}.
\end{equation}
Eigenvalue-positivity and separation tests must instead use scale-aware
thresholds based on backward error.  No fixed decimal cutoff should be used
for quantities whose natural scale changes rapidly with $N$.

The quotient spectrum is unchanged when
$\mathbf S\mapsto a\mathbf S+b\mathbf I$ with $a>0$.  Therefore all quotient
calculations should use the gap-normalized metric
\begin{equation}\label{eq:num-gap-normalized-W}
  \widetilde{\mathbf W}
  :=\frac{\mathbf S-\lambda_1\mathbf I}{\lambda_2-\lambda_1}.
\end{equation}
This prevents an arbitrary matrix scale from entering the conditioning report.

\subsection{Construction and first diagnostics for the Weil matrix}

The prime-power, pole, and archimedean computations are performed exactly as
specified in
\cref{subsec:closed-archimedean-evaluation,subsec:pole-prime-compensation}.
The matrix is
\begin{equation}\label{eq:num-S-construction-v7}
  S_{mn,L}
  =\begin{cases}
     a_{n,L},&m=n,\\[1mm]
     \dfrac{b_{m,L}-b_{n,L}}{m-n},&m\neq n.
   \end{cases}
\end{equation}
Let $\mathbf R_N$ be the reversal matrix and put
$\mathbf D_{L,N}=\operatorname{diag}(\nu_{n,L})$.  The first residuals are
\begin{align}\label{eq:num-S-structural-residuals-v7}
  r_{\mathrm{sym}}
  &:=\frac{\norm{\mathbf S-\mathbf S^{\mathsf T}}_{\mathrm F}}
           {d(\mathbf S)},
  \\
  r_{\mathrm{cent}}
  &:=\frac{\norm{\mathbf R_N\mathbf S\mathbf R_N-\mathbf S}_{\mathrm F}}
           {d(\mathbf S)},
  \\
  r_{\mathrm{comm}}
  &:=\frac{\norm{
      [\mathbf D_{L,N},\mathbf S]
      -\alpha_L(
        \boldsymbol\beta\boldsymbol\delta^{\mathsf T}
        -\boldsymbol\delta\boldsymbol\beta^{\mathsf T})
      }_{\mathrm F}}
      {d(\mathbf S)}.
\end{align}
All three should be below $\tau_{\mathrm{res}}$.

Let $\lambda_1\leq\lambda_2\leq\cdots$ be the computed eigenvalues.  Put
\begin{equation}\label{eq:num-ground-gap-v7}
  g_\lambda:=\lambda_2-\lambda_1,
  \qquad
  u_\lambda:=C_{\mathrm{eig}}\,\varepsilon_{\mathrm{wp}}
     \max\{1,\norm{\mathbf S}_{\mathrm{op}}\},
  \qquad
  q_\lambda:=\frac{g_\lambda}{u_\lambda},
\end{equation}
where $\varepsilon_{\mathrm{wp}}$ is the arithmetic unit roundoff and
$C_{\mathrm{eig}}$ is a stated safety factor.  Report these quantities
together with the backward residuals of the first two eigenpairs, and require
$q_\lambda>1$.  The scale-invariant quotient requires a numerically resolved
simple least eigenvalue, not a prescribed lower bound on $\lambda_1$ itself.
In particular, the proved limit $\lambda_1(N,N)\to0$ under RH is compatible
with a stable quotient as long as the gap and compressed metric remain
resolved.

\subsection{Parity-adapted contrast matrices}

The odd contrast space has the explicit orthonormal basis
\begin{equation}\label{eq:num-odd-contrast-basis}
  \mathbf c_k^{-}
  :=\frac{\mathbf u_k-\mathbf u_{-k}}{\sqrt2},
  \qquad 1\leq k\leq N.
\end{equation}
For the even sector, first use
\begin{align}\label{eq:num-even-ambient-basis}
  \mathbf h_0&:=\mathbf u_0,\\
  \mathbf h_k&:=\frac{\mathbf u_k+\mathbf u_{-k}}{\sqrt2},
  \qquad1\leq k\leq N.
\end{align}
Let
\[
  \mathcal E_N:=\operatorname{span}
  \{\mathbf h_0,\ldots,\mathbf h_N\}.
\]
Then compute an orthonormal basis of the $N$-dimensional subspace of
$\mathcal E_N$ orthogonal to $\boldsymbol\delta_{N,L}$.  Assemble these bases into matrices
$\mathbf C_+$ and $\mathbf C_-$.  Verify
\begin{align}\label{eq:num-contrast-residuals}
  r_{C,\pm}
  &:=\norm{\mathbf C_\pm^{*}\mathbf C_\pm-\mathbf I}_{\mathrm F},
  \\
  r_{\delta,\pm}
  &:=\norm{\mathbf C_\pm^{*}\boldsymbol\delta}_{2},
  \\
  r_{R,+}
  &:=\norm{\mathbf R_N\mathbf C_+-\mathbf C_+}_{\mathrm F},
  \\
  r_{R,-}
  &:=\norm{\mathbf R_N\mathbf C_-+\mathbf C_-}_{\mathrm F}.
\end{align}

\subsection{The shifted scale-invariant pencil}

Using $\widetilde{\mathbf W}$ from \cref{eq:num-gap-normalized-W}, form
\begin{align}\label{eq:num-parity-pencil-v7}
  \mathbf G_+
  &:=\mathbf C_+^{*}\widetilde{\mathbf W}\mathbf C_+,
  &
  \mathbf G_-
  &:=\mathbf C_-^{*}\widetilde{\mathbf W}\mathbf C_-,
  \\
  \mathbf B
  &:=\mathbf C_-^{*}\widetilde{\mathbf W}
      \mathbf D_{L,N}\mathbf C_+.
\end{align}
At working precision, replace each metric block by its Hermitian part.  Report
\begin{equation}\label{eq:num-metric-conditioning-v7}
  \lambda_{\min}(\mathbf G_+),
  \quad
  \lambda_{\min}(\mathbf G_-),
  \quad
  \kappa(\mathbf G_+),
  \quad
  \kappa(\mathbf G_-).
\end{equation}
The quotient passes the definiteness test only when both least eigenvalues are
positive by more than their estimated backward errors.

Let
$\mathbf G_+=\mathbf L_+^{*}\mathbf L_+$ and
$\mathbf G_-=\mathbf L_-^{*}\mathbf L_-$ be Cholesky factorizations.  Define
\begin{equation}\label{eq:num-whitened-parity-block-v7}
  \mathbf A
  :=\mathbf L_-^{-*}\mathbf B\mathbf L_+^{-1}.
\end{equation}
Its singular values, arranged increasingly, are
$\mu_{1,N,L},\ldots,\mu_{N,N,L}$.  Equivalently, the whitened full quotient
matrix is
\begin{equation}\label{eq:num-whitened-full-v7}
  \widehat{\mathbf A}
  :=\begin{pmatrix}0&\mathbf A^{*}\\\mathbf A&0\end{pmatrix}.
\end{equation}
The positive square is
\begin{equation}\label{eq:num-positive-square-v7}
  \mathbf S_{N,L}^{+,\mathrm{siq}}:=\mathbf A^{*}\mathbf A.
\end{equation}
No lowest squared state is deleted in the primary comparison with the zeta
ordinates.

\subsection{Pencil and spectral residuals}

For each generalized eigenpair $(\mu_j,\mathbf y_j)$ of the full parity
pencil, use the relative residual
\begin{equation}\label{eq:num-generalized-eigen-residual}
  r_{\mathrm{gev},j}
  :=\frac{
    \norm{\mathbf K\mathbf y_j-
           \mu_j\mathbf G\mathbf y_j}_2}
   {\bigl(\norm{\mathbf K}_{\mathrm{op}}
      +|\mu_j|\norm{\mathbf G}_{\mathrm{op}}\bigr)
      \norm{\mathbf y_j}_2}.
\end{equation}
Here $\mathbf G=\operatorname{diag}(\mathbf G_+,\mathbf G_-)$ and
$\mathbf K=\bigl(\begin{smallmatrix}0&\mathbf B^{*}\\
\mathbf B&0\end{smallmatrix}\bigr)$.  Also report
\begin{align}\label{eq:num-pencil-residuals-v7}
  r_{K}
  &:=\frac{\norm{\mathbf K-\mathbf K^{*}}_{\mathrm F}}{d(\mathbf K)},
  \\
  r_{\mathrm{sign}}
  &:=\max_{1\leq j\leq N}
    \frac{|\mu_j^{+}+\mu_j^{-}|}
         {\max\{1,|\mu_j^{+}|,|\mu_j^{-}|\}},
  \\
  r_{\mathrm{sq}}
  &:=\max_j
    \frac{|\sigma_j(\mathbf A)^2-
            \lambda_j(\mathbf A^{*}\mathbf A)|}
         {\max\{1,\sigma_j(\mathbf A)^2\}}.
\end{align}
These quantities test the quotient algebra independently of agreement with
zeta zeros.

For selected positive eigenvalues, report
\begin{equation}\label{eq:num-zeta-comparison-v7}
  e_j^{\mathrm{abs}}:=|\mu_{j,N,L}-\gamma_j|,
  \qquad
  e_j^{\mathrm{rel}}:=\frac{|\mu_{j,N,L}-\gamma_j|}{\gamma_j}.
\end{equation}
Use full-precision values of $\gamma_j$, not ordinates rounded to four decimal
places.

\subsection{Optional unshifted and zero-side checks}

The unshifted arithmetic regularization is obtained by replacing
$\widetilde{\mathbf W}$ in \cref{eq:num-parity-pencil-v7} with
$\mathbf S_{N,L}$.  Report its positive generalized eigenvalues alongside the
shifted values.  Agreement between the two as $N,L$ grow is a numerical test
of the perturbative comparison in
\cref{eq:siq-shifted-unshifted-difference}; it is not assumed a priori.

For an exact zero-side check, choose $K=N$ real ordinates, construct
\begin{equation}\label{eq:num-zero-side-v7}
  \mathbf S_{N,L}^{[N]}
  :=2\sum_{k=1}^{N}\mathbf H_L^G(\gamma_k),
\end{equation}
and apply the same contrast-pencil procedure with no scalar shift.  Verify
\begin{equation}\label{eq:num-zero-side-spectrum-v7}
  \operatorname{spec}
  (\mathbf K_{N,L}^{(0)},\mathbf G_{N,L}^{(0)})
  =\{\pm\gamma_1,\ldots,\pm\gamma_N\}
\end{equation}
and the null-vector residual in
\cref{eq:zero-side-ground-normalization-nullity}.  This tests the entire
implementation against the exact reconstruction theorem.

\subsection{Moving-dimension table}

Repeat the calculation for the prescribed pairs $(N,L)$, for example
$N=L\in\{2,5,7,11,13\}$.  The table should contain at least
\begin{equation}\label{eq:num-moving-table-columns-v7}
\begin{gathered}
  N,\quad L,\quad \lambda_1,\quad \lambda_2-\lambda_1,\quad
  \lambda_{\min}(\mathbf G_+),\quad
  \lambda_{\min}(\mathbf G_-),\\
  \kappa(\mathbf G_+),\quad \kappa(\mathbf G_-),\quad
  \mu_1,\quad\mu_2,\quad\mu_3.
\end{gathered}
\end{equation}
The decrease of $\lambda_1$ is expected under RH and is not itself a failure.
The decisive quantities are the resolved gap, definiteness and conditioning
of the contrast metrics, and the generalized-eigenpair residuals.

\subsection{Required pass--fail report}

Every run should report:
\begin{enumerate}
  \item $L,N,T$, working precision, accuracy goal, precision goal, and the
        number of retained prime powers;
  \item $r_{\mathrm{sym}}$, $r_{\mathrm{cent}}$, and
        $r_{\mathrm{comm}}$;
  \item $\lambda_1$, $\lambda_2$, the gap, and the first two eigenpair
        residuals;
  \item all contrast-basis residuals in
        \cref{eq:num-contrast-residuals};
  \item the least eigenvalues and condition numbers of
        $\mathbf G_+$ and $\mathbf G_-$;
  \item $r_K$, all requested generalized-eigenpair residuals,
        $r_{\mathrm{sign}}$, and $r_{\mathrm{sq}}$;
  \item the positive quotient values and their absolute and relative
        differences from the selected zeta ordinates;
  \item when Lemke's quotient is used as a benchmark, the maximum absolute and
        relative differences between its positive eigenvalues and those of the
        contrast pencil;
  \item the corresponding unshifted-pencil values, when that optional test is
        performed;
  \item the exact zero-side residuals, when the optional calibration is
        performed.
\end{enumerate}
A calculation should not be reported as a verification if a failed
definiteness test or a large residual has merely been suppressed.  Every
failure must remain in the output together with the parameters and precision
that produced it.

\section{Conclusion}\label{sec:conclusion}

This paper develops a finite-dimensional Hilbert--P\'olya--Weil framework
starting from the Riemann--$\Xi$ specialization of Weil's explicit formula.
Real polarizations of logarithmic Fourier autocorrelations give a
real-symmetric arithmetic matrix $\mathbf S_{N,L}$ whose entries are computed
from the pole, archimedean, and finite prime-power terms.  Its off-diagonal
entries are divided differences of one generating sequence, its commutator
with the diagonal frequency matrix has rank at most two, and it is invariant
under Fourier reflection.

The principal structural improvement is the scale-invariant quotient
formulation.  After subtracting the algebraically smallest eigenvalue, the
metric and differentiation form are compressed to the fixed zero-mean
contrast space:
\[
  \mathbf G_{N,L}=\mathbf C_{N,L}^{*}\mathbf W_{N,L}\mathbf C_{N,L},
  \qquad
  \mathbf K_{N,L}=\mathbf C_{N,L}^{*}\mathbf W_{N,L}
                   \mathbf D_{L,N}\mathbf C_{N,L}.
\]
The rank-two commutator makes $\mathbf K_{N,L}$ Hermitian.  Whenever the least
eigenvalue is simple and $\mathbf G_{N,L}\succ0$, the generalized eigenproblem
$\mathbf K_{N,L}\mathbf y=\mu\mathbf G_{N,L}\mathbf y$ is Hermitian definite.
Reflection gives a real opposite-sign spectrum and reduces the positive
square to the singular values of one parity block.  Unlike Lemke's
ground-state oblique-projection realization \cite{Lemke2026}, this construction
does not require a ground-state
eigenvector, does not divide by a small evaluation overlap, and does not form
a large-norm projector.  It is exactly invariant under every positive affine
rescaling $\mathbf S\mapsto a\mathbf S+b\mathbf I$.  The reformulation removes
an avoidable instability while displaying genuine metric degeneration through
the spectrum of $\mathbf G_{N,L}$ itself.

At $N=L=13$, the updated arithmetic computation confirms that Lemke's
quotient and the contrast pencil give the same first three positive eigenvalues
to at least sixty-two decimal places, namely the values in
\cref{eq:intro-Lemke-N13-values}.  Their errors against the first three zeta
ordinates are approximately $1.54\times10^{-27}$,
$1.52\times10^{-16}$, and $1.72\times10^{-11}$.  This agreement verifies the
finite-dimensional equivalence of the two realizations; it is not a substitute
for an asymptotic convergence theorem.

The dimension-matched zero-side problem supplies an exact calibration.  For
$N$ distinct positive real ordinates
$0<\gamma_1<\cdots<\gamma_N$, rational interpolation determines the unique
null vector of the $(2N+1)$-dimensional Gram matrix, and the contrast pencil
satisfies
\begin{equation}\label{eq:conclusion-zero-side-characteristic-polynomial}
  \det\!\left(
    z\mathbf G_{N,L}^{(0)}-\mathbf K_{N,L}^{(0)}
  \right)
  =\det(\mathbf G_{N,L}^{(0)})
   \prod_{k=1}^{N}(z^2-\gamma_k^2).
\end{equation}
Consequently, the generalized quotient spectrum is
$\{\pm\gamma_1,\ldots,\pm\gamma_N\}$ and the positive-parity square has the
exact spectrum
\begin{equation}\label{eq:conclusion-zero-side-square-spectrum}
  \{\gamma_1^2,\ldots,\gamma_N^2\}.
\end{equation}
The auxiliary zero mode is absent because the pencil is already written on
the contrast space.  The state with squared energy $\gamma_1^2$ is retained.
Since the ordinates are inputs, this is a reconstruction and consistency
theorem rather than an independent proof of RH.

The interpolation null vector also yields a genuinely arithmetic asymptotic
statement.  Under RH, it annihilates the first $N$ positive zero atoms and
leaves only the positive zero tail.  The resulting Rayleigh identity proves
$\lambda_1(N,N)\to0$, while the product estimate
\cref{eq:least-eigenvalue-product-bound} explains the very rapid decrease in
the moving-dimension computations.  The new quotient algorithm does not
amplify this small eigenvalue by an explicit overlap denominator.  The
remaining issue is instead the relative size of the entire tail in the
compressed metric.

Hypothetical off-critical-line zeros distinguish exact interpolation from
positive-metric spectral theory.  In the dimension-matched indefinite pencil,
a conjugate pair $\gamma_k\pm i\eta_k$ produces the full quartet
\[
  \{\pm(\gamma_k+i\eta_k),\ \pm(\gamma_k-i\eta_k)\},
\]
not the radial values $\pm\sqrt{\gamma_k^2+\eta_k^2}$.  Shifting the paired
matrix by its algebraically smallest eigenvalue creates a positive definite
contrast pencil and hence a real sign-paired spectrum, but it changes the
spectral problem.  At the atom level, the fixed-$L$ Fourier completion gives a
complementary operator interpretation: a real spectral parameter produces a
positive rank-two operator of trace one, whereas a genuine nonreal conjugate
pair produces an indefinite rank-four self-adjoint operator.

The analytic and computational parts of the paper support the finite
construction.  Simultaneous prime-power and zero-tail estimates justify fixed
and uniformly norm-bounded test families.  The archimedean integrals admit
closed digamma--polygamma--Lerch evaluations with certified term-count bounds,
and compensated formulas expose the cancellation of the large pole and prime
contributions.  The revised numerical protocol works directly with
parity-adapted contrast pencils, gap-normalized metrics, Cholesky whitening,
generalized residuals, condition numbers, and a benchmark comparison with
Lemke's finite quotient.  Failed positivity or poor
conditioning is retained in the output rather than hidden by numerical
chopping.

The main unresolved problem is now precise.  One must prove, for a controlled
coupling of $N$ and $L$, that the full arithmetic pencil is a small relative
perturbation of the exact finite-zero pencil.  A natural target is
\cref{eq:siq-relative-transfer-target}, which controls
$\Delta\mathbf K-z\Delta\mathbf G$ in the finite-zero quotient-energy norm.
Together with separation of the reference generalized eigenvalues and uniform
control of the compressed metric, such an estimate would transfer each fixed
low ordinate to the arithmetic side and would provide the correct starting
point for a limiting operator or determinant.  Entrywise convergence of the
uncompressed Weil matrices, or smallness of one Rayleigh quotient, is not
sufficient.

Accordingly, the present work establishes a rigorous finite arithmetic
matrix, a projector-free and affine-scale-invariant quotient mechanism, an
exact zero-side reconstruction theorem, and a reproducible numerical
framework.  It isolates the remaining bottleneck as a relative
prime-to-zero spectral transfer theorem followed by a limiting construction
that preserves multiplicities and excludes spectral pollution.  No proof of
the Riemann hypothesis is claimed.

\section{Acknowledgements and contribution statement}
\label{sec:acknowledgements}

The author thanks S\"oren Lemke for the private communication
\emph{From Truncated Weil Arithmetic to a Finite Real Spectrum}, for the
original ground-state removal and oblique-quotient construction, and for the
numerical implementation that first exhibited the finite real spectrum used
in this manuscript.  The author also thanks the developers of Mathematica and
of the open-source numerical and \LaTeX{} tools used for independent checks.

This manuscript was developed through an extended collaboration between the
author and OpenAI's ChatGPT, GPT--5.6 Pro, during August--September 2026.
ChatGPT contributed materially to algebraic verification, identification and
correction of inconsistent formulas, derivation of the scale-invariant
contrast-pencil realization and several zero-side identities, formulation of
error and conditioning estimates, translation and independent validation of
numerical algorithms, organization of the argument, and drafting and editing
of the \LaTeX{} source.  ChatGPT is not an author and cannot accept
responsibility for the mathematical claims.

The author's principal contribution was conceptual and directive: he selected
the Hilbert--P\'olya--Weil problem, proposed the test-function and matrix
program, supplied heuristic mechanisms and conjectural interpretations, posed
the questions that drove each calculation, specified the numerical
experiments, evaluated competing formulations, and made the final decisions
about the contents of the paper.  Many initial proposals were deliberately
heuristic and question-driven rather than complete proofs.  The author has
reviewed the resulting statements and assumes full responsibility for the
correctness, interpretation, and submission of the manuscript.

\section{Appendix A: Sample test functions}\label{sec:test-functions}
Here we analyze several test functions and the resulting Weil explicit
formulas.  The transform and autocorrelation conventions are recorded first.

\subsection{Transform convention and Weil autocorrelations}\label{subsec:explicit-test-functions}

For all examples below, write
\begin{equation}\label{eq:Fhat-sample-definition}
	\widehat F(t):=\int_{-\infty}^{\infty}F(x)e^{itx}\dd x
	=\Phi_F\!\left(\frac12+it\right).
\end{equation}
This is the Fourier-transform convention used throughout the
Riemann--$\Xi$ specialization.  We also put
\begin{equation}\label{eq:z-shift-test-functions}
	z:=s-\frac12,
	\qquad
	\Phi_F(s)=\int_{\mathbb R}F(x)e^{zx}\dd x.
\end{equation}

It is useful to distinguish two roles for a function $f$.  It may be used
\emph{directly} as the test function in the explicit formula, or it may be
used as the generating function in Weil's autocorrelation construction.  With
\begin{equation}\label{eq:sample-convolution-definition}
	(f*g)(x):=\int_{\mathbb R}f(x-y)g(y)\dd y,
	\qquad
	\widetilde f(x):=\overline{f(-x)},
\end{equation}
set
\begin{equation}\label{eq:sample-autocorrelation-conversion}
	A_f:=f*\widetilde f.
\end{equation}
Then
\begin{align}\label{eq:sample-autocorrelation-transform}
	A_f(x)
	&=\int_{\mathbb R}f(x+y)\overline{f(y)}\dd y,
	\\
	\widehat{A_f}(t)
	&=\abs{\widehat f(t)}^2\geq0,
	\\
	\Phi_{A_f}(s)
	&=\Phi_f(s)\,
	\overline{\Phi_f(1-\overline s)}.
\end{align}
On the critical line $s=\frac12+it$, the last expression is precisely a
squared modulus.  In particular, a direct test function need not have a
nonnegative Fourier transform, whereas an autocorrelation always does.

\subsection{Laguerre functions}\label{subsec:Laguerre-functions}

\subsubsection{The basic half-line Laguerre system}

Let $L_n=L_n^{(0)}$ be the ordinary Laguerre polynomial and let $a>0$.
To incorporate Weil's midpoint convention, put
\[
\mathbf 1_+^*(x)
:=\begin{cases}
	1,&x>0,\\
	\frac12,&x=0,\\
	0,&x<0.
\end{cases}
\]
The correctly scaled version of the proposed function
``$e^{-x}$ times a Laguerre polynomial'' is
\begin{equation}\label{eq:basic-Laguerre-functions}
	u_{n,a}^{+}(x)
	:=\sqrt{2a}\,e^{-ax}L_n(2ax)\,\mathbf 1_+^*(x),
	\qquad n\in\mathbb N_0.
\end{equation}
For each fixed $a$, the family $\{u_{n,a}^{+}\}_{n\geq0}$ is an
orthonormal basis of $L^2(0,\infty)$.  Its transform is the elementary
rational function
\begin{align}\label{eq:basic-Laguerre-transform}
	\Phi_{u_{n,a}^{+}}(s)
	&=\sqrt{2a}\,(-1)^n
	\frac{(a+z)^n}{(a-z)^{n+1}},
	\qquad \RePart z<a,
	\\
	\widehat{u_{n,a}^{+}}(t)
	&=\sqrt{2a}\,(-1)^n
	\frac{(a+it)^n}{(a-it)^{n+1}}.
\end{align}
The reflected functions
$u_{n,a}^{-}(x):=u_{n,a}^{+}(-x)$, together with the positive-half-line
functions, form an orthonormal basis of $L^2(\mathbb R)$.

There are two qualifications.  First, $e^{-x}L_n(x)$ by itself grows as
$x\to-\infty$, so it is not a test function on the whole real line without
a half-line restriction, reflection, or replacement of $x$ by $\abs{x}$.
Second, because $L_n(0)=1$, the half-line extension in
\cref{eq:basic-Laguerre-functions} has a jump at the origin.  The value
$\mathbf 1_+^*(0)=1/2$ makes Weil's midpoint convention explicit.  The function
is admissible under Weil's original piecewise-smooth hypotheses when
$a>\frac12+b$, but it does not belong to
$W_{\mathrm{loc}}^{2,1}(\mathbb R)$ and therefore is not covered by
\cref{thm:uniform-Xi-truncation}.

\subsubsection{A $W_{\mathrm{loc}}^{2,1}$-compatible Laguerre basis}

Let $L_n^{(4)}$ be the generalized Laguerre polynomial of parameter $4$ and
define
\begin{equation}\label{eq:regularized-Laguerre-functions}
	\ell_{n,a}^{+}(x)
	:=\left(\frac{2a\,n!}{(n+4)!}\right)^{1/2}
	(2ax)^2e^{-ax}L_n^{(4)}(2ax)\,
	\mathbf 1_{[0,\infty)}(x).
\end{equation}
The factor $x^2$ makes both the function and its first derivative vanish at
the origin.  Hence its zero extension is $C^1$, has a locally integrable
second weak derivative, and belongs to
$W_{\mathrm{loc}}^{2,1}(\mathbb R)$.  The generalized Laguerre
orthogonality relation gives
\begin{equation}\label{eq:regularized-Laguerre-orthogonality}
	\int_0^\infty
	\ell_{n,a}^{+}(x)\ell_{m,a}^{+}(x)\dd x
	=\delta_{nm}.
\end{equation}
For every $b>0$, $\ell_{n,a}^{+}\in\mathcal A_b^2$ whenever $a>\frac12+b$.
Its Fourier transform is the finite rational sum
\begin{align}\label{eq:regularized-Laguerre-transform}
	\widehat{\ell_{n,a}^{+}}(t)
	={}&\left(\frac{2a\,n!}{(n+4)!}\right)^{1/2}
	\sum_{k=0}^{n}(-1)^k
	\binom{n+4}{n-k}(k+1)(k+2)
	\frac{(2a)^{k+2}}{(a-it)^{k+3}}.
\end{align}
The corresponding formula for $\Phi_{\ell_{n,a}^{+}}(s)$ is obtained by
replacing $it$ with $z=s-\frac12$.  Let
\begin{equation}\label{eq:regularized-Laguerre-reflections}
	\ell_{n,a}^{-}(x):=\ell_{n,a}^{+}(-x),
	\qquad
	\ell_{n,a}^{\mathrm e}
	:=\frac{\ell_{n,a}^{+}+\ell_{n,a}^{-}}{\sqrt2},
	\qquad
	\ell_{n,a}^{\mathrm o}
	:=\frac{\ell_{n,a}^{+}-\ell_{n,a}^{-}}{\sqrt2}.
\end{equation}
Then
\begin{equation}\label{eq:regularized-Laguerre-full-basis}
	\left\{
	\ell_{n,a}^{\mathrm e},\ell_{n,a}^{\mathrm o}:n\geq0
	\right\}
\end{equation}
is an orthonormal basis of $L^2(\mathbb R)$ consisting of functions in
$\mathcal A_b^2$ whenever $a>\frac12+b$.  The even subfamily is an
orthonormal basis of $L^2_{\mathrm{even}}(\mathbb R)$ and is especially
natural for the Riemann $\Xi$-function.  Since
$\widehat{\ell_{n,a}^{-}}(t)
=\widehat{\ell_{n,a}^{+}}(-t)$, its even and odd transforms are obtained
from \cref{eq:regularized-Laguerre-transform} by taking, respectively,
$\sqrt2$ times the real part and $i\sqrt2$ times the imaginary part.

The elementary identities used here are
\begin{equation}\label{eq:Laguerre-Laplace-identity}
	\int_0^\infty e^{-qx}L_n(cx)\dd x
	=\frac{(q-c)^n}{q^{n+1}},
	\qquad \RePart q>0,
\end{equation}
together with the finite expansion
\begin{equation}\label{eq:generalized-Laguerre-finite-expansion}
	L_n^{(4)}(y)
	=\sum_{k=0}^{n}(-1)^k
	\binom{n+4}{n-k}\frac{y^k}{k!}.
\end{equation}

\subsection{Hermite functions}\label{subsec:Hermite-functions}

\subsubsection{The exact $e^{-x^2}$ polynomial--Gaussian family}

Let $\mathsf H_n$ denote the physicists' Hermite polynomial.  The identity
\begin{equation}\label{eq:Hermite-derivative-Gaussian}
	q_n(x):=\mathsf H_n(x)e^{-x^2}
	=(-1)^n\frac{\dd^n}{\dd x^n}e^{-x^2}
\end{equation}
shows immediately that
\begin{equation}\label{eq:Hermite-exact-Gaussian-transform}
	\widehat q_n(t)
	=\sqrt\pi\,(it)^n e^{-t^2/4}.
\end{equation}
In particular, $q_{2k}$ is even and
\begin{equation}\label{eq:Hermite-exact-even-transform}
	\widehat q_{2k}(t)
	=\sqrt\pi\,(-1)^k t^{2k}e^{-t^2/4}.
\end{equation}
This is one of the simplest explicit transforms available for a
polynomial--Gaussian test function.  The family $\{q_n\}$ is not orthogonal
in unweighted $L^2(\mathbb R)$, because the product $q_nq_m$ carries the
weight $e^{-2x^2}$ rather than the Hermite weight $e^{-x^2}$.

The autocorrelation is also explicit.  Fourier inversion applied to
\cref{eq:Hermite-exact-Gaussian-transform} gives
\begin{align}\label{eq:Hermite-exact-Gaussian-autocorrelation}
	(q_n*\widetilde q_n)(x)
	&=\sqrt{\frac{\pi}{2}}\,
	\frac{(-1)^n}{2^n}
	\mathsf H_{2n}\!\left(\frac{x}{\sqrt2}\right)e^{-x^2/2},
	\\
	\widehat{q_n*\widetilde q_n}(t)
	&=\pi t^{2n}e^{-t^2/2}\geq0.
\end{align}
For the even functions $q_{2k}$, one has
$\widetilde q_{2k}=q_{2k}$, and hence
\begin{equation}\label{eq:Hermite-exact-even-convolution}
	(q_{2k}*q_{2k})(x)
	=\sqrt{\frac{\pi}{2}}\,
	2^{-2k}\mathsf H_{4k}\!\left(\frac{x}{\sqrt2}\right)e^{-x^2/2}.
\end{equation}

\subsubsection{The orthonormal Hermite basis}

Define the normalized Hermite functions by
\begin{equation}\label{eq:Hermite-functions}
	h_n(x)
	:=\frac{\mathsf H_n(x)e^{-x^2/2}}
	{\pi^{1/4}\sqrt{2^n n!}},
	\qquad
	h_{n,a}(x):=a^{1/2}h_n(ax),
	\qquad a>0.
\end{equation}
For every fixed $a>0$, $\{h_{n,a}\}_{n\geq0}$ is an orthonormal basis of
$L^2(\mathbb R)$, and
\begin{equation}\label{eq:Hermite-Fourier-transform}
	\widehat{h_{n,a}}(t)
	=\sqrt{2\pi}\,i^n h_{n,1/a}(t)
	=\sqrt{2\pi}\,i^n a^{-1/2}h_n(t/a).
\end{equation}
Moreover, $h_{2k,a}$ is even and $h_{2k+1,a}$ is odd.  Consequently,
$\{h_{2k,a}\}_{k\geq0}$ is an orthonormal basis of
$L^2_{\mathrm{even}}(\mathbb R)$.

Set
\begin{equation}\label{eq:Hermite-autocorrelation-definition}
	\mathcal H_{n,a}:=h_{n,a}*\widetilde h_{n,a}.
\end{equation}
The Hermite correlation integral, equivalently the Fourier eigenfunction
identity followed by inversion, gives
\begin{equation}\label{eq:Hermite-autocorrelation-physical}
	\boxed{
		\mathcal H_{n,a}(x)
		=e^{-a^2x^2/4}
		L_n\!\left(\frac{a^2x^2}{2}\right).}
\end{equation}
Here $L_n=L_n^{(0)}$ is the ordinary Laguerre polynomial.  Since
$L_n(0)=1$,
\begin{equation}\label{eq:Hermite-autocorrelation-at-zero}
	\mathcal H_{n,a}(0)
	=\norm{h_{n,a}}_2^2=1.
\end{equation}
For $n=2k$, the generator is real and even, so
$\mathcal H_{2k,a}=h_{2k,a}*h_{2k,a}$ is an ordinary self-convolution.

The Fourier and bilateral Laplace transforms are
\begin{align}\label{eq:Hermite-autocorrelation-transform}
	\widehat{\mathcal H_{n,a}}(t)
	&=2\pi\abs{h_{n,1/a}(t)}^2
	\\
	&=\frac{2\sqrt\pi}{a\,2^n n!}
	\mathsf H_n\!\left(\frac{t}{a}\right)^2e^{-t^2/a^2}
	\geq0,
	\\
	\Phi_{\mathcal H_{n,a}}(s)
	&=\frac{2\sqrt\pi}{a\,2^n n!}
	\mathsf H_n\!\left(-\frac{iz}{a}\right)^2e^{z^2/a^2}.
\end{align}
Thus the autocorrelation remains a polynomial times a Gaussian on both the
physical and transform sides.

Two scalings are especially convenient.  Taking $a=\sqrt2$ places the
factor $e^{-x^2}$ directly in the orthonormal generating function:
\begin{equation}\label{eq:Hermite-sqrt-two-scaling}
	h_{2k,\sqrt2}(x)
	=\frac{2^{1/4}}
	{\pi^{1/4}\sqrt{2^{2k}(2k)!}}
	\mathsf H_{2k}(\sqrt2\,x)e^{-x^2},
\end{equation}
and
\begin{equation}\label{eq:Hermite-sqrt-two-autocorrelation}
	\mathcal H_{2k,\sqrt2}(x)
	=e^{-x^2/2}L_{2k}(x^2).
\end{equation}
Taking $a=2$ instead places $e^{-x^2}$ in the autocorrelation itself:
\begin{equation}\label{eq:Hermite-two-autocorrelation}
	\boxed{
		\mathcal H_{2k,2}(x)
		=e^{-x^2}L_{2k}(2x^2),}
\end{equation}
with
\begin{equation}\label{eq:Hermite-two-autocorrelation-transform}
	\widehat{\mathcal H_{2k,2}}(t)
	=\frac{\sqrt\pi}{2^{2k}(2k)!}
	\mathsf H_{2k}\!\left(\frac t2\right)^2e^{-t^2/4}
	\geq0.
\end{equation}

\begin{lemma}[Polynomial--Gaussian admissibility]\label{lem:polynomial-Gaussian-admissibility}
	Let $P$ be a polynomial and let $c>0$.  Then
	\begin{equation}\label{eq:polynomial-Gaussian-general}
		G(x):=P(x)e^{-cx^2}
	\end{equation}
	belongs to $\mathcal A_b^2$ for every $b>0$.
\end{lemma}

\begin{proof}
	Put $d_b:=\frac12+b$.  Young's inequality gives
	\begin{equation}\label{eq:Gaussian-linear-absorption}
		d_b\abs{x}
		\leq\frac c2x^2+\frac{d_b^2}{2c}.
	\end{equation}
	Consequently,
	\[
	e^{d_b\abs{x}}\abs{G(x)}
	\leq e^{d_b^2/(2c)}\abs{P(x)}e^{-cx^2/2},
	\]
	and the right-hand side is bounded.  Furthermore, for $j=0,1,2$ there is
	a polynomial $P_j$ such that
	\[
	G^{(j)}(x)=P_j(x)e^{-cx^2}.
	\]
	Using
	\[
	\frac{\abs{x}}2
	\leq\frac c2x^2+\frac{1}{8c},
	\]
	we obtain
	\[
	e^{\abs{x}/2}\abs{G^{(j)}(x)}
	\leq e^{1/(8c)}\abs{P_j(x)}e^{-cx^2/2}\in L^1(\mathbb R).
	\]
	Thus both quantities in \cref{eq:Mb-J2-definitions} are finite.  Since $G$ is
	smooth and both $G$ and $G'$ decrease faster than every fixed exponential,
	Weil's admissibility hypotheses are also satisfied.
\end{proof}

It follows that $q_n$, $q_n*\widetilde q_n$, $h_{n,a}$, and
$\mathcal H_{n,a}$ belong to $\mathcal A_b^2$ for every $a>0$, every
$b>0$, and every fixed $n$.  No condition such as
$a>\frac12+b$ is needed: Gaussian decay dominates every fixed exponential
weight.  This is stronger than the corresponding assertion for merely
exponentially decreasing test functions.

There is one uniformity qualification.  Membership of every individual
$\mathcal H_{n,a}$ does not imply that
$\{\mathcal H_{n,a}:n\geq0\}$ is norm-bounded in $\mathcal A_b^2$.
Uniform application of \cref{thm:uniform-Xi-truncation} while $n\to\infty$
therefore requires separate bounds for
$M_b(\mathcal H_{n,a})+J_2(\mathcal H_{n,a})$.
Standard Hermite and Laguerre identities may be found in \cite{Szego1975}.

\subsection{The \texorpdfstring{$E_a$}{Ea} functions}\label{subsec:Ea-functions}

For $a>0$, define
\begin{equation}\label{eq:smooth-bilateral-exponential}
	g_a(x):=\sqrt a\,e^{-a\abs{x}},
	\qquad
	E_a(x):=(1+a\abs{x})e^{-a\abs{x}}.
\end{equation}
The generator is $L^2$-normalized:
\begin{equation}\label{eq:ga-L2-normalization}
	\norm{g_a}_2^2
	=2a\int_0^\infty e^{-2ax}\dd x=1.
\end{equation}
Because $g_a$ is real and even, $\widetilde g_a=g_a$.  Splitting the
convolution integral at $0$ and at $x$ gives
\begin{equation}\label{eq:smooth-bilateral-autocorrelation}
	(g_a*\widetilde g_a)(x)
	=a\int_{\mathbb R}e^{-a\abs{x-y}}e^{-a\abs{y}}\dd y
	=(1+a\abs{x})e^{-a\abs{x}}
	=E_a(x).
\end{equation}
Thus $E_a(0)=1$ and $E_a$ is already a normalized Weil autocorrelation.
The cusp of $g_a$ is removed by this convolution.  More precisely, the
second distributional derivative of $g_a$ contains a Dirac mass,
\begin{equation}\label{eq:ga-distributional-second-derivative}
	D^2g_a
	=a^{5/2}e^{-a\abs{x}}-2a^{3/2}\delta_0,
\end{equation}
so $g_a\notin W_{\mathrm{loc}}^{2,1}(\mathbb R)$, whereas
$E_a\in C^2(\mathbb R)\subset W_{\mathrm{loc}}^{2,1}(\mathbb R)$.

The transforms of the generator and of its autocorrelation are
\begin{align}\label{eq:ga-and-Ea-transforms}
	\Phi_{g_a}(s)
	&=\frac{2a^{3/2}}{a^2-z^2},
	&
	\widehat g_a(t)
	&=\frac{2a^{3/2}}{a^2+t^2},
	\\
	\Phi_{E_a}(s)
	&=\frac{4a^3}{(a^2-z^2)^2},
	&
	\widehat E_a(t)
	&=\frac{4a^3}{(a^2+t^2)^2}\geq0,
\end{align}
valid for $\abs{\RePart z}<a$ on the bilateral-Laplace side.
For $x\neq0$ one has
\begin{equation}\label{eq:Ea-first-two-derivatives}
	E_a'(x)=-a^2x e^{-a\abs{x}},
	\qquad
	E_a''(x)=a^2(a\abs{x}-1)e^{-a\abs{x}},
\end{equation}
and the second formula extends continuously to $E_a''(0)=-a^2$.
It follows directly from \cref{eq:Ea-first-two-derivatives} that
\begin{equation}\label{eq:Ea-Ab-admissibility}
	E_a\in\mathcal A_b^2
	\qquad\text{whenever}\qquad
	a>\frac12+b.
\end{equation}
The strict inequality is necessary for $M_b(E_a)<\infty$, because of the
linear factor $1+a\abs{x}$.

One may also use $E_a$ itself as a generating function.  A direct integral
gives
\begin{equation}\label{eq:Ea-L2-norm}
	\norm{E_a}_2^2=\frac{5}{2a}.
\end{equation}
Define the normalized function and its autocorrelation by
\begin{equation}\label{eq:Ea-second-generator}
	v_a(x):=\sqrt{\frac{2a}{5}}\,E_a(x),
	\qquad
	K_a:=v_a*\widetilde v_a.
\end{equation}
Squaring the rational transform in \cref{eq:ga-and-Ea-transforms}, or
computing the convolution directly, yields
\begin{align}\label{eq:Ea-self-convolution}
	(E_a*E_a)(x)
	&=\frac{e^{-a\abs{x}}}{6a}
	\left(
	a^3\abs{x}^3+6a^2\abs{x}^2+15a\abs{x}+15
	\right),
	\\
	K_a(x)
	&=\frac{e^{-a\abs{x}}}{15}
	\left(
	a^3\abs{x}^3+6a^2\abs{x}^2+15a\abs{x}+15
	\right).
\end{align}
In particular, $K_a(0)=1$.  The transforms are
\begin{align}\label{eq:Ea-second-autocorrelation-transforms}
	\Phi_{E_a*E_a}(s)
	&=\frac{16a^6}{(a^2-z^2)^4},
	&
	\widehat{E_a*E_a}(t)
	&=\frac{16a^6}{(a^2+t^2)^4},
	\\
	\Phi_{K_a}(s)
	&=\frac{32a^7}{5(a^2-z^2)^4},
	&
	\widehat K_a(t)
	&=\frac{32a^7}{5(a^2+t^2)^4}\geq0.
\end{align}
For example, the first formula in \cref{eq:Ea-self-convolution} follows
from the inverse-transform recurrence
\begin{equation}\label{eq:rational-inverse-transform-recurrence}
	I_1(a,x)=\frac{\pi}{a}e^{-a\abs{x}},
	\qquad
	I_{m+1}(a,x)
	=-\frac{1}{2ma}\frac{\partial}{\partial a}I_m(a,x),
\end{equation}
where
$I_m(a,x)=\int_{\mathbb R}e^{-itx}(a^2+t^2)^{-m}\dd t$.
For every $b>0$, both $E_a*E_a$ and $K_a$ belong to
$\mathcal A_b^2$ whenever $a>\frac12+b$, because their first two derivatives
are again polynomials in $\abs{x}$ times $e^{-a\abs{x}}$.

\subsection{Hyperbolic-secant functions}\label{subsec:sech-functions}

For $a>0$, let
\begin{equation}\label{eq:sech-test-function}
	S_a(x):=\operatorname{sech}(ax).
\end{equation}
The beta integral gives
\begin{align}\label{eq:sech-transform}
	\Phi_{S_a}(s)
	&=\frac{\pi}{a}
	\sec\!\left(\frac{\pi z}{2a}\right),
	\qquad \abs{\RePart z}<a,
	\\
	\widehat S_a(t)
	&=\frac{\pi}{a}
	\operatorname{sech}\!\left(\frac{\pi t}{2a}\right)>0.
\end{align}
Thus $S_a$ may be used directly as a positive-type test function.  Its first
two derivatives are
\begin{align}\label{eq:sech-first-two-derivatives}
	S_a'(x)
	&=-aS_a(x)\tanh(ax),
	\\
	S_a''(x)
	&=a^2S_a(x)\bigl(1-2S_a(x)^2\bigr).
\end{align}
Since $\abs{S_a'}\leq aS_a$ and
$\abs{S_a''}\leq a^2S_a$, these formulas show that
\begin{equation}\label{eq:sech-Ab-admissibility}
	S_a\in\mathcal A_b^2
	\qquad\text{for every $b>0$ such that}\qquad
	a\geq\frac12+b.
\end{equation}
At the endpoint $a=\frac12+b$, the weighted supremum remains finite because
$e^{a\abs{x}}\operatorname{sech}(ax)$ is bounded.

For the normalized autocorrelation calculation, put
\begin{equation}\label{eq:normalized-sech-generator}
	\sigma_a(x):=\sqrt{\frac a2}\,S_a(x).
\end{equation}
Since
\begin{equation}\label{eq:sech-L2-normalization}
	\norm{S_a}_2^2=\frac2a,
\end{equation}
one has $\norm{\sigma_a}_2=1$.  The elementary identity
\begin{equation}\label{eq:tanh-difference-identity}
	\tanh u-\tanh v
	=\frac{\sinh(u-v)}{\cosh u\cosh v}
\end{equation}
gives, for $x\neq0$,
\begin{align}\label{eq:sech-unnormalized-convolution}
	(S_a*S_a)(x)
	&=\int_{\mathbb R}
	\frac{\dd y}{\cosh(a(x+y))\cosh(ay)}
	\\
	&=\frac{2x}{\sinh(ax)}.
\end{align}
The right-hand side extends continuously to the value $2/a$ at $x=0$.
Consequently, the normalized autocorrelation is
\begin{equation}\label{eq:sech-normalized-autocorrelation}
	C_a(x):=(\sigma_a*\widetilde\sigma_a)(x)
	=\begin{cases}
		\displaystyle\frac{ax}{\sinh(ax)},&x\neq0,\\[2mm]
		1,&x=0.
	\end{cases}
\end{equation}
It is a real analytic even function on $\mathbb R$.  The convolution theorem
and \cref{eq:sech-transform} give
\begin{align}\label{eq:sech-autocorrelation-transforms}
	\Phi_{C_a}(s)
	&=\frac{\pi^2}{2a}
	\sec^2\!\left(\frac{\pi z}{2a}\right),
	\qquad \abs{\RePart z}<a,
	\\
	\widehat C_a(t)
	&=\frac{\pi^2}{2a}
	\operatorname{sech}^2\!\left(\frac{\pi t}{2a}\right)
	\geq0.
\end{align}
Because
\begin{equation}\label{eq:sech-autocorrelation-asymptotic}
	C_a(x)=2a\abs{x}e^{-a\abs{x}}
	\left(1+O(e^{-2a\abs{x}})\right)
	\qquad (\abs{x}\to\infty),
\end{equation}
and the same type of bound holds for its first two derivatives,
\begin{equation}\label{eq:sech-autocorrelation-Ab}
	C_a\in\mathcal A_b^2
	\qquad\text{for every $b>0$ such that}\qquad
	a>\frac12+b.
\end{equation}
Here the inequality must be strict because of the factor $\abs{x}$ in
\cref{eq:sech-autocorrelation-asymptotic}.

\subsection{Compactly supported cardinal \texorpdfstring{$B$}{B}-splines}\label{subsec:Bspline-functions}

Let $\delta>0$ and define
\begin{equation}\label{eq:box-and-Bspline-definition}
	q_\delta(x):=\frac{1}{\delta}
	\mathbf 1_{[-\delta/2,\delta/2]}(x),
	\qquad
	B_{r,\delta}:=\underbrace{q_\delta*\cdots*q_\delta}_{r\text{ factors}}.
\end{equation}
For every integer $r\geq3$, the centered $B$-spline $B_{r,\delta}$ has
support in $[-r\delta/2,r\delta/2]$.  It is $C^{r-2}$ and piecewise
polynomial, and it belongs to $W_{\mathrm{loc}}^{2,1}(\mathbb R)$.  The
convolution theorem gives
\begin{align}\label{eq:Bspline-transforms}
	\Phi_{B_{r,\delta}}(s)
	&=\left(
	\frac{\sinh(\delta z/2)}{\delta z/2}
	\right)^r,
	\\
	\widehat{B_{r,\delta}}(t)
	&=\left(
	\frac{\sin(\delta t/2)}{\delta t/2}
	\right)^r,
\end{align}
with the values at $z=0$ and $t=0$ understood by continuity.  When
$r=2k\geq4$,
\begin{equation}\label{eq:Bspline-autocorrelation}
	B_{2k,\delta}
	=B_{k,\delta}*\widetilde{B_{k,\delta}},
	\qquad
	\widehat{B_{2k,\delta}}(t)\geq0.
\end{equation}
Because $B_{r,\delta}$ has compact support, it belongs to $\mathcal A_b^2$
for every $b>0$.  Its prime-power sum is exactly finite once
$T\geq\exp(r\delta/2)$, while its Fourier transform is
$O((1+\abs{t})^{-r})$.  These properties make even-order $B$-splines
particularly convenient for simultaneous prime and zero truncation; see
\cite{deBoor2001} for general spline theory.

\subsection{Finite Fourier modes on a logarithmic interval}
\label{subsec:finite-logarithmic-Fourier-modes}

\subsubsection{Sharp cutoff modes and their orthogonality}

Let $T\geq3$ and put 
\begin{equation}\label{eq:finite-Fourier-L-definition}
	L:=\log T,
\end{equation}
and choose an integer $N$ such that
$0\leq N\leq\lfloor L\rfloor$.  For $\abs{n}\leq N$, set
\begin{equation}\label{eq:finite-Fourier-frequency}
	\nu_{n,L}:=\frac{2\pi n}{L}.
\end{equation}
To implement Weil's midpoint convention at the two discontinuities, define
\begin{equation}\label{eq:midpoint-box-definition}
	\chi_L^{*}(x)
	:=\begin{cases}
		1,&\abs{x}<L,\\
		\frac12,&\abs{x}=L,\\
		0,&\abs{x}>L,
	\end{cases}
\end{equation}
and consider the sharp logarithmic Fourier modes
\begin{equation}\label{eq:sharp-logarithmic-Fourier-modes}
	f_{n,L}(x)
	:=\frac1L e^{i\nu_{n,L}x}\chi_L^{*}(x),
	\qquad -N\leq n\leq N.
\end{equation}
The endpoint values in \cref{eq:midpoint-box-definition} do not affect any
Lebesgue integral or Fourier transform.  They matter only when $T$ itself is
a prime power; the midpoint convention then assigns one-half of the boundary
prime-power term.

A direct calculation gives
\begin{equation}\label{eq:sharp-Fourier-orthogonality}
	\left\langle f_{n,L},f_{m,L}\right\rangle_{L^2(\mathbb R)}
	=\frac{2}{L}\,\delta_{nm},
	\qquad
	\norm{f_{n,L}}_1=2,
	\qquad
	\norm{f_{n,L}}_2^2=\frac2L.
\end{equation}
Thus
\begin{equation}\label{eq:sharp-Fourier-normalized-modes}
	e_{n,L}(x)
	:=\sqrt{\frac L2}\,f_{n,L}(x)
	=\frac{1}{\sqrt{2L}}
	e^{2\pi i n x/L}\chi_L^{*}(x)
\end{equation}
is an orthonormal set in $L^2(-L,L)$.  It is not a complete orthonormal
basis there: the complete complex Fourier basis on the interval of length
$2L$ is
\begin{equation}\label{eq:full-Fourier-basis-minus-L-L}
	\left\{
	\frac{1}{\sqrt{2L}}e^{i\pi kx/L}:k\in\mathbb Z
	\right\},
\end{equation}
and \cref{eq:sharp-Fourier-normalized-modes} consists of its even-indexed
members $k=2n$.  The corresponding real orthonormal set is
\begin{equation}\label{eq:finite-real-Fourier-system}
	\frac{\chi_L^{*}(x)}{\sqrt{2L}},
	\qquad
	\frac{\cos(2\pi nx/L)}{\sqrt L}\chi_L^{*}(x),
	\qquad
	\frac{\sin(2\pi nx/L)}{\sqrt L}\chi_L^{*}(x),
	\quad 1\leq n\leq N.
\end{equation}
The restriction $N\leq\lfloor L\rfloor$ implies the useful uniform frequency
bound
\begin{equation}\label{eq:finite-Fourier-frequency-bound}
	\abs{\nu_{n,L}}\leq2\pi,
\end{equation}
but it does not by itself provide uniform control of the weighted Sobolev
norms as $L\to\infty$.

\subsubsection{Explicit Fourier and bilateral Laplace transforms}

Write $z=s-\frac12$.  Since $L\nu_{n,L}=2\pi n$, direct integration yields
\begin{align}\label{eq:sharp-Fourier-transforms}
	\Phi_{f_{n,L}}(s)
	&=\int_{\mathbb R}f_{n,L}(x)e^{zx}\dd x
	=\frac{2\sinh(Lz)}{Lz+2\pi i n},
	\\
	H_{n,L}(w)
	&:=\int_{\mathbb R}f_{n,L}(x)e^{iwx}\dd x
	=\frac{2\sin(Lw)}{Lw+2\pi n}.
\end{align}
All apparent singularities in \cref{eq:sharp-Fourier-transforms} are
removable; for example, $H_{n,L}(-2\pi n/L)=2$.  On the real Fourier axis,
\begin{equation}\label{eq:sharp-Fourier-real-axis}
	\widehat{f_{n,L}}(t)
	=\frac{2\sin(Lt)}{Lt+2\pi n}.
\end{equation}

For the Riemann $\Xi$-function, only the even part of the test function is
needed.  Define
\begin{equation}\label{eq:sharp-cosine-mode}
	c_{n,L}^{\sharp}(x)
	:=\frac{f_{n,L}(x)+f_{-n,L}(x)}2
	=\frac1L\cos(\nu_{n,L}x)\chi_L^{*}(x),
	\qquad 0\leq n\leq N.
\end{equation}
Its transform is the even entire function
\begin{equation}\label{eq:sharp-cosine-transform}
	H_{n,L}^{\sharp}(w)
	=\frac{2Lw\sin(Lw)}{L^2w^2-(2\pi n)^2},
\end{equation}
again with the removable values understood by continuity.  The odd sine
part contributes zero to every term of the specialized $\Xi$ explicit
formula: its values at $x$ and $-x$ cancel on the prime side, the
archimedean kernel is even, and its two pole terms cancel.  Consequently,
$f_{n,L}$ and $c_{n,L}^{\sharp}$ give the same scalar explicit formula.

\subsubsection{An exact finite prime-power formula}

Because $f_{n,L}$ is supported in $[-L,L]$ and $L=\log T$, its full
prime-power contribution is already finite:
\begin{align}\label{eq:sharp-Fourier-prime-side}
	\mathcal P(f_{n,L})
	&=\frac2L
	\sum_{\substack{p,m\geq1\\p^m\leq T}}^{\prime}
	\frac{\log p}{p^{m/2}}
	\cos\!\left(\frac{2\pi n\,m\log p}{L}\right)
	\\
	&=\frac2L
	\sum_{2\leq q\leq T}^{\prime}
	\frac{\Lambda(q)}{\sqrt q}
	\cos\!\left(\frac{2\pi n\log q}{L}\right).
\end{align}
Here the prime on the summation sign means that the term $q=T$ receives
weight $1/2$ if $T$ is a prime power.  If $T$ is not a prime power, the
prime may be omitted and \cref{eq:sharp-Fourier-prime-side} is exactly the
sharp truncation $q\leq T$ requested above.  Thus these functions perform a
finite Fourier analysis of the prime-power measure on the logarithmic
interval $[0,\log T]$.

Substitution into \cref{eq:Xi-Weil-even} gives the following completely
explicit identity.

\begin{proposition}[Exact sharp logarithmic Fourier-mode formula]
	\label{prop:sharp-logarithmic-Fourier-explicit-formula}
	For every fixed $T\geq3$ and $0\leq n\leq N$, one has
	\begin{align}\label{eq:sharp-logarithmic-Fourier-explicit-formula}
		\sum_{\rho}^{*}
		\frac{2Lz_\rho\sin(Lz_\rho)}
		{L^2z_\rho^2-(2\pi n)^2}
		={}&
		\frac{2L\sinh(L/2)}{(2\pi n)^2+L^2/4}
		-\frac{\log(2\pi)}{L}
		\notag\\
		&-\frac2L\PF\!\int_0^L
		\cos\!\left(\frac{2\pi nx}{L}\right)
		\frac{e^{x/2}}{e^x-e^{-x}}\dd x
		\notag\\
		&-\frac2L
		\sum_{2\leq q\leq T}^{\prime}
		\frac{\Lambda(q)}{\sqrt q}
		\cos\!\left(\frac{2\pi n\log q}{L}\right).
	\end{align}
	All removable values on the zero side are interpreted by continuity, and the
	star denotes the symmetric limit in the zero ordinates.  The one-sided
	finite part in \cref{eq:sharp-logarithmic-Fourier-explicit-formula} means one
	half of the symmetric finite part of the corresponding even integrand on
	$[-L,L]$.  Under the Riemann hypothesis, $z_\rho=\gamma\in\mathbb R$, and the
	zero weight becomes
	\begin{equation}\label{eq:sharp-Fourier-zero-weight-RH}
		\frac{2L\gamma\sin(L\gamma)}
		{L^2\gamma^2-(2\pi n)^2}.
	\end{equation}
\end{proposition}

\begin{proof}
	For the even test function $c_{n,L}^{\sharp}$, one has
	$c_{n,L}^{\sharp}(0)=1/L$.  Moreover,
	\begin{equation}\label{eq:sharp-Fourier-pole-pair}
		2H_{n,L}^{\sharp}\!\left(\frac i2\right)
		=\frac{2L\sinh(L/2)}{(2\pi n)^2+L^2/4}.
	\end{equation}
	The archimedean kernel in \cref{eq:A-infinity-zeta} is even and the test
	function is supported in $[-L,L]$, so its finite-part integral is twice the
	one-sided finite part over $[0,L]$.  Finally,
	\cref{eq:sharp-Fourier-prime-side} gives the exact finite prime-power term.
	Substitution into the starred identity \cref{eq:Xi-Weil-even} proves
	\cref{eq:sharp-logarithmic-Fourier-explicit-formula}.  No quantitative finite
	zero cutoff is asserted here, because the sharp mode is not in
	$\mathcal A_b^2$.
\end{proof}

\subsubsection{Autocorrelation and the obstruction to
	\texorpdfstring{$W_{\mathrm{loc}}^{2,1}$}{Wloc21} admissibility}

The sharp modes are compactly supported and are admissible under Weil's
original piecewise-smooth hypotheses, but they do not belong to the stronger
class $\mathcal A_b^2$.  Indeed, in the distributional sense,
\begin{equation}\label{eq:sharp-Fourier-distributional-derivative}
	Df_{n,L}
	=i\nu_{n,L}f_{n,L}
	+\frac1L\bigl(\delta_{-L}-\delta_L\bigr),
\end{equation}
because $e^{\pm i\nu_{n,L}L}=1$.  Thus
$f_{n,L}\notin W_{\mathrm{loc}}^{1,1}(\mathbb R)$ and, a fortiori,
$f_{n,L}\notin W_{\mathrm{loc}}^{2,1}(\mathbb R)$.  Although
\begin{equation}\label{eq:sharp-Fourier-Mb}
	M_b(f_{n,L})
	=\frac{e^{(1/2+b)L}}{L}
	=\frac{T^{1/2+b}}{\log T},
\end{equation}
the quantity $J_2(f_{n,L})$ in \cref{eq:Mb-J2-definitions} is not finite in
the required weak-derivative sense.  Correspondingly,
$H_{n,L}(t)=O(\abs{t}^{-1})$, and the zero sum is generally only a starred,
symmetric sum rather than an absolutely convergent series.  Therefore
\cref{thm:uniform-Xi-truncation} cannot be applied directly to the sharp
modes.

The Weil autocorrelation is nevertheless elementary.  Since
$\widetilde f_{n,L}=f_{n,L}$,
\begin{equation}\label{eq:sharp-Fourier-autocorrelation}
	A_{n,L}(x)
	:=\bigl(f_{n,L}*\widetilde f_{n,L}\bigr)(x)
	=\frac{(2L-\abs{x})_+}{L^2}
	e^{i\nu_{n,L}x},
\end{equation}
where $(y)_+:=\max\{y,0\}$.  Its transforms are
\begin{align}\label{eq:sharp-Fourier-autocorrelation-transform}
	\Phi_{A_{n,L}}(s)
	&=\left(
	\frac{2\sinh(Lz)}{Lz+2\pi i n}
	\right)^2,
	\\
	\widehat A_{n,L}(t)
	&=\left(
	\frac{2\sin(Lt)}{Lt+2\pi n}
	\right)^2\geq0.
\end{align}
Thus the autocorrelation has a nonnegative transform and its zero series is
absolutely convergent for each fixed $L$ and $n$, since its transform is
$O(\abs{t}^{-2})$.  However, the triangular factor in
\cref{eq:sharp-Fourier-autocorrelation} has corners at $0$ and at
$\pm2L$.  Hence $A_{n,L}\notin W_{\mathrm{loc}}^{2,1}(\mathbb R)$ either.
Its support is $[-2L,2L]$, so its prime-power side extends to $q\leq T^2$.

\subsubsection{A simple Sobolev-compatible raised-cosine window}

A minimal smoothing that preserves compact support and an elementary
transform is obtained with the raised-cosine, or Hann, window
\begin{equation}\label{eq:raised-cosine-Fourier-mode}
	\psi_{n,L}(x)
	:=\frac1L e^{i\nu_{n,L}x}
	\cos^2\!\left(\frac{\pi x}{2L}\right)
	\mathbf 1_{[-L,L]}(x).
\end{equation}
Both the function and its first derivative vanish at $x=\pm L$.  Its zero
extension is therefore $C^1$, has a locally integrable second weak
derivative, and belongs to $\mathcal A_b^2$ for every $b>0$.  Put
\begin{equation}\label{eq:raised-cosine-u-definition}
	u_{n,L}(z):=z+i\nu_{n,L}.
\end{equation}
Expanding the cosine squared into three exponentials gives
\begin{align}\label{eq:raised-cosine-transforms}
	\Phi_{\psi_{n,L}}(s)
	&=\frac{\pi^2\sinh(Lz)}
	{L^3u_{n,L}(z)
		\left(u_{n,L}(z)^2+\pi^2/L^2\right)},
	\\
	\widehat{\psi_{n,L}}(t)
	&=-\frac{\pi^2\sin(Lt)}
	{L^3(t+\nu_{n,L})
		\left((t+\nu_{n,L})^2-\pi^2/L^2\right)}.
\end{align}
Again all apparent singularities are removable.  In particular,
$\widehat{\psi_{n,L}}(t)=O(\abs{t}^{-3})$ for fixed $L$ and $n$, so the
zero series is absolutely convergent with a stronger tail than that of the
sharp mode.

The prime-power side remains exactly finite, but now has a smooth endpoint
weight:
\begin{align}\label{eq:raised-cosine-prime-side}
	\mathcal P(\psi_{n,L})
	=\frac2L\sum_{2\leq q\leq T}
	\frac{\Lambda(q)}{\sqrt q}
	\cos\!\left(\frac{2\pi n\log q}{L}\right)
	\cos^2\!\left(\frac{\pi\log q}{2L}\right).
\end{align}
The boundary term at $q=T$ vanishes.  Its autocorrelation
\begin{equation}\label{eq:raised-cosine-autocorrelation}
	\Psi_{n,L}:=\psi_{n,L}*\widetilde\psi_{n,L}
\end{equation}
is supported in $[-2L,2L]$, belongs to $\mathcal A_b^2$ for every $b>0$,
and satisfies
\begin{equation}\label{eq:raised-cosine-autocorrelation-transform}
	\widehat\Psi_{n,L}(t)
	=\abs{\widehat{\psi_{n,L}}(t)}^2\geq0.
\end{equation}
Thus \cref{eq:raised-cosine-Fourier-mode} is a direct
$W_{\mathrm{loc}}^{2,1}$-compatible replacement for the sharp cutoff.

\subsubsection{A positive-type spline-windowed Fourier mode}

The spline family of \cref{subsec:Bspline-functions} gives a positive-type
regularization whose support remains exactly $[-L,L]$.  Let $k\geq2$,
put $\delta=L/k$, and define
\begin{equation}\label{eq:spline-windowed-Fourier-mode}
	G_{n,L,k}(x)
	:=e^{i\nu_{n,L}x}B_{k,\delta}(x),
	\qquad
	F_{n,L,k}(x)
	:=G_{n,L,k}*\widetilde G_{n,L,k}(x).
\end{equation}
Because the centered splines are even,
\begin{equation}\label{eq:spline-windowed-Fourier-mode-physical}
	F_{n,L,k}(x)
	=e^{i\nu_{n,L}x}B_{2k,L/k}(x),
	\qquad
	\operatorname{supp}F_{n,L,k}\subseteq[-L,L].
\end{equation}
It belongs to $\mathcal A_b^2$ for every $b>0$, and
\begin{align}\label{eq:spline-windowed-Fourier-mode-transforms}
	\Phi_{F_{n,L,k}}(s)
	&=\left(
	\frac{
		\sinh\!\left(\frac{L}{2k}(z+i\nu_{n,L})\right)}
	{\frac{L}{2k}(z+i\nu_{n,L})}
	\right)^{2k},
	\\
	\widehat F_{n,L,k}(t)
	&=\left(
	\frac{
		\sin\!\left(\frac{L}{2k}(t+\nu_{n,L})\right)}
	{\frac{L}{2k}(t+\nu_{n,L})}
	\right)^{2k}\geq0.
\end{align}
The real even average
\begin{equation}\label{eq:spline-windowed-cosine-mode}
	C_{n,L,k}(x)
	:=\frac{F_{n,L,k}(x)+F_{-n,L,k}(x)}2
	=B_{2k,L/k}(x)\cos(\nu_{n,L}x)
\end{equation}
also has a nonnegative Fourier transform, namely the average of the two
nonnegative shifted even powers in
\cref{eq:spline-windowed-Fourier-mode-transforms}.  Its prime-power term is
the finite weighted cosine sum
\begin{equation}\label{eq:spline-windowed-cosine-prime-side}
	\mathcal P(C_{n,L,k})
	=2\sum_{2\leq q\leq T}
	\frac{\Lambda(q)}{\sqrt q}
	B_{2k,L/k}(\log q)
	\cos\!\left(\frac{2\pi n\log q}{L}\right).
\end{equation}
Among the Fourier-mode constructions in this subsection,
\cref{eq:spline-windowed-cosine-mode} simultaneously provides compact
support in $[-\log T,\log T]$, $W_{\mathrm{loc}}^{2,1}$ admissibility, an
explicit transform, and positive type.

\begin{remark}[Dependence on the moving cutoff]
	\label{rem:Fourier-mode-moving-cutoff}
	For each fixed integer $T\geq3$, the raised-cosine and spline-windowed modes lie in
	$\mathcal A_1^2$.  However, the test function itself depends on the same
	parameter $T$ through $L=\log T$.  Thus membership for each fixed $T$ does not
	by itself justify substituting this moving family into the diagonal limit of
	\cref{cor:parameter-uniformity}.  A sufficient condition for a family
	$F_T\in\mathcal A_1^2$ is
	\begin{equation}\label{eq:moving-family-sufficient-condition}
		\norm{F_T}_{\mathcal A_1^2}
		=o\!\left(\frac{T}{\log(T+2)+1}\right),
	\end{equation}
	because \cref{eq:single-parameter-truncation-bound} then gives
	$\mathcal E_{T,T}(F_T)\to0$.  Uniformity for
	$\{F_{n,T}:\abs{n}\leq N(T)\}$ follows from the corresponding condition with
	$\sup_{\abs{n}\leq N(T)}$ in front of the norm.  The restriction
	$N\leq\lfloor L\rfloor$ controls the carrier frequencies through
	\cref{eq:finite-Fourier-frequency-bound}, but it does not by itself control
	these weighted Sobolev norms.  The sharp modes are excluded altogether from
	\cref{cor:parameter-uniformity}, since they do not belong to
	$W_{\mathrm{loc}}^{2,1}(\mathbb R)$.
\end{remark}

\subsection{Hermitian matrices from raised-cosine modes}
\label{subsec:raised-cosine-Hermitian-matrix}

Retain $L=\log T$, the index set
\begin{equation}\label{eq:Qmn-index-set}
	I_N:=\{-N,-N+1,\ldots,N-1,N\},
	\qquad M:=\abs{I_N}=2N+1,
\end{equation}
and the raised-cosine functions $\psi_{n,L}$ from
\cref{eq:raised-cosine-Fourier-mode}.  For $m,n\in I_N$, define
\begin{equation}\label{eq:Qmn-definition}
	Q_{mn,L}(x)
	:=\frac12\left(
	\psi_{n,L}(x)\psi_{-m,L}(x)
	+\psi_{-n,L}(x)\psi_{m,L}(x)
	\right),
\end{equation}
and let
\begin{equation}\label{eq:Q-matrix-definition}
	\mathbf Q_{N,L}(x)
	:=\bigl(Q_{mn,L}(x)\bigr)_{m,n\in I_N}.
\end{equation}
Since $\psi_{-k,L}=\overline{\psi_{k,L}}$ on $\mathbb R$, one obtains the
real even test function
\begin{equation}\label{eq:Qmn-physical-form}
	Q_{mn,L}(x)
	=\frac{1}{L^2}
	\cos\!\left(\nu_{n-m,L}x\right)
	\cos^4\!\left(\frac{\pi x}{2L}\right)
	\mathbf 1_{[-L,L]}(x).
\end{equation}
Thus $Q_{mn,L}=Q_{nm,L}\in\mathbb R$, and
$\mathbf Q_{N,L}(x)$ is a real symmetric Toeplitz matrix for every fixed
$x$.

\subsubsection{Gram representation and explicit eigenvalues}

Put
\begin{equation}\label{eq:Qmn-kappa-theta-definition}
	\kappa_L(x)
	:=\frac1L\cos^2\!\left(\frac{\pi x}{2L}\right)
	\mathbf 1_{[-L,L]}(x),
	\qquad
	\theta_L(x):=\frac{2\pi x}{L},
\end{equation}
and introduce the real vectors
\begin{equation}\label{eq:Qmn-cos-sin-vectors}
	\mathbf c(x)
	:=\bigl(\cos(n\theta_L(x))\bigr)_{n\in I_N},
	\qquad
	\mathbf s(x)
	:=\bigl(\sin(n\theta_L(x))\bigr)_{n\in I_N}.
\end{equation}
The identity
$\cos((n-m)\theta)=\cos(n\theta)\cos(m\theta)
+\sin(n\theta)\sin(m\theta)$ gives
\begin{equation}\label{eq:Qmn-Gram-representation}
	\mathbf Q_{N,L}(x)
	=\kappa_L(x)^2
	\left(
	\mathbf c(x)\mathbf c(x)^{\mathsf T}
	+\mathbf s(x)\mathbf s(x)^{\mathsf T}
	\right).
\end{equation}
Consequently,
\begin{equation}\label{eq:Qmn-positive-semidefinite}
	\mathbf u^*\mathbf Q_{N,L}(x)\mathbf u
	=\kappa_L(x)^2
	\left(
	\abs{\mathbf c(x)^{\mathsf T}\mathbf u}^2
	+\abs{\mathbf s(x)^{\mathsf T}\mathbf u}^2
	\right)
	\geq0
\end{equation}
for every $\mathbf u\in\mathbb C^M$.  Hence the matrix is not merely
Hermitian: it is positive semidefinite and has rank at most two.

The nonzero eigenvalues can also be written explicitly.  Recall that
$M=\abs{I_N}=2N+1$, and let
\begin{equation}\label{eq:Dirichlet-kernel-Qmn}
	D_N(u)
	:=\sum_{k=-N}^{N}e^{iku}
	=1+2\sum_{k=1}^{N}\cos(ku)
	=\frac{\sin((N+\tfrac12)u)}{\sin(u/2)},
\end{equation}
where the last quotient is interpreted by continuity.  Symmetry of $I_N$
gives $\mathbf c(x)^{\mathsf T}\mathbf s(x)=0$ and
\begin{align}\label{eq:Qmn-cos-sin-norms}
	\norm{\mathbf c(x)}_2^2
	&=\frac12\left(
	M+D_N\!\left(\frac{4\pi x}{L}\right)
	\right),
	\\
	\norm{\mathbf s(x)}_2^2
	&=\frac12\left(
	M-D_N\!\left(\frac{4\pi x}{L}\right)
	\right).
\end{align}
For $N\geq1$, the spectrum therefore consists of $M-2$ zero eigenvalues and
the two eigenvalues
\begin{equation}\label{eq:Qmn-explicit-eigenvalues}
	\lambda_{1,2}(x)
	=\frac{\kappa_L(x)^2}{2}
	\left(
	M\mathbin{\pm}
	\abs{D_N\!\left(\frac{4\pi x}{L}\right)}
	\right).
\end{equation}
Indeed, the finite-sum representation and the triangle inequality give
\[
\abs{D_N(u)}
=\left|\sum_{k=-N}^{N}e^{iku}\right|
\leq\sum_{k=-N}^{N}\abs{e^{iku}}
=2N+1=M.
\]
Consequently, the two eigenvalues are real and nonnegative.  When $N=0$,
the only eigenvalue is $\kappa_L(x)^2$.  At exceptional points one of the
two quantities in \cref{eq:Qmn-explicit-eigenvalues} may vanish, so the rank
can drop from two to one or zero.

\subsubsection{Regularity and explicit transform}

Writing $r:=n-m$, one has $\abs{r}\leq2N$ and therefore
\begin{equation}\label{eq:Qmn-frequency-bound}
	\abs{\nu_{r,L}}\leq4\pi
\end{equation}
under the standing assumption $N\leq\lfloor L\rfloor$.  The fourth-power
window in \cref{eq:Qmn-physical-form} vanishes to order four at $x=\pm L$.
Its zero extension is $C^3$, and
\begin{equation}\label{eq:Qmn-regularity}
	Q_{mn,L}\in
	C^3(\mathbb R)\cap W_{\mathrm{loc}}^{4,1}(\mathbb R)
	\subset\mathcal A_b^2
	\qquad\text{for every }b>0
\end{equation}
when $L$, $m$, and $n$ are fixed.

For $w\in\mathbb C$, define the entire function
\begin{align}\label{eq:fourth-power-cosine-transform}
	\mathcal J_L(w)
	&:=\int_{-L}^{L}
	\cos^4\!\left(\frac{\pi x}{2L}\right)e^{iwx}\dd x
	\\
	&=\frac{3(\pi/L)^4\sin(Lw)}
	{w\left(w^2-\pi^2/L^2\right)
		\left(w^2-4\pi^2/L^2\right)}.
\end{align}
All apparent poles in the second line are removable; in particular,
$\mathcal J_L(0)=3L/4$.  The transform of $Q_{mn,L}$ is
\begin{equation}\label{eq:Qmn-Fourier-transform}
	H_{mn,L}^{Q}(z)
	:=\int_{\mathbb R}Q_{mn,L}(x)e^{izx}\dd x
	=\frac{1}{2L^2}
	\left(
	\mathcal J_L(z+\nu_{r,L})
	+\mathcal J_L(z-\nu_{r,L})
	\right).
\end{equation}
It is even and entire in $z$, is real on the real axis, and satisfies
\begin{equation}\label{eq:Qmn-transform-decay}
	H_{mn,L}^{Q}(t)=O_{L,m,n}\!\left((1+\abs{t})^{-5}\right)
	\qquad(t\in\mathbb R).
\end{equation}
Thus the corresponding zero series is absolutely convergent; the starred
summation convention is unnecessary for these test functions.

\subsubsection{The entrywise Weil explicit formula}

Set
\begin{equation}\label{eq:Qmn-alpha-nu-definition}
	\alpha_L:=\frac{\pi}{L},
	\qquad
	\nu:=\nu_{r,L}=\frac{2\pi r}{L},
\end{equation}
and define the explicit pole contribution
\begin{align}\label{eq:Qmn-pole-contribution}
	\mathcal B_{r,L}
	:=\frac{\sinh(L/2)}{4L^2}
	\Bigg[{}&
	\frac{3}{\nu^2+\tfrac14}
	-\frac{2}{(\nu+\alpha_L)^2+\tfrac14}
	-\frac{2}{(\nu-\alpha_L)^2+\tfrac14}
	\notag\\
	&+\frac{1}{2\bigl((\nu+2\alpha_L)^2+\tfrac14\bigr)}
	+\frac{1}{2\bigl((\nu-2\alpha_L)^2+\tfrac14\bigr)}
	\Bigg].
\end{align}

\begin{proposition}[Weil formula for the Hermitian raised-cosine matrix]
	\label{prop:Qmn-Weil-explicit-formula}
	For every fixed $T\geq3$, $L=\log T$, and $m,n\in I_N$, one has the
	absolutely convergent identity
	\begin{align}\label{eq:Qmn-Weil-explicit-formula}
		\sum_{\rho}H_{mn,L}^{Q}(z_\rho)
		={}&\mathcal B_{r,L}
		-\frac{\log(2\pi)}{L^2}
		\notag\\
		&-\frac{2}{L^2}\PF\!\int_0^L
		\cos(\nu_{r,L}x)
		\cos^4\!\left(\frac{\pi x}{2L}\right)
		\frac{e^{x/2}}{e^x-e^{-x}}\dd x
		\notag\\
		&-\frac{2}{L^2}
		\sum_{2\leq q\leq T}
		\frac{\Lambda(q)}{\sqrt q}
		\cos\!\left(\frac{2\pi r\log q}{L}\right)
		\cos^4\!\left(\frac{\pi\log q}{2L}\right),
		\qquad r=n-m.
	\end{align}
	The term $q=T$ is automatically zero, even when $T$ is a prime power, so no
	endpoint half weight is required.  Under the Riemann hypothesis,
	$z_\rho=\gamma\in\mathbb R$, and the zero weight in
	\cref{eq:Qmn-Weil-explicit-formula} is
	\begin{equation}\label{eq:Qmn-zero-weight-RH}
		H_{mn,L}^{Q}(\gamma)
		=\frac{1}{2L^2}
		\left(
		\mathcal J_L(\gamma+\nu_{r,L})
		+\mathcal J_L(\gamma-\nu_{r,L})
		\right).
	\end{equation}
\end{proposition}

\begin{proof}
	The function $Q_{mn,L}$ is real and even, and
	$Q_{mn,L}(0)=L^{-2}$.  Expanding the window gives
	\begin{align*}
		\cos(\nu x)\cos^4(\alpha_Lx/2)
		=\frac18\Bigl[{}&3\cos(\nu x)
		+2\cos((\nu+\alpha_L)x)
		+2\cos((\nu-\alpha_L)x)
		\\
		&+\frac12\cos((\nu+2\alpha_L)x)
		+\frac12\cos((\nu-2\alpha_L)x)
		\Bigr].
	\end{align*}
	If $cL\in\pi\mathbb Z$, direct integration gives
	\begin{equation}\label{eq:Qmn-pole-integral-identity}
		\int_{-L}^{L}\cos(cx)e^{-x/2}\dd x
		=\frac{\sinh(L/2)\cos(cL)}{c^2+\tfrac14}.
	\end{equation}
	Since $\nu L=2\pi r$, the five cosine frequencies in the preceding
	expansion have endpoint signs $+,-,-,+,+$, respectively.  Hence
	$2H_{mn,L}^{Q}(i/2)=\mathcal B_{r,L}$.  The archimedean kernel is even, so
	its contribution is twice the one-sided finite part displayed in
	\cref{eq:Qmn-Weil-explicit-formula}.  The support is $[-L,L]$, and the
	fourth-power window vanishes at the endpoints; consequently, the complete
	prime-power side is exactly the finite sum $q\leq T$ in that formula.
	Substitution into \cref{eq:Xi-Weil-even} proves the identity.  Finally,
	\cref{eq:Qmn-transform-decay} together with the Riemann--von Mangoldt
	estimate shows that the zero series converges absolutely.
\end{proof}

\subsubsection{Matrix form and a uniform finite-zero approximation}

For $r\in\{-2N,\ldots,2N\}$, let $\mathcal A_{r,L}$ and
$\mathcal P_{r,L}$ denote, respectively, the archimedean integral and the
finite prime-power sum on the last two lines of
\cref{eq:Qmn-Weil-explicit-formula}, including their displayed positive
prefactors.  Define the real symmetric Toeplitz matrices
\begin{align}\label{eq:Qmn-explicit-matrices}
	\mathbf Z_L^{(\infty)}
	&:=\left(
	\sum_\rho H_{mn,L}^{Q}(z_\rho)
	\right)_{m,n\in I_N},
	&
	\mathbf B_L&:=\bigl(\mathcal B_{n-m,L}\bigr)_{m,n\in I_N},
	\\
	\mathbf A_L&:=\bigl(\mathcal A_{n-m,L}\bigr)_{m,n\in I_N},
	&
	\mathbf P_L&:=\bigl(\mathcal P_{n-m,L}\bigr)_{m,n\in I_N}.
\end{align}
If $\mathbf 1_M=(1,\ldots,1)^{\mathsf T}\in\mathbb R^M$, the entrywise
formula \cref{eq:Qmn-Weil-explicit-formula} is the Hermitian matrix identity
\begin{equation}\label{eq:Qmn-Weil-matrix-form}
	\mathbf Z_L^{(\infty)}
	=\mathbf B_L
	-\frac{\log(2\pi)}{L^2}\mathbf 1_M\mathbf 1_M^{\mathsf T}
	-\mathbf A_L-\mathbf P_L.
\end{equation}
Every matrix in \cref{eq:Qmn-Weil-matrix-form} is real symmetric, and hence
all of its eigenvalues are real.  This Hermitianity should not be confused
with positivity: the pointwise matrix $\mathbf Q_{N,L}(x)$ is positive
semidefinite by \cref{eq:Qmn-positive-semidefinite}, whereas the oscillatory
zero, archimedean, and prime matrices in
\cref{eq:Qmn-Weil-matrix-form} need not be positive semidefinite separately.

The smooth fourth-power window also permits a uniform finite-zero cutoff,
even though the family moves with $T$.  Define
\begin{equation}\label{eq:Qmn-finite-zero-matrix}
	\mathbf Z_L^{(T)}
	:=\left(
	\sum_{\substack{\rho=\beta+i\gamma\\\abs{\gamma}<T}}
	H_{mn,L}^{Q}(z_\rho)
	\right)_{m,n\in I_N}.
\end{equation}
The functional equation and complex conjugation pair the zeros inside the
symmetric cutoff in \cref{eq:Qmn-finite-zero-matrix}; together with the even
real transform \cref{eq:Qmn-Fourier-transform}, this shows that
$\mathbf Z_L^{(T)}$ is real symmetric.  There is an absolute constant $C>0$
such that, uniformly for $m,n\in I_N$,
\begin{equation}\label{eq:Qmn-J2-uniform-bound}
	J_2(Q_{mn,L})
	\leq C\frac{T^{1/2}}{L^2}.
\end{equation}
Indeed, \cref{eq:Qmn-frequency-bound} and direct differentiation of
\cref{eq:Qmn-physical-form} give
\[
\abs{Q_{mn,L}(x)}
+\abs{Q_{mn,L}'(x)}
+\abs{Q_{mn,L}''(x)}
\leq \frac{C}{L^2}\mathbf 1_{[-L,L]}(x),
\]
and integration against $e^{\abs{x}/2}$ proves
\cref{eq:Qmn-J2-uniform-bound}.  Since the prime-power tail is exactly zero,
\cref{eq:zero-tail-bound} yields the entrywise estimate
\begin{equation}\label{eq:Qmn-entrywise-zero-tail}
	\max_{m,n\in I_N}
	\abs{
		\left(\mathbf Z_L^{(T)}-\mathbf Z_L^{(\infty)}\right)_{mn}
	}
	\leq
	C\frac{\log(T+2)+1}{\sqrt T\,L^2}.
\end{equation}
Because $M=2N+1\leq2L+1$, the matrix operator norm satisfies
\begin{equation}\label{eq:Qmn-operator-zero-tail}
	\norm{
		\mathbf Z_L^{(T)}-\mathbf Z_L^{(\infty)}
	}_{\mathrm{op}}
	\leq\frac{C}{\sqrt T}.
\end{equation}
Consequently, Weyl's eigenvalue inequality gives
\begin{equation}\label{eq:Qmn-eigenvalue-zero-tail}
	\max_{1\leq j\leq M}
	\abs{
		\lambda_j\!\left(\mathbf Z_L^{(T)}\right)
		-\lambda_j\!\left(\mathbf Z_L^{(\infty)}\right)
	}
	\leq\frac{C}{\sqrt T},
\end{equation}
when the real eigenvalues are arranged in nondecreasing order.  Thus the
finite-zero matrix approximates the full Weil matrix uniformly in operator
norm for the entire range $\abs{m},\abs{n}\leq N\leq\lfloor\log T\rfloor$.

\subsection{Hermitian convolution matrices from raised-cosine modes}
\label{subsec:raised-cosine-convolution-matrix}

Retain the notation of
\cref{subsec:raised-cosine-Hermitian-matrix}.  For $m,n\in I_N$, define
\begin{equation}\label{eq:Rmn-definition}
	R_{mn,L}(x)
	:=\frac12\left(
	(\psi_{n,L}*\psi_{-m,L})(x)
	+(\psi_{-n,L}*\psi_{m,L})(x)
	\right),
\end{equation}
and let
\begin{equation}\label{eq:R-matrix-definition}
	\mathbf R_{N,L}(x)
	:=\bigl(R_{mn,L}(x)\bigr)_{m,n\in I_N}.
\end{equation}
For real $x$, the two convolutions in \cref{eq:Rmn-definition} are complex
conjugates.  Hence $R_{mn,L}$ is real.  Commutativity of convolution also
gives
\begin{equation}\label{eq:Rmn-symmetries}
	R_{mn,L}(x)=R_{nm,L}(x)=R_{-m,-n,L}(x),
\end{equation}
so $\mathbf R_{N,L}(x)$ is a real symmetric, centrosymmetric matrix and all
of its eigenvalues are real.  Unlike the product matrix
$\mathbf Q_{N,L}(x)$, however, it is not generally positive semidefinite.

\subsubsection{Explicit physical-space convolution}

Put
\begin{equation}\label{eq:Rmn-r-s-a-definition}
	r:=n-m,
	\qquad
	s:=n+m,
	\qquad
	\alpha_L:=\frac{\pi}{L},
	\qquad
	a_L(x):=L-\frac{\abs{x}}2.
\end{equation}
For $a\geq0$ define the entire sinc-type function
\begin{equation}\label{eq:Rmn-S-function}
	\mathcal S_a(w)
	:=
	\begin{cases}
		\dfrac{2\sin(aw)}{w},&w\neq0,\\[2mm]
		2a,&w=0.
	\end{cases}
\end{equation}
For $\abs{x}\leq2L$, set
\begin{align}\label{eq:Rmn-K-kernel}
	\mathcal K_{s,L}(x)
	:=\frac1{L^2}\Bigg[{}&
	\frac{2+\cos(\alpha_Lx)}8
	\mathcal S_{a_L(x)}(\nu_{s,L})
	\notag\\
	&+\frac{\cos(\alpha_Lx/2)}4
	\left(
	\mathcal S_{a_L(x)}(\nu_{s,L}+\alpha_L)
	+\mathcal S_{a_L(x)}(\nu_{s,L}-\alpha_L)
	\right)
	\notag\\
	&+\frac1{16}
	\left(
	\mathcal S_{a_L(x)}(\nu_{s,L}+2\alpha_L)
	+\mathcal S_{a_L(x)}(\nu_{s,L}-2\alpha_L)
	\right)
	\Bigg],
\end{align}
and put $\mathcal K_{s,L}(x)=0$ for $\abs{x}>2L$.  Then
\begin{equation}\label{eq:Rmn-physical-form}
	R_{mn,L}(x)
	=\cos\!\left(\frac{\nu_{r,L}x}{2}\right)
	\mathcal K_{s,L}(x),
	\qquad r=n-m,\quad s=n+m.
\end{equation}
In particular, $R_{mn,L}$ is real, even, and supported in $[-2L,2L]$.

To prove \cref{eq:Rmn-physical-form}, write $y=x/2+u$.  On the overlap of
the two intervals $[-L,L]$ one has
$\abs{u}\leq a_L(x)$, and
\begin{align*}
	(\psi_{n,L}*\psi_{-m,L})(x)
	={}&e^{i\nu_{r,L}x/2}\frac1{L^2}
	\int_{-a_L(x)}^{a_L(x)}
	\cos^2\!\left(\frac{\alpha_Lx}{4}
	+\frac{\alpha_Lu}{2}\right)
	\cos^2\!\left(\frac{\alpha_Lx}{4}
	-\frac{\alpha_Lu}{2}\right)
	e^{i\nu_{s,L}u}\dd u.
\end{align*}
The product of the two windows equals
\begin{equation}\label{eq:Rmn-window-product-identity}
	\frac{2+\cos(\alpha_Lx)}8
	+\frac12\cos\!\left(\frac{\alpha_Lx}{2}\right)
	\cos(\alpha_Lu)
	+\frac18\cos(2\alpha_Lu).
\end{equation}
It is even in $u$, so the integral is real.  Product-to-sum identities then
give \cref{eq:Rmn-K-kernel}, while averaging with the complex conjugate
gives \cref{eq:Rmn-physical-form}.

At the origin the values collapse to three Fourier coefficients of the
fourth-power cosine window:
\begin{equation}\label{eq:Rmn-origin-value}
	R_{mn,L}(0)
	=\frac1{L^2}\mathcal J_L(\nu_{s,L})
	=:\eta_{s,L}
	=\frac1L
	\begin{cases}
		\dfrac34,&s=0,\\[1mm]
		\dfrac18,&\abs{s}=1,\\[1mm]
		0,&\abs{s}\geq2.
	\end{cases}
\end{equation}
Consequently, when $N\geq1$, the principal submatrix at $x=0$ indexed by
$\{-1,1\}$ is
\[
\begin{pmatrix}
	0&3/(4L)\\
	3/(4L)&0
\end{pmatrix},
\]
which has eigenvalues $\pm3/(4L)$.  Thus the Hermitian matrix
$\mathbf R_{N,L}(x)$ is generally indefinite.

\subsubsection{Regularity, transform, and transform-side spectrum}

Define the even entire function
\begin{equation}\label{eq:Rmn-W-window-transform}
	\mathcal W_L(z)
	:=-\frac{\pi^2\sin(Lz)}
	{L^3z\left(z^2-\alpha_L^2\right)},
\end{equation}
where the apparent singularities at $z=0$ and $z=\pm\alpha_L$ are removed
by continuity.  Thus $\mathcal W_L=\widehat{\psi_{0,L}}$ and
\begin{equation}\label{eq:Rmn-shifted-window-transform}
	\widehat{\psi_{n,L}}(z)
	=\mathcal W_L(z+\nu_{n,L}).
\end{equation}
The convolution theorem gives
\begin{align}\label{eq:Rmn-Fourier-transform}
	H_{mn,L}^{R}(z)
	&:=\int_{\mathbb R}R_{mn,L}(x)e^{izx}\dd x
	\notag\\
	&=\frac12\Bigl[
	\mathcal W_L(z+\nu_{n,L})
	\mathcal W_L(z-\nu_{m,L})
	\notag\\
	&\hspace{29mm}
	+\mathcal W_L(z-\nu_{n,L})
	\mathcal W_L(z+\nu_{m,L})
	\Bigr].
\end{align}
This transform is even and entire in $z$, is real on the real axis, and
satisfies
\begin{equation}\label{eq:Rmn-transform-decay}
	H_{mn,L}^{R}(z)
	=O_{L,m,n}\!\left((1+\abs{\RePart z})^{-6}\right)
	\qquad(\abs{\ImPart z}\leq\tfrac12).
\end{equation}
Moreover, since each $\psi_{n,L}$ lies in
$C^1_c(\mathbb R)\cap W^{2,1}(\mathbb R)$,
\begin{equation}\label{eq:Rmn-regularity}
	R_{mn,L}\in
	C_c^4(\mathbb R)\cap W^{4,1}(\mathbb R)
	\subset\mathcal A_b^2
	\qquad\text{for every }b>0
\end{equation}
for fixed $L,m,n$.

The transform matrix has a particularly simple rank-two representation.
For $t\in\mathbb R$, define
\begin{equation}\label{eq:Rmn-transform-vectors}
	\mathbf a_L(t)
	:=\bigl(\mathcal W_L(t+\nu_{n,L})\bigr)_{n\in I_N},
	\qquad
	\mathbf b_L(t)
	:=\bigl(\mathcal W_L(t-\nu_{n,L})\bigr)_{n\in I_N}.
\end{equation}
Then
\begin{equation}\label{eq:Rmn-transform-matrix-rank-two}
	\widehat{\mathbf R}_{N,L}(t)
	:=\bigl(H_{mn,L}^{R}(t)\bigr)_{m,n\in I_N}
	=\frac12\left(
	\mathbf a_L(t)\mathbf b_L(t)^{\mathsf T}
	+\mathbf b_L(t)\mathbf a_L(t)^{\mathsf T}
	\right).
\end{equation}
Because $I_N$ is symmetric and $\mathcal W_L$ is even,
$\norm{\mathbf a_L(t)}_2=\norm{\mathbf b_L(t)}_2$.  Put
\begin{equation}\label{eq:Rmn-transform-c-d-definition}
	c_L(t):=\mathbf a_L(t)^{\mathsf T}\mathbf b_L(t),
	\qquad
	d_L(t):=\norm{\mathbf a_L(t)}_2^2.
\end{equation}
For $N\geq1$, the spectrum of
\cref{eq:Rmn-transform-matrix-rank-two} consists of $M-2$ zeros and the two
eigenvalues
\begin{equation}\label{eq:Rmn-transform-eigenvalues}
	\mu_{\pm}(t)
	=\frac12\bigl(c_L(t)\mathbin{\pm}d_L(t)\bigr),
\end{equation}
with the usual reduction when the rank is less than two.  When $N=0$, the
only eigenvalue is $H_{00,L}^{R}(t)$.  The Cauchy--Schwarz
inequality gives $\abs{c_L(t)}\leq d_L(t)$, and therefore
\begin{equation}\label{eq:Rmn-transform-eigenvalue-signs}
	\mu_+(t)\geq0,
	\qquad
	\mu_-(t)\leq0.
\end{equation}
Thus the transform matrix is real symmetric and of rank at most two, but it
is generally indefinite rather than positive semidefinite.

\subsubsection{The entrywise Weil explicit formula}

Define
\begin{equation}\label{eq:Rmn-pole-vector}
	\mathfrak w_{k,L}
	:=\mathcal W_L\!\left(\nu_{k,L}+\frac{i}{2}\right)
	=-\frac{i\pi^2\sinh(L/2)}
	{L^3(\nu_{k,L}+i/2)
		\left((\nu_{k,L}+i/2)^2-\alpha_L^2\right)}
\end{equation}
and the real pole contribution
\begin{equation}\label{eq:Rmn-pole-contribution}
	\mathcal B_{mn,L}^{R}
	:=2\RePart\!\left(
	\mathfrak w_{n,L}\overline{\mathfrak w_{m,L}}
	\right).
\end{equation}
Since $\mathcal W_L$ is even and real on the real axis,
\begin{equation}\label{eq:Rmn-pole-identity}
	2H_{mn,L}^{R}\!\left(\frac{i}{2}\right)
	=\mathcal B_{mn,L}^{R}.
\end{equation}

\begin{proposition}[Weil formula for the Hermitian raised-cosine convolution matrix]
	\label{prop:Rmn-Weil-explicit-formula}
	For every fixed $T\geq3$, $L=\log T$, and $m,n\in I_N$, one has the
	absolutely convergent identity
	\begin{align}\label{eq:Rmn-Weil-explicit-formula}
		\sum_{\rho}H_{mn,L}^{R}(z_\rho)
		={}&\mathcal B_{mn,L}^{R}
		-\eta_{n+m,L}\log(2\pi)
		\notag\\
		&-2\PF\!\int_0^{2L}
		\cos\!\left(\frac{\nu_{n-m,L}x}{2}\right)
		\mathcal K_{n+m,L}(x)
		\frac{e^{x/2}}{e^x-e^{-x}}\dd x
		\notag\\
		&-2\sum_{2\leq q\leq T^2}
		\frac{\Lambda(q)}{\sqrt q}
		\cos\!\left(
		\frac{\pi(n-m)\log q}{L}
		\right)
		\mathcal K_{n+m,L}(\log q).
	\end{align}
	The endpoint $q=T^2$ contributes zero, even when $T^2$ is a prime power, so
	no endpoint half weight is required.  Under the Riemann hypothesis,
	$z_\rho=\gamma\in\mathbb R$, and the zero weight is the real quantity
	\begin{align}\label{eq:Rmn-zero-weight-RH}
		H_{mn,L}^{R}(\gamma)
		=\frac12\Bigl[{}&
		\mathcal W_L(\gamma+\nu_{n,L})
		\mathcal W_L(\gamma-\nu_{m,L})
		\notag\\
		&+\mathcal W_L(\gamma-\nu_{n,L})
		\mathcal W_L(\gamma+\nu_{m,L})
		\Bigr].
	\end{align}
\end{proposition}

\begin{proof}
	The physical formula \cref{eq:Rmn-physical-form} shows that $R_{mn,L}$ is
	real and even.  Hence \cref{eq:Xi-Weil-even} applies.  Its pole term is
	\cref{eq:Rmn-pole-identity}, and its conductor term is obtained from
	\cref{eq:Rmn-origin-value}.  The support is $[-2L,2L]$, so the prime-power
	sum stops at $q=e^{2L}=T^2$; the endpoint vanishes because the overlap of the
	two raised-cosine windows has zero length there.  Substitution of
	\cref{eq:Rmn-physical-form} into the archimedean and prime-power terms of
	\cref{eq:Xi-Weil-even} proves \cref{eq:Rmn-Weil-explicit-formula}.
	Finally, \cref{eq:Rmn-transform-decay} and the Riemann--von Mangoldt estimate
	show that the zero series converges absolutely, so no starred summation is
	needed.
\end{proof}

\subsubsection{Matrix form}

Let $\mathcal A_{mn,L}^{R}$ and $\mathcal P_{mn,L}^{R}$ denote the
archimedean and prime-power expressions on the last two lines of
\cref{eq:Rmn-Weil-explicit-formula}.  Define
\begin{align}\label{eq:Rmn-explicit-matrices}
	\mathbf Z_{R,L}^{(\infty)}
	&:=\left(
	\sum_\rho H_{mn,L}^{R}(z_\rho)
	\right)_{m,n\in I_N},
	&
	\mathbf B_{R,L}
	&:=\bigl(\mathcal B_{mn,L}^{R}\bigr)_{m,n\in I_N},
	\\
	\mathbf C_{R,L}
	&:=\bigl(\eta_{n+m,L}\bigr)_{m,n\in I_N},
	&
	\mathbf A_{R,L}
	&:=\bigl(\mathcal A_{mn,L}^{R}\bigr)_{m,n\in I_N},
	\\
	\mathbf P_{R,L}
	&:=\bigl(\mathcal P_{mn,L}^{R}\bigr)_{m,n\in I_N}.
\end{align}
Then the entrywise identities assemble into
\begin{equation}\label{eq:Rmn-Weil-matrix-form}
	\mathbf Z_{R,L}^{(\infty)}
	=\mathbf B_{R,L}
	-\log(2\pi)\mathbf C_{R,L}
	-\mathbf A_{R,L}
	-\mathbf P_{R,L}.
\end{equation}
Every matrix in \cref{eq:Rmn-Weil-matrix-form} is real symmetric.  Moreover,
if
$\boldsymbol{\mathfrak w}_L
:=(\mathfrak w_{n,L})_{n\in I_N}$,
then
\begin{equation}\label{eq:Rmn-pole-matrix-Gram-form}
	\mathbf B_{R,L}
	=2\RePart\!\left(
	\boldsymbol{\mathfrak w}_L
	\boldsymbol{\mathfrak w}_L^*
	\right)
	=2\left(
	\RePart\boldsymbol{\mathfrak w}_L
	(\RePart\boldsymbol{\mathfrak w}_L)^{\mathsf T}
	+\ImPart\boldsymbol{\mathfrak w}_L
	(\ImPart\boldsymbol{\mathfrak w}_L)^{\mathsf T}
	\right),
\end{equation}
so the pole matrix is positive semidefinite and has rank at most two.  The
full matrix in \cref{eq:Rmn-Weil-matrix-form}, as well as its separate
archimedean and prime components, need not be positive semidefinite.

\subsection{Comparison of the explicit test families}\label{subsec:test-family-comparison}

\begin{remark}[Choice of a practical test family]\label{rem:choice-test-family}
	The Hermite functions give the simplest smooth orthonormal basis on
	$\mathbb R$, and their autocorrelations are explicit Gaussian--Laguerre
	functions.  The regularized Laguerre functions retain the half-line and
	rational-transform structure suggested by $e^{-x}L_n(x)$.  The function
	$E_a$ is a normalized rational autocorrelation, while its second
	autocorrelation is again an elementary polynomial times $e^{-a\abs{x}}$.
	The hyperbolic-secant family is self-reciprocal in form, and its normalized
	autocorrelation is $ax/\sinh(ax)$.  Even-order $B$-splines make the prime
	side exactly finite.  Sharp logarithmic Fourier modes turn the prime-power
	side into a finite cosine transform, while raised-cosine and spline windows
	restore $W_{\mathrm{loc}}^{2,1}$ regularity and, in the spline case, positive
	type.  The Hermitian raised-cosine product matrices are pointwise
	positive semidefinite of rank at most two; their spectra, transforms, and
	entrywise Weil formulae are all explicit.  The corresponding symmetrized
	convolution matrices are real symmetric but generally indefinite; their
	physical kernels, rank-two transform matrices, and entrywise Weil formulae
	are likewise explicit, with support $[-2\log T,2\log T]$.  The one-sided
	Fourier autocorrelation matrix has the divided-difference structure
	$S_{mn,L}=(b_{m,L}-b_{n,L})/(m-n)$ off the diagonal, a rank-two commutator
	with the index matrix, and is positive semidefinite under the Riemann
	hypothesis.  Hence the most
	useful family depends on whether orthogonality, Gaussian decay, rationality,
	positivity, matrix structure, a sharp logarithmic cutoff, or compact support
	is the principal objective.
\end{remark}

All transforms in these subsections are elementary and explicit.  The
Laguerre formulas follow from \cref{eq:Laguerre-Laplace-identity}, the
Hermite formulas from differentiation of a Gaussian and the Fourier
eigenfunction identity, the $E_a$ formulas from rational transforms, the
hyperbolic-secant formulas from the beta integral, the spline formulas from
the convolution theorem, the product-matrix formulas from the explicit
transform of the fourth-power cosine window, and the convolution-matrix
formulas from the shifted raised-cosine transform.  No limiting or implicit
transform is required.

\end{document}